\documentclass[amssymb,12pt ]{article}
\pdfoutput=1
\usepackage{geometry}
\usepackage{amsfonts}
\usepackage{amssymb}
\usepackage{amsmath}
\usepackage{amsthm}
\usepackage{graphicx}
\usepackage{picture}
\usepackage{curves}
\usepackage{subfigure}
\usepackage{epic}
\usepackage{color}

\newtheorem{thm}{Theorem}[section]
\newtheorem{cor}[thm]{Corollary}
\newtheorem{lem}[thm]{Lemma}
\newtheorem{pro}[thm]{Proposition}

\newtheorem{obs}[thm]{Observation}
\newenvironment{ack}{\noindent{\bf Acknowledgments}}

\newcommand{\vol}{{\rm vol}}
\newcommand{\cs}{{\rm cs}}
\newcommand{\li}{{\rm Li}_2}
\newcommand{\imaginary}{{\rm Im}\,}
\newcommand{\olim}{\underset{N\rightarrow\infty}{\text{\rm o-lim}}}
\newcommand{\modulo}{~~({\rm mod}~\pi^2)}
\newcommand{\modulos}{~~({\rm mod}~4\pi^2)}

\newcommand{\modu}{~~({\rm mod}~2\pi i)}

\begin{document}

\title{Optimistic limits of the colored Jones polynomials
}
\author{\sc Jinseok Cho and Jun Murakami}
\maketitle
\begin{abstract}
We show that the optimistic limits of the colored Jones polynomials of the hyperbolic knots
coincide with the optimistic limits of the Kashaev invariants modulo $4\pi^2$.
\end{abstract}

\section{Introduction}\label{ch1}
  \subsection{Preliminaries}
Kashaev conjectured the following relation in \cite{Kashaev97} :
  $$\vol(L)=2\pi \lim_{N\rightarrow\infty}\frac{\log|\langle L\rangle_N|}{N},$$
  where $L$ is a hyperbolic link, vol($L$) is the hyperbolic volume of $S^3-L$,
  and $\langle L\rangle_N$ is the $N$-th Kashaev invariant.
  After, a generalized conjecture was proposed in \cite{Murakami02} that
 $$i(\vol(L)+i\,\cs(L))\equiv 2\pi i\lim_{N\rightarrow\infty}\frac{\log\langle L\rangle_N}{N}\modulo,$$
 where cs($L$) is the Chern-Simons invariant of $S^3-L$ defined in \cite{Meyerhoff86}.

The calculation of the actual limit of the Kashaev invariant is very hard, and only several cases are known.
On the other hand, while proposing the conjecture, Kashaev used a formal approximation to predict
the actual limit. His formal approximation was formulated as {\it optimistic limit} by H. Murakami in \cite{Murakami00b}.
This method can be summarized in the following way.
First, we fix an expression of $\langle L\rangle_N$, then apply the following {\it formal substitution}
\begin{eqnarray}
(q)_k&\sim&\exp\left\{\frac{N}{2\pi i}\left(-\li(q^k)+\frac{\pi^2}{6}\right)\right\},\label{formalsubsti}\\
~(q^{-1})_k&\sim&\exp\left\{\frac{N}{2\pi i}\left(\li(q^{-k})-\frac{\pi^2}{6}\right)\right\},\nonumber\\
q^{kl}&\sim&\exp\left\{\frac{N}{2\pi i}\left(\log q^k \cdot\log q^l\right)\right\},\nonumber
\end{eqnarray}
to the expression, where $q=\exp(2\pi i/N)$, $\li(z)=-\int_0^z\frac{\log(1-t)}{t}dt$ for $z\in\mathbb{C}$,
$[k]$ is the residue of an integer $k$ modulo $N$, $(q)_k=\prod_{n=1}^{[k]}(1-q^n)$ and $(q)_0=1$.
Then by substituting each $q^k$ with a complex variable $z$, we obtain a potential function
$\exp\left\{\frac{N}{2\pi i}F(\ldots,z,\ldots)\right\}$.
Finally, let
$$F_0(\ldots,z,\ldots):=F-\sum_z \left(z\frac{\partial F}{\partial z}\right)\log z$$
and evaluate $F_0$ for an {\it appropriate solution} of the equations
$\left\{\exp\left(z\frac{\partial F}{\partial z}\right)=1\right\}$.
Then the resulting complex number is called {\it the optimistic limit}. 


For example, the optimistic limit of the Kashaev invariant of the $5_2$ knot was calculated
in \cite{Kashaev97} and \cite{Ohnuki05} as follows.
By the formal substitution,
$$\langle 5_2 \rangle_N=\sum_{k\leq l}\frac{(q)_l^2}{(q^{-1})_k}q^{-k(l+1)}
\sim\exp\left\{\frac{N}{2\pi i}\left(-2\li(q^l)-\li(\frac{1}{q^k})-\log q^l \log q^k+\frac{\pi^2}{2} \right)\right\}.$$
By substituting $z=q^l$ and $u=q^k$, we obtain
$$F(z,u)=-2\li(z)-\li(\frac{1}{u})-\log z \log u +\frac{\pi^2}{2},$$
and
$$F_0(z,u)=F(z,u)-\left(z\frac{\partial F}{\partial z}\right)\log z-\left(u\frac{\partial F}{\partial u}\right)\log u.$$
For the choice of a solution $(z_0,u_0)=(0.3376...-i\,0.5623...,~0.1226...+i\,0.7449...)$ of the equations
$\left\{\exp\left(z\frac{\partial F}{\partial z}\right)=1,~\exp\left(u\frac{\partial F}{\partial u}\right)=1\right\}$,
the optimistic limit becomes
$$
F_0(z_0,u_0)=i\,(2.8281...-i\,3.0241...) 
\equiv i(\vol(5_2)+i\,\cs(5_2))\modulo.$$

As seen above, the optimistic limit depends on the expression and the choice of the solution,
so it is not well-defined. However, Yokota made a very useful way to determine 
the optimistic limit of a hyperbolic knot $K$ in \cite{YokotaPre} and \cite{Yokota10}
by defining a potential function $V(z_1,\ldots,z_g)$ of the knot diagram, 
which also comes from the formal substitution of certain expression of the Kashaev invariant $\langle K\rangle_N$
(the definition of $V(z_1,\ldots,z_g)$ will be given in Section \ref{ch31}).
As above, he also defined
$$V_0(z_1,\ldots,z_g):=V-\sum_{k=1}^g\left(z_k\frac{\partial V}{\partial z_k}\right)\log z_k$$
and $$\mathcal{H}_1:=\left\{\exp(z_k\frac{\partial V}{\partial z_k})=1~\vert ~k=1,\ldots,g\right\}.$$
After proving that $\mathcal{H}_1$ is the hyperbolicity equation of Yokota triangulation, 
he chose the geometric solution $\bold{z}^{(0)}=(z_1^{(0)},\ldots,z_g^{(0)})$ of $\mathcal{H}$
(Yokota triangulation will be discussed in Section \ref{ch21}. The hyperbolicity equation consists of edge relations and
the cusp conditions of a triangulation, and the geometric solution is the one which gives the hyperbolic structure of the triangulation.
Details are in Section \ref{ch4}).
Then he proved
\begin{equation}\label{Yokota}
V_0(\bold{z}^{(0)})\equiv i(\vol(K)+i\,\cs(K))\modulo
\end{equation}
in \cite{Yokota10}. Therefore, we denote
$$2\pi i\,\olim\frac{\log\langle K\rangle_N}{N}:=V_0(\bold{z}^{(0)})$$
and call it {\it the optimistic limit of the Kashaev invariant} $\langle K\rangle_N$.

To obtain (\ref{Yokota}), Yokota assumed several assumptions on the knot diagram 
and the existence of an essential solution of $\mathcal{H}_1$.
The assumptions on the diagram essentially mean to reduce redundant crossings of the diagram before finding the potential function $V$.
Exact statements are \textbf{Assumption 1.1--1.4.} and \textbf{Assumption 2.2.} in \cite{Yokota10}.
We remark that these assumptions are needed so that, after the collapsing process,
Yokota triangulation becomes a topological triangulation of the knot complement $S^3-K$
(see Section \ref{ch31} for details).

As mentioned before, the set of equations $\mathcal{H}_1$ becomes the hyperbolicity equation of Yokota triangulation.
Therefore, each solution ${\bold z}=(z_1,\ldots,z_g)$ of $\mathcal{H}_1$ determines the shape parameters of the ideal tetrahedra of the triangulation 
and the parameters are expressed by the ratios of $z_1,\ldots,z_g$ (details are in Section \ref{ch4}).
We call a solution $\bold z$ of $\mathcal{H}_1$ {\it essential} if no shape parameters are in $\{0,1,\infty\}$, which implies
no edges of the triangulation are homotopically nontrivial.
A well-known fact is that if the hyperbolicity equation has an essential solution, 
then there is a unique geometric solution ${\bold z}^{(0)}$ of $\mathcal{H}_1$ (for details, see Section 2.8 of \cite{Tillmann05}).
Therefore, to guarantee the existence of the geometric solution, Yokota assumed the existence of an essential solution.

On the other hand, it is proved in \cite{Murakami01a} that
$$J_L(N;\exp\frac{2\pi i}{N})=\langle L \rangle_N,$$
where $J_L(N;x)$ is the $N$-th colored Jones polynomial of the link $L$ with a complex variable $x$.
Therefore, it is natural to define the optimistic limit of the colored Jones polynomial
so that it gives the volume and the Chern-Simons invariant. Although it looks trivial,
due to the ambiguity of the optimistic limit, only few results are known.
It was numerically confirmed for few examples in \cite{Murakami02},
actually proved only for the volume part of two bridge links in \cite{Ohnuki05} and
for the Chern-Simons part of twist knots in \cite{Cho10a}.
In a nutshell, the purpose of this paper is to propose a general method to define the optimistic limit of the colored Jones polynomial
of a hyperbolic knot $K$ and to prove the following relation :
\begin{equation}\label{goal}
2\pi i\,\olim\frac{\log\langle K\rangle_N}{N}\equiv 2\pi i\,\olim\frac{\log J_K(N;\exp\frac{2\pi i}{N})}{N}\modulos.
\end{equation}

\subsection{Main result}

For a hyperbolic knot $K$, we define a potential function $W(w_1,\ldots, w_m)$ of a knot diagram in Section \ref{ch32},
which also comes from the formal substitution of certain expression of the colored Jones polynomial $J_L(N;\exp\frac{2\pi i}{N})$.
We define
$$W_0(w_1,\ldots,w_m):=W-\sum_{l=1}^m\left(w_l\frac{\partial W}{\partial w_l}\right)\log w_l,$$
and
$$\mathcal{H}_2:=\left\{\exp\left(w_l\frac{\partial W}{\partial w_l}\right)=1~\vert~ l=1,\ldots,m\right\}.$$
Also, we discuss Thurston triangulation of the knot complement $S^3-K$ in Section \ref{ch22}, which was introduced in \cite{Thurston99}.

\begin{pro}\label{prop11} 
For a hyperbolic knot $K$ with a fixed diagram, 
we assume the diagram satisfies \textbf{Assumption 1.1.--1.4.} and \textbf{Assumption 2.2.} in \cite{Yokota10}.
For the potential function $W(w_1,\ldots, w_m)$ of the diagram, $\mathcal{H}_2$ becomes the hyperbolicity equation of Thurston triangulation.
\end{pro}

Proof of Proposition \ref{prop11} will be given in Section \ref{ch4}.

Each solution ${\bold w}=(w_1,\ldots,w_m)$ of $\mathcal{H}_2$ determines the shape parameters of the ideal tetrahedra of Thurston triangulation 
and the parameters are expressed by the ratios of $w_1,\ldots,w_m$ (details are in Section \ref{ch4}).
We call a solution $\bold w$ of $\mathcal{H}_2$ {\it essential} if no shape parameters are in $\{0,1,\infty\}$.
Comparing Yokota triangulation and Thurston triangulation, we obtain the following Lemma.

\begin{lem}\label{lem12} 
  For a hyperbolic knot $K$ with a fixed diagram and the assumptions of Proposition \ref{prop11},
 an essential solution $\bold{z}=(z_1,\ldots,z_g)$ of $\mathcal{H}_1$ determines 
  the unique solution $\bold{w}=(w_1,\ldots,w_m)$ of $\mathcal{H}_2$, and vice versa. 
  Furthermore, if the determined solution $\bold{w}$ is essential,
  then $\bold{w}$ also induces $\bold{z}$, and vice versa.
\end{lem}

Proof of Lemma \ref{lem12} will be given in Section \ref{ch5}. 
Although there is a possibility that an essential solution $\bold{z}$ of $\mathcal{H}_1$
determines a non-essential solution $\bold{w}$ of $\mathcal{H}_2$, we expect this not to happen in almost all cases
(this is discussed in Appendix \ref{app2}).
In this paper, we only consider the case when the determined solution $\bold{w}$ is essential.


\begin{thm}\label{thm}
For a hyperbolic knot $K$ with a fixed diagram, assume the assumptions of Proposition \ref{prop11}.
Let $V(z_1,\ldots,z_g)$ and $W(w_1,\ldots,w_m)$ be the potential functions of the knot diagram.
Also assume the hyperbolicity equation $\mathcal{H}_1$ has an essential solution $\bold{z}=(z_1,\ldots,z_g)$
and let $\bold{z}^{(0)}=(z_1^{(0)},\ldots,z_g^{(0)})$ be the geometric solution of $\mathcal{H}_1$.
From Lemma \ref{lem12}, let $\bold{w}=(w_1,\ldots,w_g)$ 
and $\bold{w}^{(0)}=(w_1^{(0)},\ldots,w_m^{(0)})$ be the corresponding solutions of $\mathcal{H}_2$ 
determined by $\bold{z}$ and by $\bold{z}^{(0)}$, respectively.
We also assume $\bold{w}$ and $\bold{w}^{(0)}$ are essential. 
Then
  \begin{enumerate}
    \item 
      $V_0(\bold{z})\equiv W_0(\bold{w})\modulos,$
    \item $\bold{w}^{(0)}$ is the geometric solution of $\mathcal{H}_2$ and
        \begin{equation*}
           W_0(\bold{w}^{(0)})\equiv i(\vol(K)+i\,\cs(K))\modulo.
         \end{equation*}
  \end{enumerate}
\end{thm}

The proof is in Section \ref{ch5}. 
We denote 
$$2\pi i\,\olim\frac{\log J_K(N;\exp\frac{2\pi i}{N})}{N}:=W_0(\bold{w}^{(0)})$$
and call it {\it the optimistic limit of the colored Jones polynomial} $J_K(N;\exp\frac{2\pi i}{N})$.
With this definition, Theorem \ref{thm} implies (\ref{goal}). Also, we obtain the colored Jones polynomial version
of Corollary 1.4 of \cite{Cho10b} as follows.

\begin{cor}\label{cor}
For a hyperbolic knot $K$ with a fixed diagram, assume the assumptions of Proposition \ref{prop11}.
Let $\bold{w}$ be an essential solution of $\mathcal{H}_2$, $\bold{w}^{(0)}$ be the geometric solution of $\mathcal{H}_2$,
and $\rho_{\bold{w}}:\pi_1(S^3-K)\rightarrow{\rm PSL}(2,\mathbb{C})$
be the parabolic representation induced by $\bold{w}$. Also, assume the corresponding solutions
$\bold{z}$ and $\bold{z}^{(0)}$ of $\mathcal{H}_1$, determined by $\bold{w}$ and by $\bold{w}^{(0)}$, respectively, 
from Lemma \ref{lem12} are essential.
Then
$$W_0(\bold{w})\equiv i\,({\rm vol}(\rho_{\bold{w}})+i\,{\rm cs}(\rho_{\bold{w}}))\text{ \rm  (mod } \pi^2),$$
where ${\rm vol}(\rho_{\bold{w}})+i\,{\rm cs}(\rho_{\bold{w}})$ is the complex volume of $\rho_{\bold{w}}$
defined in \cite{Zickert09}. Furthermore, the following inequality holds: 
\begin{equation}\label{representation}
  {\rm Im}\, W_0(\bold{w})\leq{\rm Im}\, W_0(\bold{w}^{(0)})={\rm vol}(K).
\end{equation}
The equality in (\ref{representation}) holds if and only if $\bold{w}=\bold{w}^{(0)}$.
\end{cor}

\begin{proof} 
It is a well-known fact that the hyperbolic volume is the maximal volume
of all possible $\rm PSL(2,\mathbb{C})$ representations and that the maximum happens if and only if
the representation is discrete and faithful (for the proof and details, see \cite{Francaviglia04}).

From the proof of Lemma \ref{lem12}, if $\bold{w}$ and $\bold{z}$ are essential, 
then the shapes of each (collapsed) octahedra in Figure \ref{pic2} and Figure \ref{pic9} of Yokota and Thurston triangulations coincide.
Therefore, these triangulations form the same geometric shape, 
and the parabolic representation $\rho_{\bold{w}}$ coincides with $\rho_{\bold{z}}$ up to conjugate,
where $\rho_{\bold{w}}$ and $\rho_{\bold{z}}$ are the parabolic representations induced by $\bold{w}$ and by $\bold{z}$, respectively.
This also implies that $\bold{z^{(0)}}$ is the geometric solution of $\mathcal{H}_1$.

Yokota proved
$$V_0(\bold{z^{(0)}})\equiv i\,({\rm vol}(K)+i\,{\rm cs}(K))\text{ \rm  (mod } \pi^2)$$
in \cite{Yokota10} using Zickert's formula of \cite{Zickert09}, but the formula also holds for
any parabolic representation $\rho_{\bold{z}}$ induced by $\bold{z}$.
Therefore, Yokota's proof also implies
$$V_0(\bold{z})\equiv i\,({\rm vol}(\rho_{\bold{z}})+i\,{\rm cs}(\rho_{\bold{z}}))\text{ \rm  (mod }\pi^2).$$

Among the essential solutions $\bold{z}$ of $\mathcal{H}_1$,
only the geometric solution $\bold{z^{(0)}}$ induces the discrete faithful representation.
Therefore, applying Theorem \ref{thm}, we complete the proof.

\end{proof}

This paper consists of the following contents.
In Section \ref{ch2}, we describe Yokota triangulation and Thurston triangulation, 
which correspond to the Kashaev invariant and the colored Jones polynomial, respectively.
We show that these two triangulations are related by finite steps of 3-2 moves and 4-5 moves on some crossings.
In Section \ref{ch3}, the potential functions $V$ and $W$ are defined.
In Section \ref{ch4}, we explain the geometries of $V$ and $W$, and prove Proposition \ref{prop11}. 
In Section \ref{ch5}, we introduce several dilogarithm identities 
and complete the proofs of Lemma \ref{lem12} and Theorem \ref{thm} using these identities.
In Appendix \ref{app1}, we show the potential function $W$ defined in Section \ref{ch3} can be obtained 
by the formal substitution of the colored Jones polynomial. Finally, in Appendix \ref{app2},
we investigate the necessary and sufficient condition for 
an essential solution of $\mathcal{H}_1$ (respectively, $\mathcal{H}_2$) to induce 
the inessential solution of $\mathcal{H}_2$ (respectively, $\mathcal{H}_1$).

\section{Two ideal triangulations of the knot complement}\label{ch2}

In this section, we explain two ideal triangulations of the knot complement.
One is Yokota triangulation corresponding to the Kashaev invariant in \cite{Yokota10}
and the other is Thurston triangulation corresponding to the colored Jones polynomial in \cite{Thurston99}.
A good reference of this section is \cite{Murakami01b}, which contains wonderful pictures.

\subsection{Yokota triangulation}\label{ch21}
\begin{figure}[h]
\centering
  \subfigure[Knot]
  {\includegraphics[scale=0.4]{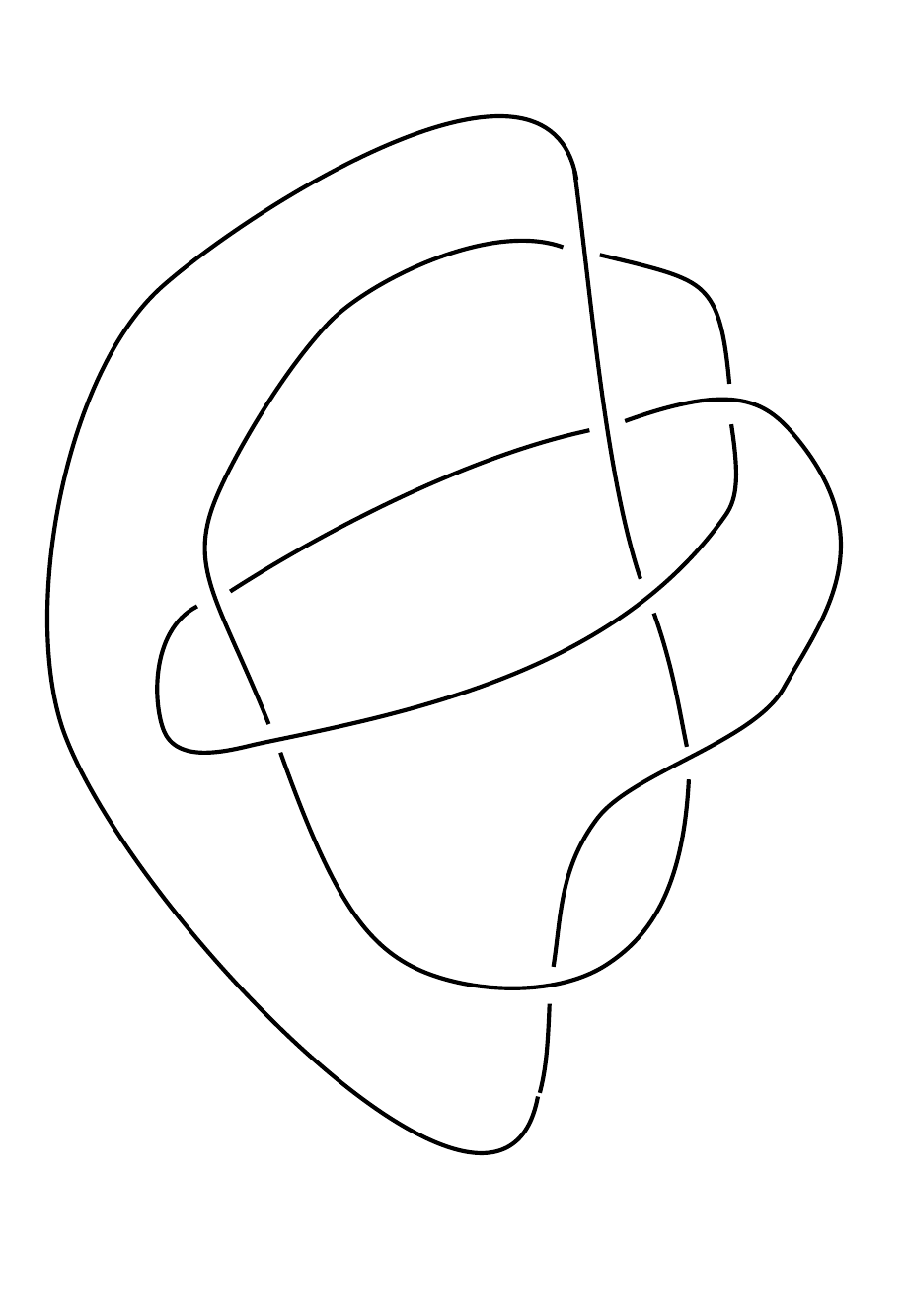}}
  \subfigure[(1,1)-tangle]
  {\includegraphics[scale=0.5]{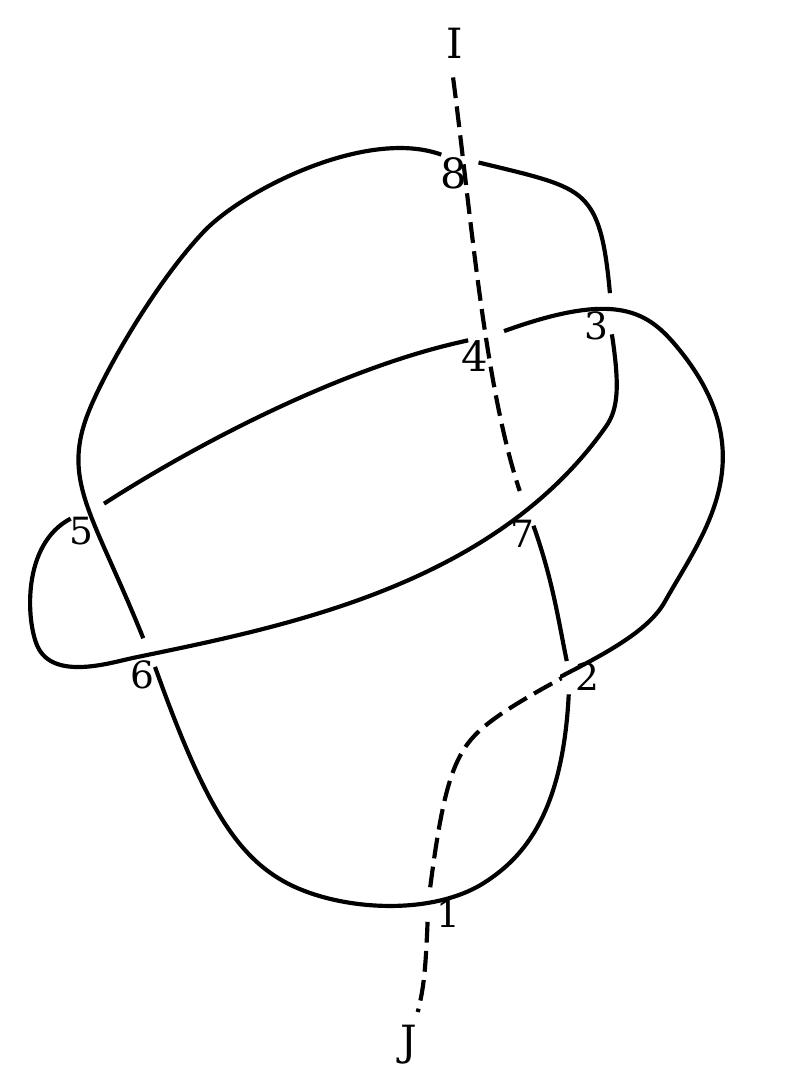}}
  \caption{Example}\label{pic1}
\end{figure}
Consider a hyperbolic knot $K$ and its diagram $D$ (see Figure \ref{pic1}(a)).
We define {\it sides} of $D$ as arcs connecting two adjacent crossing points.
For example, Figure \ref{pic1}(a) has 16 sides.

Now split a side of $D$ open so as to make a (1,1)-tangle diagram
and label crossings with integers (see Figure \ref{pic1}(b)).
Yokota assumed several conditions on this (1,1)-tangle diagram
(for the exact statement, see \textbf{Assumption 1.1.--1.4.} and \textbf{Assumption 2.2.} in \cite{Yokota10}).
The assumptions roughly mean that we remove all the crossing points that can be reduced trivially.
Also, let the two open sides be $I$ and $J$ and consider the orientation from $J$ to $I$. 
Assume $I$ and $J$ are in an over-bridge and in an under-bridge, respectively 
(Over-bridge is a union of sides, following the orientation of the knot diagram,
from one over-crossing point to the next under-crossing point.
Under-bridge is the one from one under-crossing point to the next over-crossing point.
The boundary endpoints of $I$ and $J$ are considered over-crossing point and under-crossing point, respectively.
For example, in Figure \ref{pic1}(b), if we follow the diagram from the below to the top, 
the first under-bridge containing $J$ ends at the crossing 2,
and the first over-bridge starts at the crossing 2 and ends at the crossing 4. 
In total, it has 5 over-bridges and 5 under-bridges. Note that if we change the orientation, the numbers of
over-bridges and under-bridges change).

Now extend $I$ and $J$ so that, when following the orientation of the knot diagram, 
non-boundary endpoints of $I$ and $J$ become the last under-crossing point
and the first over-crossing point, respectively, as in Figure \ref{pic1}(b).
Then we assume the two non-boundary endpoints
of $I$ and $J$ do not coincide, because, if they coincide, then we cut other side open and make a different tangle diagram.
Yokota proved in \cite{Yokota10} that we can always make two non-boundary endpoints different by cutting certain side open because,
if not, then the diagram should be that of a link or the trefoil knot 
(for details, see \textbf{Assumption 1.3.} and the discussion that follows in \cite{Yokota10}).

\begin{figure}[h]
\centering
  \subfigure[Octahedron on the crossing $n$]
  {\includegraphics[scale=0.6]{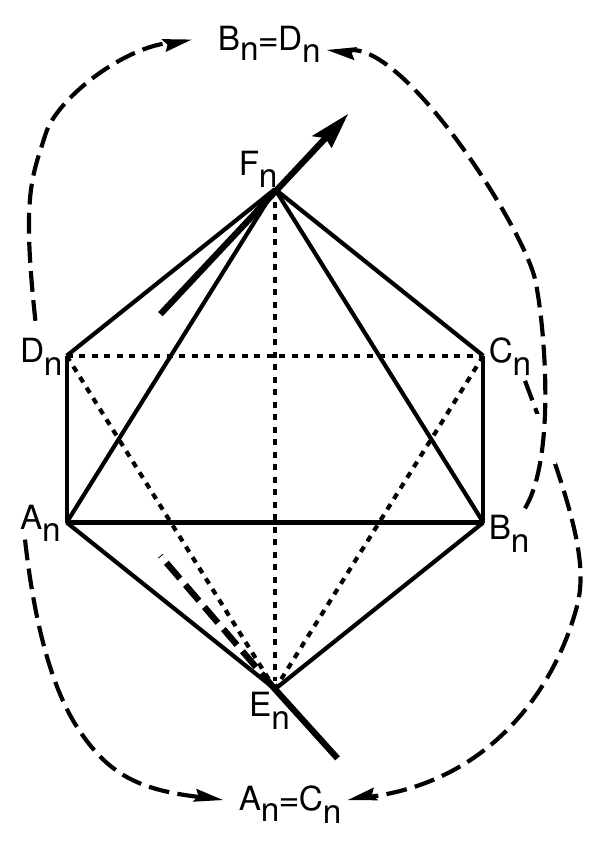}}
  \subfigure[Octahedra on crossings]
  {\includegraphics[scale=0.5]{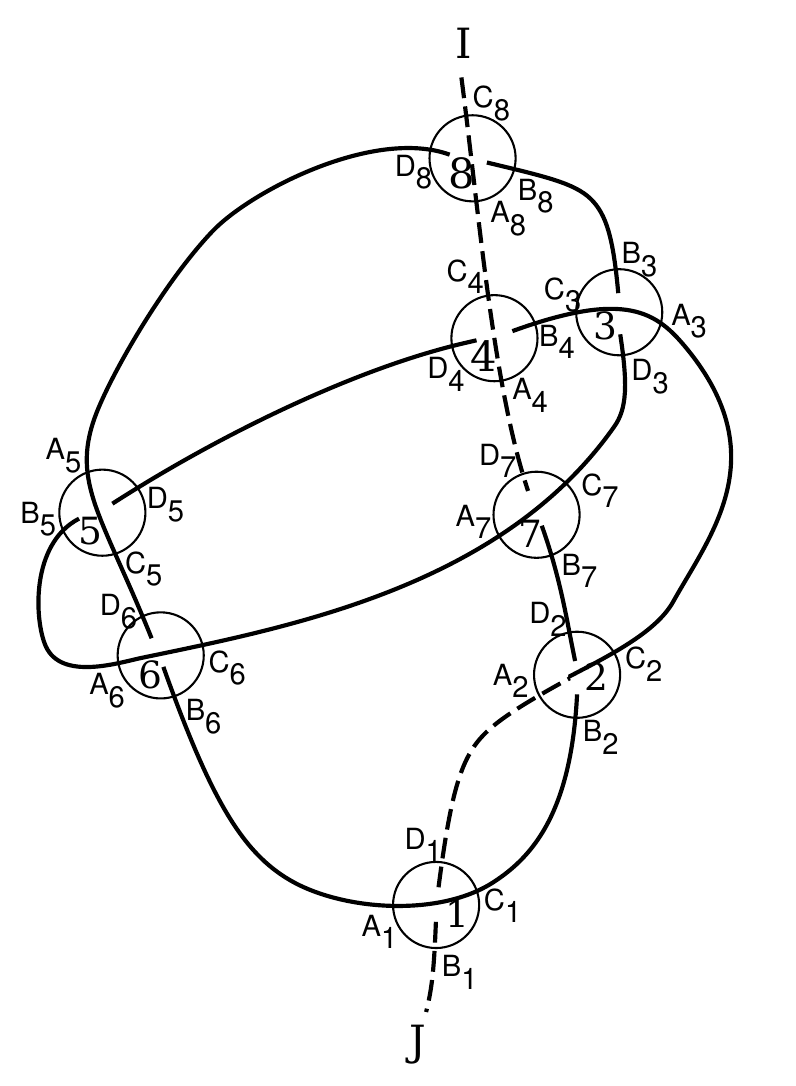}}
  \caption{Example(continued)}\label{pic2}
\end{figure}

To obtain an ideal triangulation of the knot complement, we 
place an ideal octahedron
${\rm A}_n{\rm B}_n{\rm C}_n{\rm D}_n{\rm E}_n{\rm F}_n$ on each crossing $n$ as in Figure \ref{pic2}(a).
We call the edges ${\rm A}_n{\rm B}_n$, ${\rm B}_n{\rm C}_n$, ${\rm C}_n{\rm D}_n$ and ${\rm D}_n{\rm A}_n$
of the octahedron {\it horizontal edges}.
Figure \ref{pic2}(b) shows the positions of ${\rm A}_n$, ${\rm B}_n$, ${\rm C}_n$, ${\rm D}_n$
and the horizontal edges. We twist the octahedron by identifying the edges
${\rm A}_n{\rm E}_n$ to ${\rm C}_n{\rm E}_n$ and ${\rm B}_n{\rm F}_n$ to ${\rm D}_n{\rm F}_n$
as in Figure \ref{pic2}(a)
(the actual shape of the resulting diagram appears in \cite{Murakami01b}).
Then we glue the faces of the twisted octahedron following the knot diagram.
For example, in Figure \ref{pic2}(b),
we glue $\triangle{\rm A}_1{\rm E}_1{\rm D}_1\cup\triangle{\rm C}_1{\rm E}_1{\rm D}_1$ to
$\triangle{\rm A}_2{\rm F}_2{\rm D}_2\cup\triangle{\rm A}_2{\rm F}_2{\rm B}_2$,
$\triangle{\rm C}_2{\rm F}_2{\rm D}_2\cup\triangle{\rm C}_2{\rm F}_2{\rm B}_2$ to
$\triangle{\rm A}_3{\rm F}_3{\rm D}_3\cup\triangle{\rm A}_3{\rm F}_3{\rm B}_3$,
$\triangle{\rm C}_3{\rm F}_3{\rm D}_3\cup\triangle{\rm C}_3{\rm F}_3{\rm B}_3$ to
$\triangle{\rm A}_4{\rm E}_4{\rm B}_4\cup\triangle{\rm C}_4{\rm E}_4{\rm B}_4$,
$\triangle{\rm A}_4{\rm E}_4{\rm D}_4\cup\triangle{\rm C}_4{\rm E}_4{\rm D}_4$ to
$\triangle{\rm C}_5{\rm E}_5{\rm D}_5\cup\triangle{\rm A}_5{\rm E}_5{\rm D}_5$, and so on.
Finally, we glue $\triangle{\rm D}_8{\rm F}_8{\rm C}_8\cup\triangle{\rm B}_8{\rm F}_8{\rm C}_8$ to
$\triangle{\rm A}_1{\rm E}_1{\rm B}_1\cup\triangle{\rm C}_1{\rm E}_1{\rm B}_1$.
Note that, by gluing likewise, all ${\rm A}_n$ and ${\rm C}_n$ are identified to one point, all ${\rm B}_n$ and ${\rm D}_n$
are identified to another point, and all ${\rm E}_n$ and ${\rm F}_n$ are identified to yet another point.
Let these points be $-\infty$, $\infty$ and $\ell$, respectively.
Then the regular neighborhoods of $-\infty$ and $\infty$ become 3-balls, whereas that of $\ell$
becomes the tubular neighborhood of the knot $K$.

We split each octahedron ${\rm A}_n{\rm B}_n{\rm C}_n{\rm D}_n{\rm E}_n{\rm F}_n$
into four tetrahedra, ${\rm A}_n{\rm B}_n{\rm E}_n{\rm F}_n$,
${\rm B}_n{\rm C}_n{\rm E}_n{\rm F}_n$, ${\rm C}_n{\rm D}_n{\rm E}_n{\rm F}_n$ and
${\rm D}_n{\rm A}_n{\rm E}_n{\rm F}_n$.
Then we collapse faces that lie on the split sides. For example, in Figure \ref{pic2}(b), we collapse the faces
$\triangle{\rm A}_1{\rm E}_1{\rm B}_1\cup\triangle{\rm C}_1{\rm E}_1{\rm B}_1$ and
$\triangle{\rm D}_8{\rm F}_8{\rm C}_8\cup\triangle{\rm B}_8{\rm F}_8{\rm C}_8$ to different points.
Note that this face collapsing makes some edges on these faces into points.
Actually, the non-horizontal edges ${\rm A}_2{\rm F}_2$, ${\rm B}_4{\rm F}_4$, ${\rm D}_4{\rm F}_4$,
${\rm D}_7{\rm E}_7$, and the horizontal edges ${\rm B}_2{\rm C}_2$, ${\rm A}_3{\rm B}_3$,
${\rm A}_5{\rm B}_5$, ${\rm A}_6{\rm B}_6$ in Figure \ref{pic2}(b) are collapsed to points
because of the face collapsing. This makes the tetrahedra ${\rm A}_1{\rm B}_1{\rm E}_1{\rm F}_1$,
${\rm B}_1{\rm C}_1{\rm E}_1{\rm F}_1$, ${\rm C}_1{\rm D}_1{\rm E}_1{\rm F}_1$,
${\rm D}_1{\rm A}_1{\rm E}_1{\rm F}_1$, ${\rm A}_2{\rm B}_2{\rm E}_2{\rm F}_2$,
${\rm B}_2{\rm C}_2{\rm E}_2{\rm F}_2$, ${\rm D}_2{\rm A}_2{\rm E}_2{\rm F}_2$,
${\rm A}_3{\rm B}_3{\rm E}_3{\rm F}_3$, ${\rm A}_4{\rm B}_4{\rm E}_4{\rm F}_4$, 
${\rm B}_4{\rm C}_4{\rm E}_4{\rm F}_4$, ${\rm C}_4{\rm D}_4{\rm E}_4{\rm F}_4$,
${\rm D}_4{\rm A}_4{\rm E}_4{\rm F}_4$, ${\rm A}_5{\rm B}_5{\rm E}_5{\rm F}_5$,
${\rm A}_6{\rm B}_6{\rm E}_6{\rm F}_6$, ${\rm C}_7{\rm D}_7{\rm E}_7{\rm F}_7$,
${\rm D}_7{\rm A}_7{\rm E}_7{\rm F}_7$, ${\rm A}_8{\rm B}_8{\rm E}_8{\rm F}_8$,
${\rm B}_8{\rm C}_8{\rm E}_8{\rm F}_8$, ${\rm C}_8{\rm D}_8{\rm E}_8{\rm F}_8$ and
${\rm D}_8{\rm A}_8{\rm E}_8{\rm F}_8$
be collapsed to points or edges.

\begin{figure}[h]
\centering
  \includegraphics[scale=0.5]{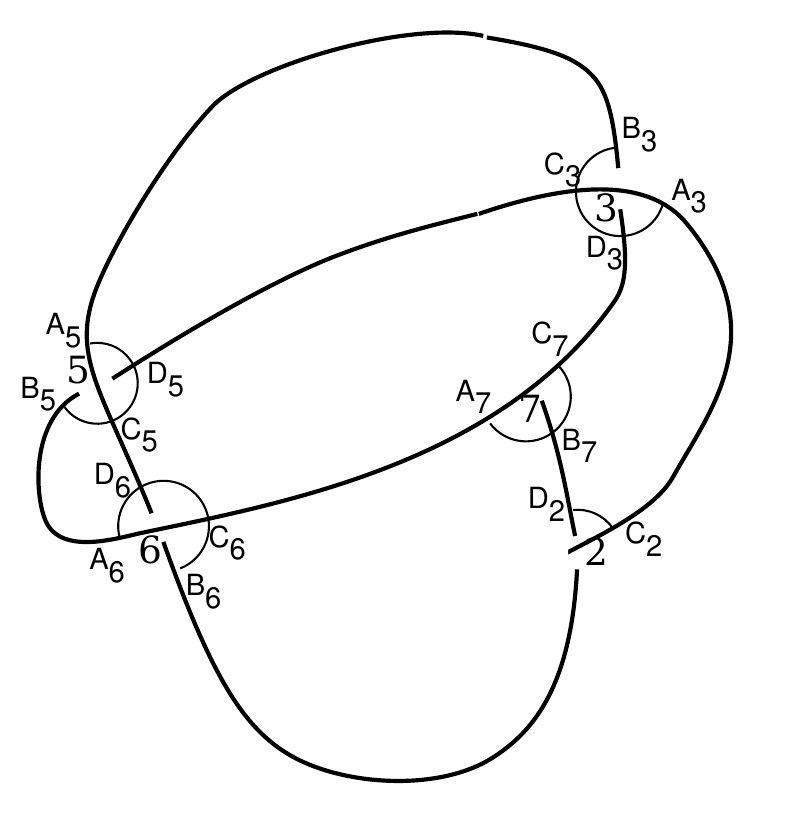}
  \caption{G with survived tetrahedra}\label{pic3}
\end{figure}

The surviving tetrahedra after the collapsing can be depicted as follows.
First, remove $I$ and $J$ on the tangle diagram and denote the result as $G$
(see Figure \ref{pic3}). Note that, by removing $I\cup J$,
some vertices are removed, two vertices become trivalent and some sides are glued together.
In Figure \ref{pic3}, vertices 1, 4, 8 are removed, 2, 7 become trivalent and $G$ has 9 sides
(we consider the sides at the trivalent vertices are not glued together).
Now we remove the horizontal edges on the removed vertices, the horizontal edges that are adjacent to $I\cup J$
and the horizontal edges in the unbounded region
(see Figure \ref{pic3} for the result).
The surviving horizontal edges mean the surviving ideal tetrahedra after the collapsing.
In the example, 12 tetrahedra survive.

The collapsing identifies the points $\infty$, $-\infty$, and $\ell$ to each other
and connects the regular neighborhoods of them.
Collapsing certain edges of a tetrahedron may change the topological type of $\ell$,
but Yokota excluded such cases by \textbf{Assumption 1.1.--1.3.} on the shape of the knot diagram.
(\textbf{Assumption 1.1.--1.2.} roughly means the diagram has no redundant crossings and 
\textbf{Assumption 1.3.} means the two non-boundary endpoints of $I$ and $J$ do not coincide.)
Therefore, the result of the collapsing makes the neighborhood of $\infty=-\infty=\ell$
to be the tubular neighborhood of the knot,
and we obtain the ideal triangulation of the knot complement
(see \cite{Yokota10} for a complete discussion).

\subsection{Thurston triangulation}\label{ch22}

Thurston triangulation, introduced in \cite{Thurston99}, uses the same octahedra and the same collapsing process,
so it also induces an ideal triangulation of the knot complement.
However, it uses a different subdivision of each octahedra. In Figure \ref{pic2}(a), Yokota triangulation subdivides
each octahedron into four tetrahedra. However, Thurston triangulation subdivides it into five tetrahedra,
${\rm A}_n{\rm B}_n{\rm D}_n{\rm F}_n$, ${\rm B}_n{\rm C}_n{\rm D}_n{\rm F}_n$,
${\rm A}_n{\rm B}_n{\rm C}_n{\rm D}_n$, ${\rm A}_n{\rm B}_n{\rm C}_n{\rm E}_n$ and
${\rm A}_n{\rm C}_n{\rm D}_n{\rm E}_n$
(see the right-hand side of Figure \ref{pic4}(a) for the shape of the subdivision).
In this subdivision, if we apply the collapsing process, then some pair of tetrahedra shares the same four vertices
(see the first case of (Case 2) in the proof of Observation \ref{obs} for an example).
For the convenience of discussion, when this happens, we remove these two tetrahedra and
call the result Thurston triangulation.

To see the relation between these two triangulations,
we define {\it 4-5 move} of an octahedron and {\it 3-2 move} of a hexahedron
as in Figure \ref{pic4}.

\begin{figure}[h]
\centering
  \subfigure[4-5 move]
  {\includegraphics[scale=0.4]{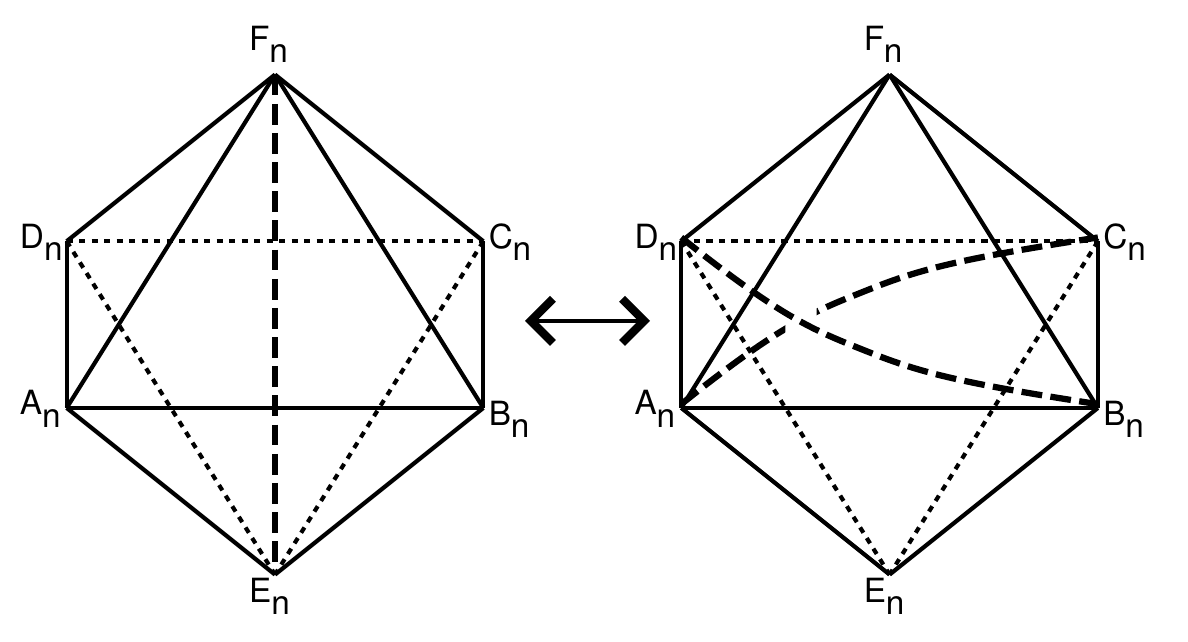}}
  \subfigure[3-2 move]
  {\includegraphics[scale=0.4]{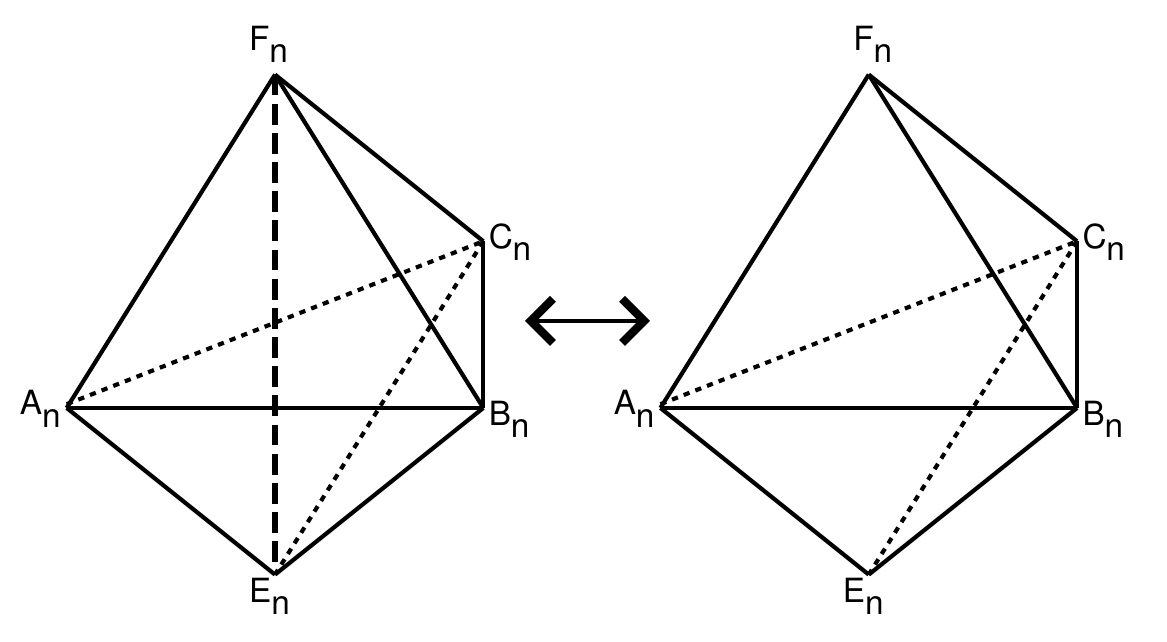}}
  \caption{4-5 and 3-2 moves}\label{pic4}
\end{figure}

Before the collapsing process, two triangulations are related by only 4-5 moves on each crossings.
However, the following observation shows they are actually related by 4-5 moves and also by 3-2 moves
on some crossings after the collapsing.

\begin{obs}\label{obs}
For a hyperbolic knot $K$ with a fixed diagram, if the diagram satisfies 
\textbf{Assumption 1.1.--1.4.} and \textbf{Assumption 2.2.} in \cite{Yokota10},
then Yokota triangulation and Thurston triangulation are related by 3-2 moves and 4-5 moves on some crossings.\end{obs}

\begin{proof}
First, for a non-trivalent vertex $n$ of $G$, we show only one horizontal edge in Figure \ref{pic2}(a)
can be collapsed. If any of two horizontal edges are collapsed,
then the (1,1)-tangle diagram should be Figure \ref{pic5}(a) or Figure \ref{pic5}(b)
for some tangle diagrams $K_1$ or $K_2$ because the collapsed edges should lie in the unbounded regions.
However, Figure \ref{pic5}(a) is excluded because, if we close up the open side, then $K=K_1\# K_2$ and $K$ cannot be prime.
We can also exclude Figure \ref{pic5}(b) because it violates \textbf{Assumption 1.1.} in \cite{Yokota10}. Actually, in the later case,
we can reduce the number of crossings as in Figure \ref{pic5}(b).

\begin{figure}[h]
\centering
  \subfigure[]
  {\includegraphics[scale=0.5]{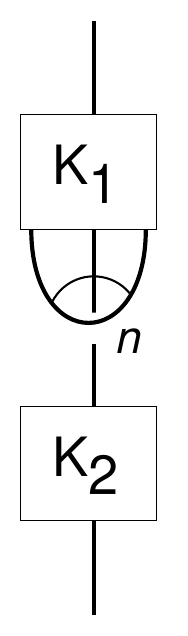}}\hspace{1cm}
  \subfigure[]
  {\includegraphics[scale=0.55]{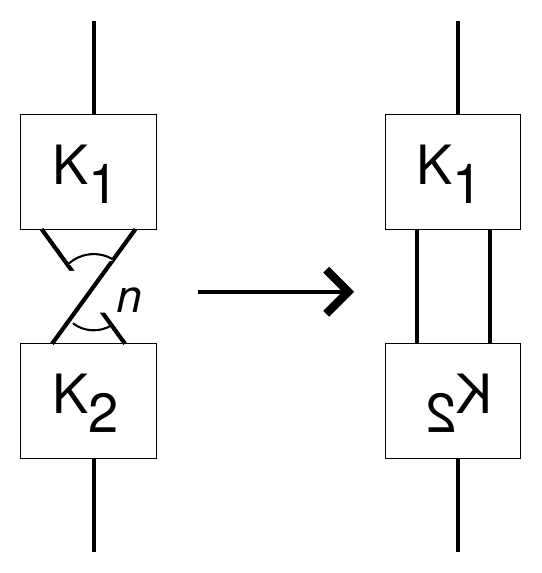}}
  \caption{When two horizontal edges are collapsed}\label{pic5}
\end{figure}

Because of this and Yokota's Assumptions,
all possible cases of collapsing edges in Figure \ref{pic2}(a) are as follows :

(Case 1)
  if $n$ is a non-trivalent vertex of $G$, then none or one of the horizontal edges is collapsed.

(Case 2)
  if $n$ is a trivalent vertex of $G$, then
  \begin{enumerate}
    \item ${\rm D}_n{\rm E}_n$ is collapsed and none or one of ${\rm A}_n{\rm B}_n$, ${\rm B}_n{\rm C}_n$ is collapsed,
    \item ${\rm B}_n{\rm E}_n$ is collapsed and none or one of ${\rm C}_n{\rm D}_n$, ${\rm D}_n{\rm A}_n$ is collapsed,
    \item ${\rm A}_n{\rm F}_n$ is collapsed and none or one of ${\rm B}_n{\rm C}_n$, ${\rm C}_n{\rm D}_n$ is collapsed.
  \end{enumerate}

(Case 1) is trivial, so we consider the first case of (Case 2).

If ${\rm D}_n{\rm E}_n$ and ${\rm A}_n{\rm B}_n$ are collapsed, then the survived tetrahedron is
${\rm B}_n{\rm C}_n{\rm E}_n{\rm F}_n$ in Yokota triangulation, and
${\rm B}_n{\rm C}_n{\rm D}_n{\rm F}_n$ in Thurston triangulation. They coincide because
${\rm D}_n={\rm E}_n$ by the collapsing of ${\rm D}_n{\rm E}_n$.

If ${\rm D}_n{\rm E}_n$ is collpased and no others are, then the survived tetrahedra are
${\rm A}_n{\rm B}_n{\rm E}_n{\rm F}_n$ and ${\rm B}_n{\rm C}_n{\rm E}_n{\rm F}_n$ in Yokota triangulation,
and ${\rm A}_n{\rm B}_n{\rm D}_n{\rm F}_n$, ${\rm B}_n{\rm C}_n{\rm D}_n{\rm F}_n$,
${\rm A}_n{\rm B}_n{\rm C}_n{\rm D}_n$ and ${\rm A}_n{\rm B}_n{\rm C}_n{\rm E}_n$ in Thurston triangulation.
However, in Thurston triangulation,
two tetrahedra ${\rm A}_n{\rm B}_n{\rm C}_n{\rm D}_n$ and ${\rm A}_n{\rm B}_n{\rm C}_n{\rm E}_n$
cancel each other because they share the same vertices ${\rm A}_n$, ${\rm B}_n$ ,${\rm C}_n$ and ${\rm D}_n={\rm E}_n$. 
The others coincide with the tetrahedra in Yokota triangulation
because ${\rm D}_n={\rm E}_n$.

Other cases of (Case 2) are the same as the first case, so the proof is completed.

\end{proof}

\section{Potential functions}\label{ch3}
\subsection{The case of Kashaev invariant}\label{ch31}

In the case of Kashaev invariant, Yokota's potential function $V(z_1,\ldots,z_g)$ is defined by the following way.

For the graph $G$, we define {\it contributing sides} as sides of $G$ which are not on the unbounded regions.
For example, there are 5 contributing sides and 4 non-contributing sides in Figure \ref{pic6}.
We assign complex variables $z_1,\ldots,z_g$ to contributing sides and real number 1 to non-contributing sides.
Then we label each ideal tetrahedra with $IT_1,IT_2,\ldots,IT_s$ and assign $t_l$ ($l=1,\ldots,s$)
to the horizontal edge of $IT_l$ as the shape parameter.
We define $t_l$ as the counterclockwise ratio of the complex variables $z_1,\ldots,z_g$.

\begin{figure}[h]
\centering
  \includegraphics[scale=0.5]{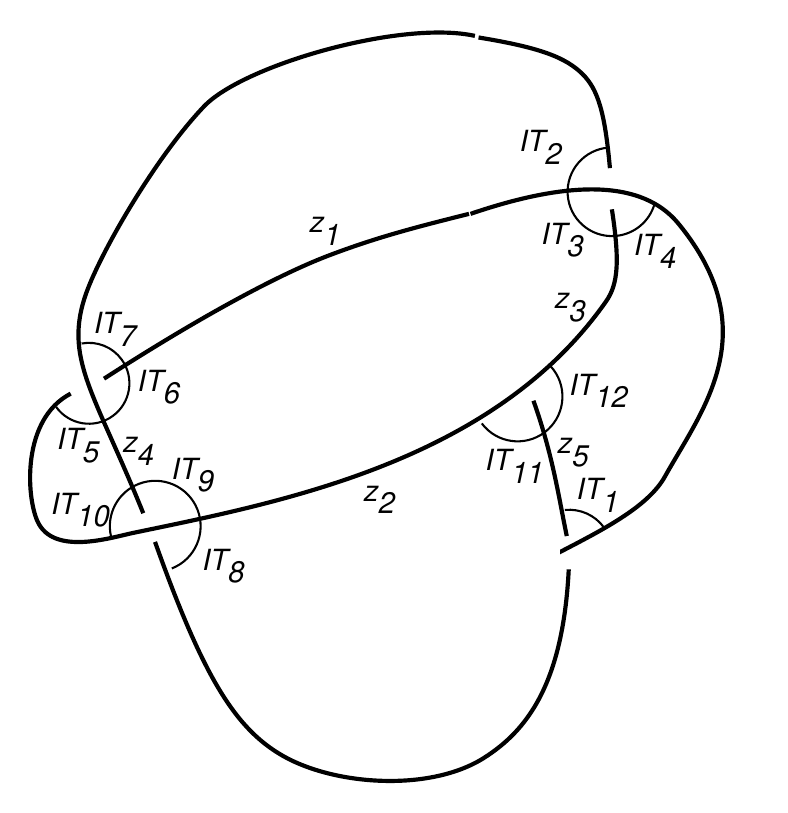}
  \caption{G with contributing sides}\label{pic6}
\end{figure}

For example, in Figure \ref{pic6},
\begin{eqnarray*}
t_1=\frac{z_5}{1},~t_{2}=\frac{z_1}{1},~ t_{3}=\frac{z_3}{z_1}, ~t_{4}=\frac{1}{z_3},~t_5=\frac{z_4}{1},~ t_6=\frac{z_1}{z_4},\\
t_7=\frac{1}{z_1},~t_8=\frac{z_2}{1},~ t_9=\frac{z_4}{z_2}, ~t_{10}=\frac{1}{z_4},~ t_{11}=\frac{z_5}{z_2},~ t_{12}=\frac{z_3}{z_5}.
\end{eqnarray*}

For each tetrahedron $IT_l$, we assign dilogarithm
function as in Figure \ref{pic7}. Then we define $V(z_1,\ldots,z_g)$ by
the summation of all these dilogarithm functions.
We also define the sign $\sigma_l$ of $IT_l$ by
$$\sigma_l=\left\{\begin{array}{ll}
  1&\text{ if }~IT_l\text{ lies as in Figure \ref{pic7}(a)},\\
  -1&\text{ if }~IT_l\text{ lies as in Figure \ref{pic7}(b)}.
  \end{array}\right.$$
Then $V(z_1,\ldots,z_g)$ is expressed by
\begin{equation*}
V(z_1,\ldots,z_g)=\sum_{l=1}^g \sigma_l\left({\rm Li}_2(t_l^{\sigma_l})-\frac{\pi^2}{6}\right).
\end{equation*}
For example, in Figure \ref{pic6},
$$\sigma_1=\sigma_3=\sigma_6=\sigma_9=\sigma_{11}=1,
~\sigma_2=\sigma_4=\sigma_5=\sigma_7=\sigma_8=\sigma_{10}=\sigma_{12}=-1,$$
and
\begin{eqnarray*}
V(z_1,\ldots,z_5)=
{\rm Li}_2(z_5)-{\rm Li}_2(\frac{1}{z_1})+{\rm Li}_2(\frac{z_3}{z_1})
-{\rm Li}_2(z_3)-{\rm Li}_2(\frac{1}{z_4})+{\rm Li}_2(\frac{z_1}{z_4})\\
-{\rm Li}_2(z_1)-{\rm Li}_2(\frac{1}{z_2})+{\rm Li}_2(\frac{z_4}{z_2})
-{\rm Li}_2(z_4)+{\rm Li}_2(\frac{z_5}{z_2})-{\rm Li}_2(\frac{z_5}{z_3})+\frac{\pi^2}{3}.
\end{eqnarray*}

\begin{figure}
\centering
  \subfigure[Positive corner]
  { {\setlength{\unitlength}{0.4cm}
  \begin{picture}(15,6)\thicklines
    \put(6,5){\line(-1,-1){4}}
    \put(6,1){\line(-1,1){1.8}}
    \put(3.8,3.2){\line(-1,1){1.8}}
    \put(4,3){\arc(-0.6,-0.6){90}}
    \put(3.5,1){$IT_l$}
    \put(7,3){$\displaystyle\longrightarrow ~~{\rm Li}_2(t_l)-\frac{\pi^2}{6}$}
  \end{picture}}}
  \subfigure[Negative corner]
  { {\setlength{\unitlength}{0.4cm}
  \begin{picture}(15,6)\thicklines
    \put(2,5){\line(1,-1){4}}
    \put(2,1){\line(1,1){1.8}}
    \put(4.2,3.2){\line(1,1){1.8}}
    \put(4,3){\arc(-0.6,-0.6){90}}
    \put(3.5,1){$IT_l$}
    \put(7,3){$\displaystyle\longrightarrow ~~\frac{\pi^2}{6}-{\rm Li}_2(\frac{1}{t_l})$}
  \end{picture}}}
  \caption{Assignning dilogarithm functions to each tetrahedra}\label{pic7}
\end{figure}
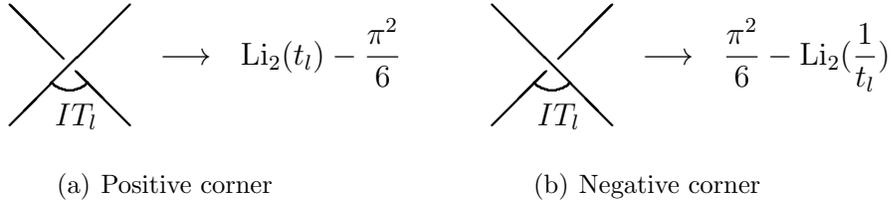
It is shown in \cite{YokotaPre} that $V(z_1,\ldots,z_g)$ can be obtained by the formal
substitution of the Kashaev invariant.\footnote{
We remark that the Kashaev invariant of a knot $K$ defined in \cite{YokotaPre} 
is the one of the mirror image $\overline{K}$ defined in \cite{Murakami01a}. 
This paper follows the definition of \cite{YokotaPre}.}

\subsection{The case of colored Jones polynomial}\label{ch32}

For each region of $G$, we choose one bounded region and assign 1 to it.
Then we assign variables $w_1,\ldots,w_m$ to
the remaining bounded regions, and 0 to the unbounded region (see Figure \ref{reg}).

\begin{figure}[h]
\centering
  \includegraphics[scale=0.5]{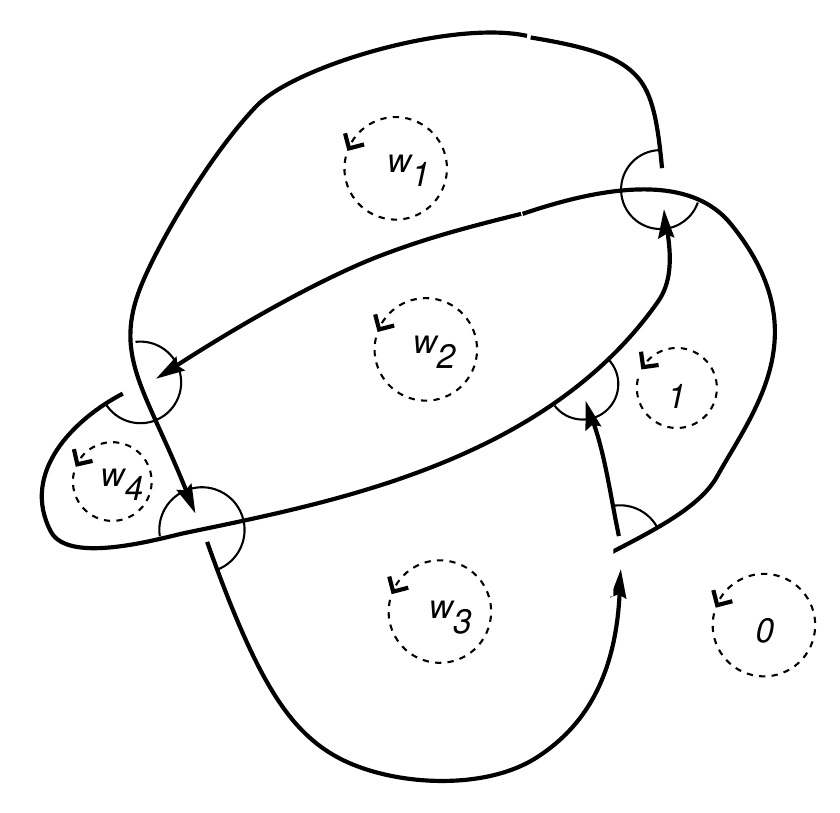}
  \caption{Assigning variables to each region}\label{reg}
\end{figure}

For each vertex of $G$, we assign the following functions according to the type of the vertex
and the horizontal edges.
For positive crossings :

  
  {\setlength{\unitlength}{0.4cm}
  \begin{picture}(18,4.5)\thicklines
    \put(5,4){\vector(-1,-1){4}}
    \put(1,4){\line(1,-1){1.8}}
    \put(3.2,1.8){\vector(1,-1){1.8}}
    \put(2.5,0){$w_j$}
    \put(4.5,1.5){$w_k$}
    \put(2.5,3.5){$w_l$}
    \put(0,1.5){$w_m$}
    \put(3,2){\arc(-0.6,0.6){270}}
    \put(6.5,3){: $P_1(w_j,w_k,w_l,w_m)=-\li(\frac{w_l}{w_m})-\li(\frac{w_l}{w_k})+\li(\frac{w_jw_l}{w_kw_m})$}
    \put(16,1){$+\li(\frac{w_m}{w_j})+\li(\frac{w_k}{w_j})-\frac{\pi^2}{6}+\log\frac{w_m}{w_j}\log\frac{w_k}{w_j},$}
  \end{picture}}

  {\setlength{\unitlength}{0.4cm}
  \begin{picture}(18,4.5)\thicklines
    \put(5,4){\vector(-1,-1){4}}
    \put(1,4){\line(1,-1){1.8}}
    \put(3.2,1.8){\vector(1,-1){1.8}}
    \put(2.5,0){$w_j$}
    \put(4.5,1.5){$w_k$}
    \put(2.5,3.5){$w_l$}
    \put(0,1.5){$w_m$}
    \put(3,2){\arc(-0.6,-0.6){270}}
    \put(6.5,3){: $P_2(w_j,w_k,w_l,w_m)=\li(\frac{w_m}{w_l})-\li(\frac{w_l}{w_k})-\li(\frac{w_kw_m}{w_jw_l})$}
    \put(16,1){$+\li(\frac{w_m}{w_j})-\li(\frac{w_j}{w_k})+\frac{\pi^2}{6}-\log\frac{w_k}{w_l}\log\frac{w_k}{w_j},$}
  \end{picture}}

  {\setlength{\unitlength}{0.4cm}
  \begin{picture}(18,4.5)\thicklines
    \put(5,4){\vector(-1,-1){4}}
    \put(1,4){\line(1,-1){1.8}}
    \put(3.2,1.8){\vector(1,-1){1.8}}
    \put(2.5,0){$w_j$}
    \put(4.5,1.5){$w_k$}
    \put(2.5,3.5){$w_l$}
    \put(0,1.5){$w_m$}
    \put(3,2){\arc(0.6,-0.6){270}}
    \put(6.5,3){: $P_3(w_j,w_k,w_l,w_m)=\li(\frac{w_m}{w_l})+\li(\frac{w_k}{w_l})+\li(\frac{w_jw_l}{w_kw_m})$}
    \put(16,1){$-\li(\frac{w_j}{w_m})-\li(\frac{w_j}{w_k})-\frac{\pi^2}{6}+\log\frac{w_m}{w_l}\log\frac{w_k}{w_l},$}
  \end{picture}}

  {\setlength{\unitlength}{0.4cm}
  \begin{picture}(18,4.5)\thicklines
    \put(5,4){\vector(-1,-1){4}}
    \put(1,4){\line(1,-1){1.8}}
    \put(3.2,1.8){\vector(1,-1){1.8}}
    \put(2.5,0){$w_j$}
    \put(4.5,1.5){$w_k$}
    \put(2.5,3.5){$w_l$}
    \put(0,1.5){$w_m$}
    \put(3,2){\arc(0.6,0.6){270}}
    \put(6.5,3){: $P_4(w_j,w_k,w_l,w_m)=-\li(\frac{w_l}{w_m})+\li(\frac{w_k}{w_l})-\li(\frac{w_kw_m}{w_jw_l})$}
    \put(16,1){$-\li(\frac{w_j}{w_m})+\li(\frac{w_k}{w_j})+\frac{\pi^2}{6}-\log\frac{w_m}{w_l}\log\frac{w_m}{w_j}.$}
  \end{picture}}

\vspace{0.5cm}

For negative crossings :

  {\setlength{\unitlength}{0.4cm}
  \begin{picture}(18,4.5)\thicklines
    \put(1,4){\vector(1,-1){4}}
   \put(5,4){\line(-1,-1){1.8}}
   \put(2.8,1.8){\vector(-1,-1){1.8}}
    \put(2.5,0){$w_j$}
    \put(4.5,1.5){$w_k$}
    \put(2.5,3.5){$w_l$}
    \put(0,1.5){$w_m$}
    \put(3,2){\arc(-0.6,0.6){270}}
    \put(6.5,3){: $N_1(w_j,w_k,w_l,w_m)=\li(\frac{w_l}{w_m})+\li(\frac{w_l}{w_k})-\li(\frac{w_jw_l}{w_kw_m})$}
    \put(16,1){$-\li(\frac{w_m}{w_j})-\li(\frac{w_k}{w_j})+\frac{\pi^2}{6}-\log\frac{w_j}{w_m}\log\frac{w_j}{w_k},$}

  \end{picture}}

  {\setlength{\unitlength}{0.4cm}
  \begin{picture}(18,4.5)\thicklines
    \put(1,4){\vector(1,-1){4}}
   \put(5,4){\line(-1,-1){1.8}}
   \put(2.8,1.8){\vector(-1,-1){1.8}}
    \put(2.5,0){$w_j$}
    \put(4.5,1.5){$w_k$}
    \put(2.5,3.5){$w_l$}
    \put(0,1.5){$w_m$}
    \put(3,2){\arc(-0.6,-0.6){270}}
    \put(6.5,3){: $N_2(w_j,w_k,w_l,w_m)=-\li(\frac{w_m}{w_l})+\li(\frac{w_l}{w_k})+\li(\frac{w_k w_m}{w_j w_l})$}
    \put(16,1){$-\li(\frac{w_m}{w_j})+\li(\frac{w_j}{w_k})-\frac{\pi^2}{6}+\log\frac{w_l}{w_k}\log\frac{w_j}{w_k},$}
  \end{picture}}

  {\setlength{\unitlength}{0.4cm}
  \begin{picture}(18,4.5)\thicklines
    \put(1,4){\vector(1,-1){4}}
   \put(5,4){\line(-1,-1){1.8}}
   \put(2.8,1.8){\vector(-1,-1){1.8}}
    \put(2.5,0){$w_j$}
    \put(4.5,1.5){$w_k$}
    \put(2.5,3.5){$w_l$}
    \put(0,1.5){$w_m$}
    \put(3,2){\arc(0.6,-0.6){270}}
    \put(6.5,3){: $N_3(w_j,w_k,w_l,w_m)=-\li(\frac{w_m}{w_l})-\li(\frac{w_k}{w_l})-\li(\frac{w_j w_l}{w_k w_m})$}
    \put(16,1){$+\li(\frac{w_j}{w_m})+\li(\frac{w_j}{w_k})+\frac{\pi^2}{6}-\log\frac{w_l}{w_m}\log\frac{w_l}{w_k},$}
  \end{picture}}

  {\setlength{\unitlength}{0.4cm}
  \begin{picture}(18,4.5)\thicklines
    \put(1,4){\vector(1,-1){4}}
   \put(5,4){\line(-1,-1){1.8}}
   \put(2.8,1.8){\vector(-1,-1){1.8}}
    \put(2.5,0){$w_j$}
    \put(4.5,1.5){$w_k$}
    \put(2.5,3.5){$w_l$}
    \put(0,1.5){$w_m$}
    \put(3,2){\arc(0.6,0.6){270}}
    \put(6.5,3){: $N_4(w_j,w_k,w_l,w_m)=\li(\frac{w_l}{w_m})-\li(\frac{w_k}{w_l})+\li(\frac{w_k w_m}{w_j w_l})$}
    \put(16,1){$+\li(\frac{w_j}{w_m})-\li(\frac{w_k}{w_j})-\frac{\pi^2}{6}+\log\frac{w_l}{w_m}\log\frac{w_j}{w_m}.$}
  \end{picture}}
 
\vspace{0.5cm}

If no horizontal edge is collapsed at the positive nor the negative crossing, 
we assign any of $P_1,\ldots,P_4$ or $N_1,\ldots,N_4$ to the crossing, respectively.
In Lemma \ref{lem1}, we will show this choice does not have any effect on the optimistic limit of the colored Jones polynomial.

For the endpoints of $I$ and $J$, we use the same formula disregarding whether certain horizontal edge is collapsed or not.
For the endpoint of $I$ :

  {\setlength{\unitlength}{0.4cm}
  \begin{picture}(18,4.5)\thicklines
    \put(5,4){\vector(-1,-1){4}}
    \dashline{0.5}(3.2,1.8)(5,0)
    \put(4.8,0.2){\vector(1,-1){0.2}}
    \put(1,4){\line(1,-1){1.8}}
    \put(2.5,0){$w_j$}
    \put(2.5,3.5){$w_l$}
    \put(0.5,2){$w_m$}
    \put(7,2.5){: $P_1(w_j,w_j,w_l,w_m)=P_2(w_j,w_j,w_l,w_m)=\li(\frac{w_m}{w_j})-\li(\frac{w_l}{w_j}),$}
  \end{picture}}

  {\setlength{\unitlength}{0.4cm}
  \begin{picture}(18,4.5)\thicklines
    \put(1,4){\vector(1,-1){4}}
    \dashline{0.5}(1,0)(2.8,1.8)
    \put(1.2,0.2){\vector(-1,-1){0.2}}
    \put(5,4){\line(-1,-1){1.8}}
    \put(2.5,0){$w_j$}
    \put(4,2){$w_k$}
    \put(2.5,3.5){$w_l$}
    \put(7,2.5){: $N_1(w_j,w_k,w_l,w_j)=N_4(w_j,w_k,w_l,w_j)=-\li(\frac{w_k}{w_j})+\li(\frac{w_l}{w_j}).$}
  \end{picture}}

\vspace{0.5cm}

For the endpoint of $J$ :

  {\setlength{\unitlength}{0.4cm}
  \begin{picture}(18,4.5)\thicklines
   \put(3.2,1.8){\vector(1,-1){2}}
   \put(3,2){\vector(-1,-1){2}}
   \put(1,4){\line(1,-1){1.8}}
   \dashline{0.5}(3,2)(5,4)
    \put(2.5,0){$w_j$}
    \put(4,2){$w_k$}
    \put(0.5,2){$w_m$}
    \put(7,2.5){: $P_2(w_j,w_k,w_k,w_m)=P_3(w_j,w_k,w_k,w_m)=\li(\frac{w_m}{w_k})-\li(\frac{w_j}{w_k}),$}
  \end{picture}}

  {\setlength{\unitlength}{0.4cm}
  \begin{picture}(18,4.5)\thicklines
   \put(2.8,1.8){\vector(-1,-1){1.8}}
   \put(3.2,2.2){\line(1,1){1.8}}
   \put(3,2){\vector(1,-1){2}}
   \dashline{0.5}(1,4)(3,2)
    \put(2.5,0){$w_j$}
    \put(4,2){$w_k$}
    \put(2.5,3.5){$w_l$}
    \put(7,2.5){: $N_3(w_j,w_k,w_l,w_l)=N_4(w_j,w_k,w_l,w_l)=-\li(\frac{w_k}{w_l})+\li(\frac{w_j}{w_l}).$}
  \end{picture}}
\vspace{5mm}

In Appendix, we show that the assigned functions above are, in fact, obtained by the formal substitution of
certain forms of the R-matrix of the colored Jones polynomial.

Now we define the potential function $W(w_1,\ldots,w_{m})$ of the knot diagram by the summation of
all functions assigned to the vertices of $G$.
For example, the potential function $W(w_1,\ldots,w_4)$ of Figure \ref{reg} is
\begin{eqnarray}
&&W(w_1,\ldots,w_4)=-\li(\frac{1}{w_3})+
\left\{\li(\frac{1}{w_2})+\li(\frac{w_1}{w_2})-\frac{\pi^2}{6}+\log\frac{1}{w_2}\log\frac{w_1}{w_2}\right\}\label{W}\\
&&~~~+\left\{\li(\frac{w_1}{w_2})+\li(\frac{w_4}{w_2})-\frac{\pi^2}{6}+\log\frac{w_1}{w_2}\log\frac{w_4}{w_2}\right\}\nonumber\\
&&~~~+\left\{\li(\frac{w_4}{w_2})+\li(\frac{w_3}{w_2})-\frac{\pi^2}{6}+\log\frac{w_4}{w_2}\log\frac{w_3}{w_2}\right\}\nonumber
+\left\{\li(\frac{1}{w_2})-\li(\frac{w_3}{w_2})\right\}.
\end{eqnarray}

We end this section with the invariance of the optimistic limit
under the choice of the four different forms of the potential functions of a crossing.

\begin{lem}\label{lem1}
For the functions $P_1,\ldots,P_4,N_1,\ldots,N_4$ defined above, let
$$P_{f0}:=P_f-\sum_{a=j,k,l,m}\left(w_a\frac{\partial P_f}{\partial w_a}\right)\log w_a,~
N_{f0}:=N_f-\sum_{a=j,k,l,m}\left(w_a\frac{\partial N_f}{\partial w_a}\right)\log w_a.$$
Then $$P_{10}\equiv P_{20}\equiv P_{30}\equiv P_{40},
~ N_{10}\equiv N_{20}\equiv N_{30}\equiv N_{40}\modulos,$$
and for $a=j,k,l,m$,
\begin{eqnarray*}
\exp\left(w_a\frac{\partial P_1}{\partial w_a}\right)=\exp\left(w_a\frac{\partial P_2}{\partial w_a}\right)
=\exp\left(w_a\frac{\partial P_3}{\partial w_a}\right)=\exp\left(w_a\frac{\partial P_4}{\partial w_a}\right),\\
\exp\left(w_a\frac{\partial N_1}{\partial w_a}\right)=\exp\left(w_a\frac{\partial N_2}{\partial w_a}\right)
=\exp\left(w_a\frac{\partial N_3}{\partial w_a}\right)=\exp\left(w_a\frac{\partial N_4}{\partial w_a}\right).
\end{eqnarray*}
\end{lem}

\begin{proof} For a given complex-valued function $F(w_j,w_k,w_l,w_m)$, let
\begin{equation}\label{eq4}
\widehat{F}(w_j,w_k,w_l,w_m):=F+\sum_{a=j,k,l,m}2n_a\pi i \log w_a+4n\pi^2
\end{equation}
for some integer constants $n_j,n_k,n_l,n_m,n$.
Then by a direct calculation,
$$\widehat{F}_0\equiv F_0\modulos$$
and
$$\exp\left(w_a\frac{\partial F}{\partial w_a}\right)=\exp\left(w_a\frac{\partial \widehat{F}}{\partial w_a}\right).$$
These show $F$ and $\widehat{F}$ define the same optimistic limit, so we define an equivalence relation $\approx$ by
$F\approx \widehat{F}$ for $F$ and $\widehat{F}$ satisfying (\ref{eq4}).

For
\begin{eqnarray*}
P_1=-\li(\frac{w_l}{w_m})-\li(\frac{w_l}{w_k})+\li(\frac{w_jw_l}{w_kw_m})+\li(\frac{w_m}{w_j})+\li(\frac{w_k}{w_j})-\frac{\pi^2}{6}+\log\frac{w_m}{w_j}\log\frac{w_k}{w_j},\\
P_2=\li(\frac{w_m}{w_l})-\li(\frac{w_l}{w_k})-\li(\frac{w_k w_m}{w_j w_l})+\li(\frac{w_m}{w_j})-\li(\frac{w_j}{w_k})+\frac{\pi^2}{6}-\log\frac{w_k}{w_l}\log\frac{w_k}{w_j},
\end{eqnarray*}
using the well-known identity $\li(z)+\li(\frac{1}{z})\approx-\frac{\pi^2}{6}-\frac{1}{2}\log^2(-z)$
for $z\in\mathbb{C}$ in \cite{Lewin1}, we obtain
\begin{eqnarray*}
\lefteqn{P_1-P_2=-\li(\frac{w_l}{w_m})-\li(\frac{w_m}{w_l})+\li(\frac{w_jw_l}{w_kw_m})+\li(\frac{w_k w_m}{w_j w_l})}\\
&&~~+\li(\frac{w_k}{w_j})+\li(\frac{w_j}{w_k})-\frac{\pi^2}{3}
+\left(\log\frac{w_m}{w_j}+\log\frac{w_k}{w_l}\right)\log\frac{w_k}{w_j}\\
&&\approx-\frac{\pi^2}{2}+\frac{1}{2}\log^2(-\frac{w_l}{w_m})-\frac{1}{2}\log^2(-\frac{w_kw_m}{w_jw_l})
-\frac{1}{2}\log^2(-\frac{w_k}{w_j})+\left(\log\frac{w_m}{w_j}+\log\frac{w_k}{w_l}\right)\log\frac{w_k}{w_j}.
\end{eqnarray*}
For any integer $n$, some integers $n_1,\ldots,n_4$ and indices $a,b\in\{i,j,k,l\}$, we have
\begin{eqnarray*}
\lefteqn{2n\pi i\log\frac{w_a}{w_b}=2n\pi i\left(\log w_a-\log w_b+2n_1\pi i\right)\approx 0,}\\
\lefteqn{\frac{1}{2}\log^2(-\frac{w_k}{w_j})=\frac{1}{2}\left\{\log\frac{w_k}{w_j}+(2n_2-1)\pi i\right\}^2}\\
&&=\frac{1}{2}\log^2\frac{w_k}{w_j}+(2n_2-1)\pi i\log\frac{w_k}{w_j}-2n_2(n_2-1)\pi^2-\frac{\pi^2}{2}\\
&&\approx\frac{1}{2}\log^2\frac{w_k}{w_j}-\pi i\log\frac{w_k}{w_j}-\frac{\pi^2}{2}
\end{eqnarray*}
and
\begin{eqnarray*}
\lefteqn{\frac{1}{2}\left\{\log\frac{w_k}{w_j}-\log(-\frac{w_kw_m}{w_jw_l})\right\}^2
=\frac{1}{2}\left\{\log(-\frac{w_l}{w_m})+2n_3\pi i\right\}^2}\\
&&=\frac{1}{2}\log^2(-\frac{w_l}{w_m})+2n_3\pi i\left\{\log\frac{w_l}{w_m}+(2n_4+1)\pi i\right\}-2n_3^2\pi^2\\
&&\approx\frac{1}{2}\log^2(-\frac{w_l}{w_m})-2n_3(n_3+1)\pi^2\approx\frac{1}{2}\log^2(-\frac{w_l}{w_m}).
\end{eqnarray*}
Therefore, we obtain
\begin{eqnarray*}
\lefteqn{P_1-P_2}\\
&&\approx\frac{1}{2}\log^2(-\frac{w_l}{w_m})-\frac{1}{2}\log^2(-\frac{w_kw_m}{w_jw_l})
-\frac{1}{2}\log^2\frac{w_k}{w_j}+\pi i\log\frac{w_k}{w_j}
+\log\frac{w_kw_m}{w_jw_l}\log\frac{w_k}{w_j}\\
&&\approx\frac{1}{2}\log^2(-\frac{w_l}{w_m})-\frac{1}{2}\log^2(-\frac{w_kw_m}{w_jw_l})
-\frac{1}{2}\log^2\frac{w_k}{w_j}+\log(-\frac{w_kw_m}{w_jw_l})\log\frac{w_k}{w_j}\\
&&=\frac{1}{2}\log^2(-\frac{w_l}{w_m})
-\frac{1}{2}\left\{\log\frac{w_k}{w_j}-\log(-\frac{w_kw_m}{w_jw_l})\right\}^2
\approx\frac{1}{2}\log^2(-\frac{w_l}{w_m})-\frac{1}{2}\log^2(-\frac{w_l}{w_m})=0.
\end{eqnarray*}
Other equalities $P_2\approx P_3\approx P_4$ and $N_1\approx N_2\approx N_3\approx N_4$ can be obtained
by the same method or by the symmetry of the equations.
\end{proof}

\section{Geometric structures of the triangulations}\label{ch4}

For Yokota triangulation and Thurston triangulation, we assign complex variables to each tetrahedra
and solve certain equations. Then one of the solutions gives the complete hyperbolic structure of the knot complement.
We describe these procedures in this section.

First, consider the positive and negative crossings in Figure \ref{label},
where $z_a, z_b, z_c, z_d$ are the variables assigned to the sides of $G$ 
and $w_j,w_k,w_l,w_m$ are the variables assigned to the regions of $G$.
Note that $z_a, z_b, z_c, z_d$ and $w_j,w_k,w_l,w_m$ are used for defining the potential functions $V(z_1,\ldots,z_g)$
and $W(w_1,\ldots,w_m)$, respectively.

\begin{figure}[h]
\centering
  \setlength{\unitlength}{0.4cm}
  \begin{picture}(8,5.5)\thicklines
    \put(6,5){\vector(-1,-1){4}}
    \put(2,5){\line(1,-1){1.8}}
    \put(4.2,2.8){\vector(1,-1){1.8}}
    \put(3.5,1){$w_j$}
    \put(6,2.5){$w_k$}
    \put(3.5,4.5){$w_l$}
    \put(1,2.5){$w_m$}
    \put(1,5.3){$z_d$}
    \put(6,5.3){$z_c$}
    \put(1,0.2){$z_a$}
    \put(6,0.2){$z_b$}
  \end{picture}
  \begin{picture}(8,5.5)\thicklines
    \put(2,5){\vector(1,-1){4}}
   \put(6,5){\line(-1,-1){1.8}}
   \put(3.8,2.8){\vector(-1,-1){1.8}}
    \put(3.5,1){$w_j$}
    \put(6,2.5){$w_k$}
    \put(3.5,4.5){$w_l$}
    \put(1,2.5){$w_m$}
    \put(1,5.3){$z_d$}
    \put(6,5.3){$z_c$}
    \put(1,0.2){$z_a$}
    \put(6,0.2){$z_b$}
  \end{picture}\caption{Assignment of variables}\label{label}
\end{figure}
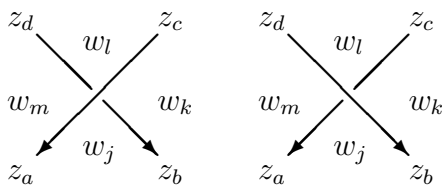

Then consider Figure \ref{pic9}.
We assign $\frac{z_b}{z_a}, \frac{z_c}{z_b}, \frac{z_d}{z_c}$, $\frac{z_a}{z_d}$ to the horizontal edges
${\rm C}_n{\rm D}_n$, ${\rm D}_n{\rm A}_n$, ${\rm A}_n{\rm B}_n$, ${\rm B}_n{\rm C}_n$, respectively.
This assignment determines the shape parameters of the tetrahedra of Yokota triangulation.
Also, for the positive crossing, we assign $\left(\frac{w_j}{w_m}\right)^{-1}$, $\frac{w_k}{w_j}$,
$\frac{w_k}{w_l}$, $\left(\frac{w_l}{w_m}\right)^{-1}$ to
${\rm C}_n{\rm F}_n$, ${\rm D}_n{\rm E}_n$,
${\rm A}_n{\rm F}_n$, ${\rm B}_n{\rm E}_n$, respectively, and
assign $\left(\frac{w_kw_m}{w_jw_l}\right)^{-1}$ to ${\rm B}_n{\rm D}_n$ and ${\rm A}_n{\rm C}_n$
for the parameter of the tetrahedron ${\rm A}_n{\rm B}_n{\rm C}_n{\rm D}_n$.
For the negative crossing, we assign
$\frac{w_j}{w_m}$, $\left(\frac{w_k}{w_j}\right)^{-1}$,
$\left(\frac{w_k}{w_l}\right)^{-1}$, $\frac{w_l}{w_m}$ to
${\rm B}_n{\rm E}_n$, ${\rm C}_n{\rm F}_n$, ${\rm D}_n{\rm E}_n$, ${\rm A}_n{\rm F}_n$,
respectively, and
assign $\left(\frac{w_jw_l}{w_kw_m}\right)^{-1}$ to ${\rm B}_n{\rm D}_n$ and ${\rm A}_n{\rm C}_n$
for the parameter of the tetrahedron ${\rm A}_n{\rm B}_n{\rm C}_n{\rm D}_n$. These assignments 
determine the shape parameters of the tetrahedra of Thurston triangulation.

\begin{figure}[h]
\centering
  \subfigure[Positive crossing]
  {\includegraphics[scale=0.6]{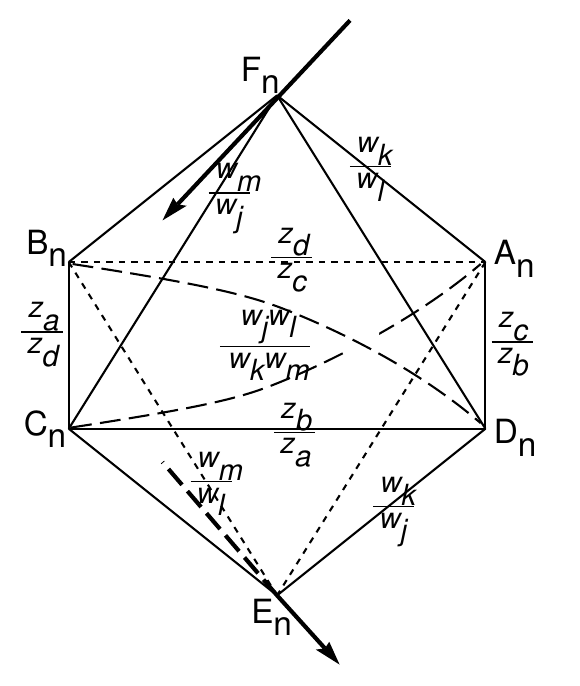}}\hspace{0.5cm}
  \subfigure[Negative crossing]
  {\includegraphics[scale=0.6]{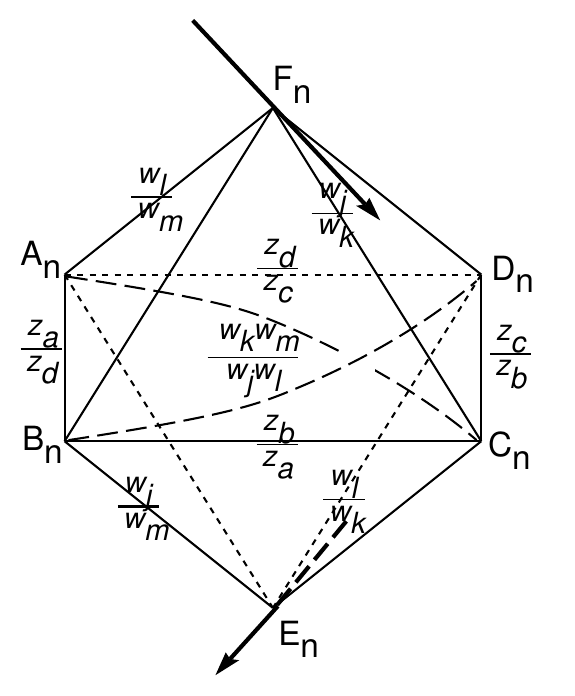}}
  \caption{Assignment of shape parameters}\label{pic9}
\end{figure}

We do not assign any shape parameters to the collapsed edges.
Also, in the case of Thurston triangulation, we do not assign any shape parameters to the edges that contain
the endpoints of the collapsed edges. For example, if ${\rm C}_n{\rm D}_n$ is collapsed,
then we do not assign any shape parameters to ${\rm C}_n{\rm F}_n$,
${\rm D}_n{\rm E}_n$ nor ${\rm B}_n{\rm D}_n$.
Also, if ${\rm D}_n{\rm E}_n$ is collapsed in Figure \ref{pic9}(a),
then we do not assign any shape parameters to ${\rm B}_n{\rm D}_n$, ${\rm B}_n{\rm E}_n$,
${\rm C}_n{\rm D}_n$ nor ${\rm D}_n{\rm A}_n$.\footnote{
The edges ${\rm C}_n{\rm D}_n$ and ${\rm D}_n{\rm A}_n$ are horizontal edges,
but are identified to non-horizontal edges. When this happens, we do not assign shape parameters to these edges.}

Yokota and Thurston triangulations are ideal triangulations, so by assigning shape parameters,
we can determine all the shapes of the hyperbolic ideal tetrahedra of the triangulations.
Note that if we assign a shape parameter $u\in\mathbb{C}-\{0,1\}$ to an edge of an ideal tetrahedron,
then other edges are also parametrized by $u, u':=\frac{1}{1-u}$ and $u'':=1-\frac{1}{u}$
as in Figure \ref{pic10}.

\begin{figure}[h]
\begin{center}
  {\setlength{\unitlength}{0.4cm}
  \begin{picture}(12,10)\thicklines
   \put(1,1){\line(1,0){8}}
   \put(1,1){\line(1,2){4}}
   \put(5,9){\line(1,-2){4}}
   \put(9,1){\line(1,2){2}}
   \put(5,9){\line(3,-2){6}}
   \dashline{0.5}(1,1)(11,5)
   \put(0,0){A}
   \put(9,0){B}
   \put(11,5){C}
   \put(4.5,9.5){D}
   \put(5,0){$u$}
   \put(8,7){$u$}
   \put(2.2,5){$u'$}
   \put(10.5,3){$u'$}
   \put(5,3){$u''$}
   \put(7,5){$u''$}
  \end{picture}}
  \caption{Parametrization of a hyperbolic ideal tetrahedron with shape parameter $u$}\label{pic10}
\end{center}
\end{figure}
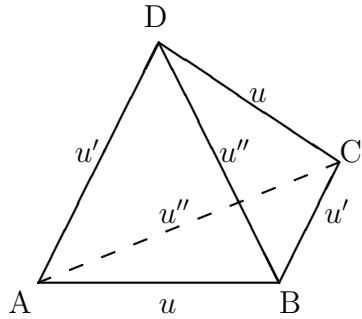

So as to get the hyperbolic structure,
these shape parameters should satisfy the {\it edge relations} and the {\it cusp conditions}. The edge relations mean
the product of all shape parameters assigned to each edge should be 1, and the cusp conditions mean
the holonomies induced by the longitude and the meridian should be translations on the cusp.
These two conditions can be expressed by a set of equations of the shape parameters, and
we call this set of equations {\it hyperbolicity equations}
(for details, see Chapter 4 of \cite{Thurston}).
We call a solution $(z_1,\ldots,z_g)$ of the hyperbolicity equations
of Yokota triangulation {\it essential}
if none of the shape parameters of the tetrahedra are one of $0,1,\infty$.
We also define an {\it essential solution} $(w_1,\ldots,w_m)$ of Thurston triangulation in the same way.
It is a well-known fact that if the hyperbolicity equations have an essential solution, 
then they have the unique solution which gives the hyperbolic structure to the triangulation\footnote{
Strictly speaking, we have unique values of shape parameters.
However, these values uniquely determine the solutions $(z_1^{(0)},\ldots,z_g^{(0)})$ and
$(w_1^{(0)},\ldots,w_m^{(0)})$. This was explained in \cite{Yokota10} for Yokota triangulation,
which will be at the end of this section for Thurston triangulation.} 
(for details, see Section 2.8 of \cite{Tillmann05}).
We call this unique solution {\it the geometric solution}, 
and denote the geometric solution of Yokota triangulation by $\bold{z}^{(0)}=(z_1^{(0)},\ldots,z_g^{(0)})$
and that of Thurston triangulation by $\bold{w}^{(0)}=(w_1^{(0)},\ldots,w_m^{(0)})$
We remark that, in Theorem \ref{thm},
we assumed the existence of the geometric solutions $\bold{z}^{(0)}$ and $\bold{w}^{(0)}$.

Yokota proved in \cite{Yokota10} that, for the potential function $V$ defined in Section \ref{ch31},\\
$\mathcal{H}_1=\left\{\exp\left(z_k\frac{\partial V}{\partial z_k}\right)=1~\vert~ k=1,\ldots,g\right\}$ becomes
the hyperbolicity equations of Yokota triangulation.
In other words, each element of $\mathcal{H}_1$
becomes an edge relation or a cusp condition
for all $k=1,\ldots,g$, and all other equations are trivially induced from the elements of $\mathcal{H}_1$.

Proposition \ref{prop11} shows the same holds for the potential function $W$ defined in Section \ref{ch32} and
$\mathcal{H}_2=\left\{\exp\left(w_l\frac{\partial W}{\partial w_l}\right)=1~\vert~ l=1,\ldots,m\right\}$.
We prove this in this section.

Let $\mathcal{A}$ be the set of non-collapsed horizontal edges of Thurston triangulation of $S^3-K$.
Let $\mathcal{B}$ be the set of non-collapsed non-horizontal edges
${\rm A}_n{\rm E}_n$, ${\rm B}_n{\rm E}_n$, ${\rm C}_n{\rm E}_n$, ${\rm D}_n{\rm E}_n$,
${\rm A}_n{\rm F}_n$, ${\rm B}_n{\rm F}_n$, ${\rm C}_n{\rm F}_n$, ${\rm D}_n{\rm F}_n$
in Figure \ref{pic9}, which are not in $\mathcal{A}$.\footnote{
Collapsing may identify some horizontal edges to non-horizontal edges.
In this case, we put these identified edges in $\mathcal{A}$.}
Finally, let $\mathcal{C}$ be the set of edges
${\rm A}_n{\rm C}_n$, ${\rm B}_n{\rm D}_n$ in Figure \ref{pic9}, which are not in $\mathcal{A\cup B}$.

For example, in Figure \ref{pic3},
$\mathcal{A}=\left\{\right.{\rm A}_7{\rm B}_7={\rm B}_6{\rm C}_6={\rm D}_2{\rm A}_2={\rm D}_2{\rm F}_2
={\rm A}_2{\rm B}_2={\rm B}_2{\rm F}_2={\rm C}_2{\rm F}_2={\rm A}_3{\rm F}_3={\rm B}_3{\rm F}_3={\rm D}_3{\rm F}_3
={\rm D}_5{\rm E}_5$, ${\rm D}_6{\rm A}_6={\rm B}_5{\rm C}_5$,
${\rm C}_6{\rm D}_6={\rm C}_5{\rm D}_5={\rm C}_3{\rm D}_3={\rm D}_7{\rm A}_7={\rm A}_7{\rm E}_7
={\rm C}_7{\rm D}_7={\rm C}_7{\rm E}_7={\rm A}_2{\rm E}_2={\rm C}_2{\rm E}_2={\rm B}_2{\rm E}_2
={\rm A}_6{\rm E}_6={\rm B}_6{\rm E}_6={\rm C}_6{\rm E}_6={\rm C}_5{\rm F}_5$,
${\rm D}_5{\rm A}_5={\rm B}_3{\rm C}_3$,
$\left.{\rm C}_2{\rm D}_2={\rm B}_7{\rm C}_7={\rm D}_3{\rm A}_3\right\}$,
$\mathcal{B}=\left\{\right.{\rm D}_3{\rm E}_3={\rm B}_7{\rm F}_7={\rm D}_7{\rm F}_7
={\rm A}_6{\rm F}_6={\rm B}_6{\rm F}_6={\rm D}_6{\rm F}_6
={\rm B}_5{\rm E}_5={\rm C}_5{\rm E}_5={\rm A}_5{\rm E}_5={\rm C}_3{\rm F}_3$,
${\rm A}_7{\rm F}_7={\rm C}_6{\rm F}_6$,
${\rm D}_6{\rm E}_6={\rm B}_5{\rm F}_5={\rm D}_5{\rm F}_5={\rm A}_5{\rm F}_5
={\rm A}_3{\rm E}_3={\rm B}_3{\rm E}_3={\rm C}_3{\rm E}_3={\rm C}_7{\rm F}_7$,
${\rm B}_7{\rm E}_7={\rm D}_2{\rm E}_2\left.\right\}$
and $\mathcal{C}=\emptyset$.

\begin{lem}\label{lem41}
For a hyperbolic knot $K$ with a fixed diagram, we assume the assumptions of Proposition \ref{prop11}.
Then the edges in $\mathcal{B}\cup\mathcal{C}$ satisfy the edge relations trivially
by the assigning rule of the shape parameters.
\end{lem}

\begin{proof}
If an edge ${\rm A}_n{\rm C}_n$ or ${\rm B}_n{\rm D}_n$ of Figure \ref{pic9} is in $\mathcal{C}$, then the octahedron
${\rm A}_n{\rm B}_n{\rm C}_n{\rm D}_n{\rm E}_n{\rm F}_n$ does not have any collapsed edge.
By the assigning rule of the shape parameters, all the edges in $\mathcal{C}$ satisfy edge relations trivially.

Now  we show the case of $\mathcal{B}$.
Consider the following four cases of two points $n_1$ and $n_2$ in Figure \ref{pic11}
and the two regions between the crossings parametrized by the variables $w_a$ and $w_b$
(for the positions of the points ${\rm A}_{n_1}, {\rm B}_{n_1}, \ldots,{\rm F}_{n_2}$, see Figure \ref{pic2}).
First, we assume no edges are collapsed in the tetrahedra 
${\rm A}_{n_1}{\rm B}_{n_1}{\rm D}_{n_1}{\rm F}_{n_1}$ and ${\rm C}_{n_2}{\rm B}_{n_2}{\rm D}_{n_2}{\rm F}_{n_2}$.
This means the two regions with $w_a$ and $w_b$ in Figure \ref{pic11} are bounded.

\begin{figure}[h]
\centering
  \subfigure[]
  {\begin{picture}(5,2)\thicklines
   \put(5,1){\vector(-1,0){5}}
   \put(1,2){\line(0,-1){0.8}}
   \put(1,0){\line(0,1){0.8}}
   \put(4,2){\line(0,-1){0.8}}
   \put(4,0){\line(0,1){0.8}}
   \put(1.2,0.6){$n_1$}
   \put(4.1,0.6){$n_2$}
   \put(2.3,1.4){$w_b$}
   \put(2.3,0.2){$w_a$}
  \end{picture}}\hspace{0.5cm}
  \subfigure[]
  {\begin{picture}(5,2)\thicklines
   \put(4.2,1){\line(1,0){0.8}}
   \put(1,2){\line(0,-1){2}}
   \put(0.8,1){\vector(-1,0){0.8}}
   \put(4,2){\line(0,-1){2}}
   \put(1.2,1){\line(1,0){2.6}}
   \put(1.2,0.6){$n_1$}
   \put(4.1,0.6){$n_2$}
   \put(2.3,1.4){$w_b$}
   \put(2.3,0.2){$w_a$}
  \end{picture}}\\
  \subfigure[]
  {\begin{picture}(5,2)\thicklines
   \put(4.2,1){\line(1,0){0.8}}
   \put(1,2){\line(0,-1){0.8}}
   \put(3.8,1){\vector(-1,0){3.8}}
   \put(4,2){\line(0,-1){2}}
   \put(1,0){\line(0,1){0.8}}
   \put(1.2,0.6){$n_1$}
   \put(4.1,0.6){$n_2$}
   \put(2.3,1.4){$w_b$}
   \put(2.3,0.2){$w_a$}
  \end{picture}}\hspace{0.5cm}
  \subfigure[]
  {\begin{picture}(5,2)\thicklines
   \put(1.2,1){\line(1,0){3.8}}
   \put(1,2){\line(0,-1){2}}
   \put(4,0){\line(0,1){0.8}}
   \put(4,2){\line(0,-1){0.8}}
   \put(0.8,1){\vector(-1,0){0.8}}
   \put(1.2,0.6){$n_1$}
   \put(4.1,0.6){$n_2$}
   \put(2.3,1.4){$w_b$}
   \put(2.3,0.2){$w_a$}
  \end{picture}}
  \caption{Four cases}\label{pic11}
\end{figure}
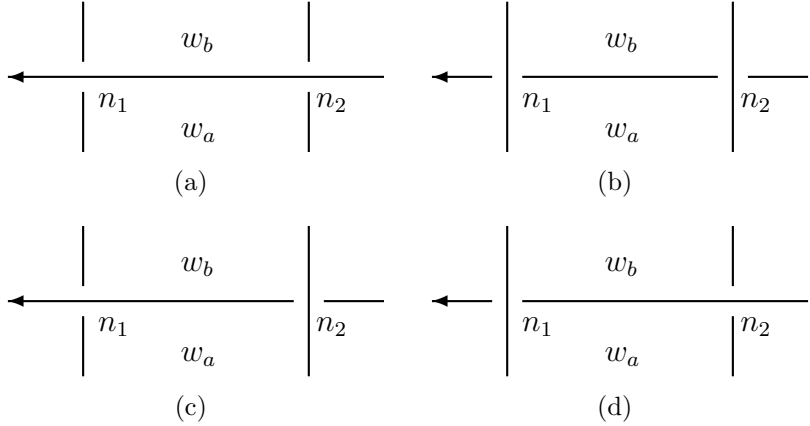

In the case of Figure \ref{pic11}(a), we want to prove that the edge relation of the edge 
${\rm A}_{n_1}{\rm F}_{n_1}={\rm C}_{n_2}{\rm F}_{n_2}\in\mathcal{B}$ holds trivially.
We draw a part of the cusp diagram in
${\rm A}_{n_1}{\rm B}_{n_1}{\rm D}_{n_1}{\rm F}_{n_1}\cup
{\rm C}_{n_2}{\rm B}_{n_2}{\rm D}_{n_2}{\rm F}_{n_2}$ near ${\rm F}_{n_1}={\rm F}_{n_2}$ as in Figure \ref{pic12}.
Our tetrahedra are all ideal, so the triangles $\triangle \alpha_1\alpha_2\alpha_3$
and $\triangle \alpha_1\alpha_4\alpha_5$ are Euclidean. Note that
$\alpha_1, \ldots,\alpha_5$ are points in the edges
${\rm A}_{n_1}{\rm F}_{n_1}={\rm C}_{n_2}{\rm F}_{n_2}$, ${\rm B}_{n_1}{\rm F}_{n_1}$,
${\rm D}_{n_1}{\rm F}_{n_1}$, ${\rm D}_{n_2}{\rm F}_{n_2}$, ${\rm B}_{n_2}{\rm F}_{n_2}$, respectively.
Furthermore, edges $\alpha_1\alpha_2$ and $\alpha_1\alpha_3$ are identified to $\alpha_1\alpha_5$ and to $\alpha_1\alpha_4$,
respectively.\footnote{
In fact, edges $\alpha_2\alpha_3$ and $\alpha_5\alpha_4$ are also identified,
so the two triangles are cancelled by each other. This means the corresponding tetrahedra
${\rm A}_{n_1}{\rm B}_{n_1}{\rm D}_{n_1}{\rm F}_{n_1}$ and
${\rm C}_{n_2}{\rm B}_{n_2}{\rm D}_{n_2}{\rm F}_{n_2}$
are cancelled by each other.}
On the edge ${\rm A}_{n_1}{\rm F}_{n_1}={\rm C}_{n_2}{\rm F}_{n_2}$,
two shape parameters $w_a/w_b$ and $w_b/w_a$ are assigned respectively by the assigning rule,
so the edge relation of ${\rm A}_{n_1}{\rm F}_{n_1}={\rm C}_{n_2}{\rm F}_{n_2}\in\mathcal{B}$ holds trivially.

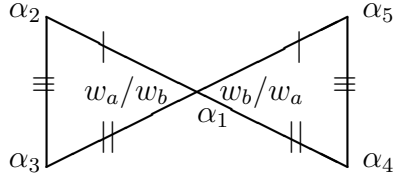
\begin{figure}[h]
\centering
\begin{picture}(4,2)
  {\thicklines\put(0,0){\line(2,1){4}}
  \put(0,0){\line(0,1){2}}
  \put(0,2){\line(2,-1){4}}
  \put(4,0){\line(0,1){2}}}
  \put(2,0.6){$\alpha_1$}
  \put(-0.5,2){$\alpha_2$}
  \put(-0.5,0){$\alpha_3$}
  \put(4.2,0){$\alpha_4$}
  \put(4.2,2){$\alpha_5$}
  \put(0.5,0.9){$w_a/w_b$}
  \put(2.3,0.9){$w_b/w_a$}
  \put(0.7,1.5){$|$}
  \put(3.3,1.5){$|$}
  \put(0.7,0.3){$||$}
  \put(3.2,0.3){$||$}
  \put(-0.2,1){$\equiv$}
  \put(3.8,1){$\equiv$}
\end{picture}\caption{Part of the cusp diagram of Figure \ref{pic11}(a)}\label{pic12}
\end{figure}

In the case of Figure \ref{pic11}(c), we want to prove that the edge relation of ${\rm A}_{n_1}{\rm F}_{n_1}\in\mathcal{B}$ holds trivially.
If $n_2$ is a positive crossing, then we draw a part of the cusp diagram in ${\rm A}_{n_1}{\rm B}_{n_1}{\rm D}_{n_1}{\rm F}_{n_1}\cup
{\rm A}_{n_2}{\rm C}_{n_2}{\rm D}_{n_2}{\rm E}_{n_2}$ near ${\rm F}_{n_1}={\rm E}_{n_2}$,
and if $n_2$ is a negative crossing,
then we draw a part of the cusp diagram in ${\rm A}_{n_1}{\rm B}_{n_1}{\rm D}_{n_1}{\rm F}_{n_1}\cup
{\rm A}_{n_2}{\rm B}_{n_2}{\rm C}_{n_2}{\rm E}_{n_2}$ near ${\rm F}_{n_1}={\rm E}_{n_2}$
as in Figure \ref{pic13}.

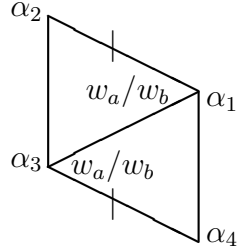
\begin{figure}[h]
\centering
\begin{picture}(3,3)
  {\thicklines
  \put(0,1){\line(2,1){2}}
  \put(0,1){\line(0,1){2}}
  \put(0,3){\line(2,-1){2}}
  \put(2,2){\line(0,-1){2}}
  \put(0,1){\line(2,-1){2}}}
  \put(2.1,1.8){$\alpha_1$}
  \put(-0.5,3){$\alpha_2$}
  \put(-0.5,1){$\alpha_3$}
  \put(2.1,0){$\alpha_4$}
  \put(0.5,1.9){$w_a/w_b$}
  \put(0.3,0.9){$w_a/w_b$}
  \put(0.8,2.5){$|$}
  \put(0.8,0.4){$|$}
\end{picture}\caption{Part of the cusp diagram of Figure \ref{pic11}(c)}\label{pic13}
\end{figure}

Note that if $n_2$ is a positive crossing, then
$\alpha_1, \ldots,\alpha_4$ are points in the edges
${\rm A}_{n_1}{\rm F}_{n_1}={\rm A}_{n_2}{\rm E}_{n_2}$, ${\rm B}_{n_1}{\rm F}_{n_1}$,
${\rm D}_{n_1}{\rm F}_{n_1}={\rm D}_{n_2}{\rm E}_{n_2}$, ${\rm C}_{n_2}{\rm E}_{n_2}$, respectively,
and if $n_2$ is a negative crossing, then
$\alpha_1, \ldots,\alpha_4$ are points in the edges
${\rm A}_{n_1}{\rm F}_{n_1}={\rm C}_{n_2}{\rm E}_{n_2}$, ${\rm B}_{n_1}{\rm F}_{n_1}$,
${\rm D}_{n_1}{\rm F}_{n_1}={\rm B}_{n_2}{\rm E}_{n_2}$, ${\rm A}_{n_2}{\rm E}_{n_2}$, respectively.
Furthermore, the edge $\alpha_2\alpha_1$ is identified to $\alpha_3\alpha_4$, so the diagram in Figure \ref{pic13}
becomes an annulus. The product of shape parameters around $\alpha_1=\alpha_4$ in the annulus is
$\displaystyle\frac{w_a}{w_b}\left(\frac{w_a}{w_b}\right)'\left(\frac{w_a}{w_b}\right)''=-1$,
and the one around $\alpha_2=\alpha_3$
is also $-1$. Therefore, if we consider the previous annulus on the right of Figure \ref{pic13}, which shares the edge $\alpha_1\alpha_4$,
then we obtain the edge relation of ${\rm A}_{n_1}{\rm F}_{n_1}$ trivially.

We remark that the previous annulus always exists because, when we follow the horizontal line in Figure \ref{pic11}(c) backwards, 
after meeting the under-crossing point $n_2$, we let the next over-crossing point $n_3$ (see Figure \ref{ann}).
(If $n_3$ does not exist, then ${\rm A}_{n_1}{\rm F}_{n_1}\in\mathcal{A}$ but this violates our assumption.)
Then a part of the cusp diagram between $n_2$ and $n_3$ also forms an annulus, and this is the previous annulus.\footnote{
As we have seen in the case of Figure \ref{pic11}(a), the crossing points between $n_2$ and $n_3$ 
do not have any effect on the part of the cusp diagram because the triangles in Figure \ref{pic12} are cancelled by each other.
Also, as explained below, the existence of the previous annulus still holds even 
if some regions between $n_2$ and $n_3$ are unbounded.}

\begin{figure}[h]
\centering
\begin{picture}(9,2)\thicklines
   \put(4.2,1){\line(1,0){4}}
   \put(1,2){\line(0,-1){0.8}}
   \put(3.8,1){\vector(-1,0){3.8}}
   \put(4,2){\line(0,-1){2}}
   \put(1,2){\line(0,-1){0.8}}
   \put(1,0){\line(0,1){0.8}}
   \put(7.5,2){\line(0,-1){0.8}}
   \put(7.5,0){\line(0,1){0.8}}
   \put(1.2,0.6){$n_1$}
   \put(4.1,0.6){$n_2$}
   \put(7.7,0.6){$n_3$}
   \put(2.3,1.4){$w_b$}
   \put(2.3,0.2){$w_a$}
   \put(5.5,1.4){$\cdots$}
   \put(5.5,0.2){$\cdots$}
  \end{picture}
  \caption{Previous annulus}\label{ann}
\end{figure}
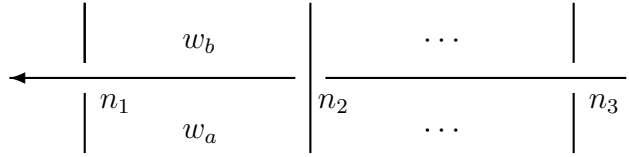
 
The cases of Figure \ref{pic11}(b) and Figure \ref{pic11}(d) are the same as
the cases of Figure \ref{pic11}(a) and Figure \ref{pic11}(c), respectively.
Therefore, we find all the edges in $\mathcal{B}$ satisfy the edge relations trivially
by the method of parametrizing edges.

Now we assume one of the regions parametrized by $w_a$ or $w_b$ in Figure \ref{pic11} is an unbounded region.
Then the cusp diagram in Figure \ref{pic12} collapses to an edge $\alpha_2\alpha_3=\alpha_5\alpha_4$ and
the one in Figure \ref{pic13} collapses to an edge $\alpha_2\alpha_3=\alpha_1\alpha_4$.
Therefore, our arguments for $\mathcal{B}$ still hold for the collapsed case.\footnote{
What we need is to consider the next annuli on the left and the right side, and do the same arguments.}

\end{proof}

\begin{proof}[Proof of Proposition \ref{prop11}]

Consider the function $P_1(w_j,w_k,w_l,w_m)$, which previously appeared in Section \ref{ch32}.
By direct calculation, we obtain
\begin{eqnarray}
\exp\left(w_j\frac{\partial P_1}{\partial w_j}\right)&=&\left(\frac{w_jw_l}{w_kw_m}\right)'
    \left(\frac{w_m}{w_j}\right)''\left(\frac{w_k}{w_j}\right)''\label{eq5},\\
\exp\left(w_k\frac{\partial P_1}{\partial w_k}\right)&=&\left(\frac{w_jw_l}{w_kw_m}\right)''
    \left(\frac{w_k}{w_l}\right)'\left(\frac{w_k}{w_j}\right)'\label{eq6},\\
\exp\left(w_l\frac{\partial P_1}{\partial w_l}\right)&=&\left(\frac{w_jw_l}{w_kw_m}\right)'
    \left(\frac{w_m}{w_l}\right)''\left(\frac{w_k}{w_l}\right)''\label{eq7},\\
\exp\left(w_m\frac{\partial P_1}{\partial w_m}\right)&=&\left(\frac{w_jw_l}{w_kw_m}\right)''
    \left(\frac{w_m}{w_l}\right)'\left(\frac{w_m}{w_j}\right)'\label{eq8}.
\end{eqnarray}
Note that (\ref{eq5}), (\ref{eq6}), (\ref{eq7}) and (\ref{eq8}) are the products of shape parameters
assigned to the edges ${\rm C}_n{\rm D}_n$, ${\rm D}_n{\rm A}_n$, 
${\rm A}_n{\rm B}_n$ and ${\rm B}_n{\rm C}_n$ of Figure \ref{pic9}(a), respectively.\footnote{
For example, consider equation (\ref{eq5}) and Figure \ref{pic9}(a). 
The shape parameters assigned to the edge ${\rm C}_n{\rm D}_n$
are  $\left(\frac{w_jw_l}{w_kw_m}\right)'$, $\left(\frac{w_m}{w_j}\right)''$ and $\left(\frac{w_k}{w_j}\right)''$,
which come from the tetrahedra ${\rm C}_n{\rm D}_n{\rm A}_n{\rm B}_n$, ${\rm C}_n{\rm D}_n{\rm B}_n{\rm F}_n$ and
${\rm C}_n{\rm D}_n{\rm A}_n{\rm E}_n$, respectively.}
Also, after evaluating $w_l=0$ to $P_1$, we obtain
\begin{eqnarray}
\exp\left(w_j\frac{\partial P_1(w_j,w_k,0,w_m)}{\partial w_j}\right)&=&
    \left(\frac{w_m}{w_j}\right)''\left(\frac{w_k}{w_j}\right)'',\label{eq9}\\
\exp\left(w_k\frac{\partial P_1(w_j,w_k,0,w_m)}{\partial w_k}\right)&=&
    \frac{w_m}{w_j}\left(\frac{w_k}{w_j}\right)',\label{eq10}\\
\exp\left(w_m\frac{\partial P_1(w_j,w_k,0,w_m)}{\partial w_m}\right)&=&
    \left(\frac{w_m}{w_j}\right)'\frac{w_k}{w_j}.\label{eq11}
\end{eqnarray}
Note that (\ref{eq9}), (\ref{eq10}) and (\ref{eq11}) are the products of shape parameters assigned to
the edges ${\rm C}_n{\rm D}_n$, ${\rm D}_n{\rm A}_n$ and ${\rm B}_n{\rm C}_n$ of Figure \ref{pic9}(a),
respectively, after collapsing the edge ${\rm A}_n{\rm B}_n$.
Direct calculation shows the same relations hold for $P_2$, $P_3$, $P_4$, $N_1$, $N_2$, $N_3$ and $N_4$.

Consider the first potential function for the end point of $I$ in Section \ref{ch32}.
Direct calculation shows
\begin{eqnarray}
\exp\left(w_l\frac{\partial P_1(w_j,w_j,w_l,w_m)}{\partial w_l}\right)&=&
   \exp\left(w_l\frac{\partial P_1(w_j,w_j,w_l,0)}{\partial w_l}\right)=
   \left(\frac{w_j}{w_l}\right)''\label{eq12},\\
\exp\left(w_m\frac{\partial P_1(w_j,w_j,w_l,w_m)}{\partial w_m}\right)&=&
   \exp\left(w_m\frac{\partial P_1(w_j,w_j,0,w_m)}{\partial w_m}\right)=
\left(\frac{w_m}{w_j}\right)'\label{eq13},\\
\exp\left(w_j\frac{\partial P_1(w_j,w_j,w_l,w_m)}{\partial w_j}\right)&=&\left(\frac{w_j}{w_m}\right)''
    \left(\frac{w_l}{w_j}\right)'=\left(\frac{w_m}{w_j}\right)''\left(\frac{w_j}{w_l}\right)'
    \frac{w_m}{w_l}\label{eq14},\\
\exp\left(w_j\frac{\partial P_1(w_j,w_j,0,w_m)}{\partial w_j}\right)&=&\left(\frac{w_j}{w_m}\right)''
=\left(\frac{w_m}{w_j}\right)''
    \frac{w_m}{w_j}\,(-1)\label{eq15},\\
\exp\left(w_j\frac{\partial P_1(w_j,w_j,w_l,0)}{\partial w_j}\right)&=&\left(\frac{w_l}{w_j}\right)'
    =\left(\frac{w_j}{w_l}\right)'
    \frac{w_j}{w_l}\,(-1)\label{eq16},
\end{eqnarray}
where (\ref{eq12}) and (\ref{eq13}) are the products of shape parameters assigned to
the edges ${\rm A}_n{\rm B}_n$ and ${\rm B}_n{\rm C}_n$ of Figure \ref{pic9}(a), respectively, after collapsing
the edge ${\rm D}_n{\rm E}_n$ without or with the collapsing of a horizontal edge.

To explain that (\ref{eq14}), (\ref{eq15}) and (\ref{eq16}) are still parts of edge relations,
we need different arguments. First, consider Figure \ref{pic14}.

\begin{figure}[h]
\centering
  \subfigure[From Figure \ref{pic12}]
  {\begin{picture}(4,2)
  {\thicklines
  \put(0,0){\line(2,1){4}}
  \put(0,0){\line(0,1){2}}
  \put(0,2){\line(2,-1){4}}
  \put(4,0){\line(0,1){2}}}
  \put(0,0){\circle*{0.2}}
  \put(0,2){\circle*{0.2}}
  \put(4,0){\circle*{0.2}}
  \put(4,2){\circle*{0.2}}
  \put(0.5,0.9){$w_a/w_b$}
  \put(2.3,0.9){$w_b/w_a$}
  \put(0.7,1.5){$|$}
  \put(3.3,1.5){$|$}
  \put(0.7,0.3){$||$}
  \put(3.2,0.3){$||$}
  \put(-0.2,1){$\equiv$}
  \put(3.8,1){$\equiv$}
\end{picture}}\hspace{2cm}
  \subfigure[From Figure \ref{pic13}]
  {\begin{picture}(3,3)
  {\thicklines
  \put(0,1){\line(2,1){2}}
  \put(0,1){\line(0,1){2}}
  \put(0,3){\line(2,-1){2}}
  \put(2,2){\line(0,-1){2}}
  \put(0,1){\line(2,-1){2}}}
  \put(0,1){\circle*{0.2}}
  \put(0,3){\circle*{0.2}}
  \put(0.5,1.9){$w_a/w_b$}
  \put(0.3,0.9){$w_a/w_b$}
  \put(0.8,2.5){$|$}
  \put(0.8,0.4){$|$}
\end{picture}}
  \caption{Parts of the cusp diagrams from Figure \ref{pic12} and Figure \ref{pic13}}\label{pic14}
\end{figure}
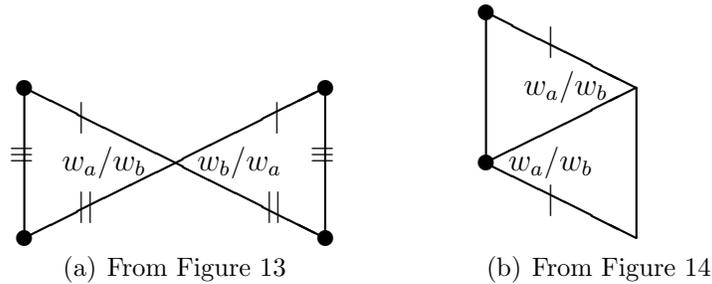

In Figure \ref{pic14}(a), the product of all shape parameters assigned to the edge expressed by dots is
\begin{equation}\label{eq17}
\left(\frac{w_a}{w_b}\right)'\left(\frac{w_a}{w_b}\right)''\left(\frac{w_b}{w_a}\right)'\left(\frac{w_b}{w_a}\right)''=1,
\end{equation}
and in Figure \ref{pic14}(b), the product is
\begin{equation}\label{eq18}
\left(\frac{w_a}{w_b}\right)'\left(\frac{w_a}{w_b}\right)''\frac{w_a}{w_b}=-1.
\end{equation}

To see the meaning of (\ref{eq14}), consider the following two cases in Figure \ref{pic15}, where $n_1$ is the end point of $I$
and $n_2$ is the previous over-crossing point. Figure \ref{pic15}(a) means 
the case when there is no crossing point between $n_1$ and $n_2$, and Figure \ref{pic15}(b) means the other case.

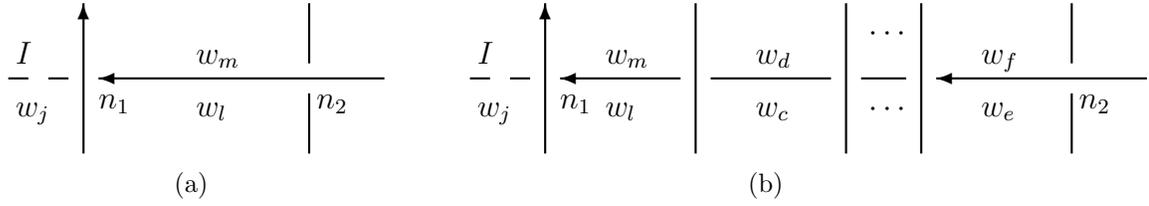
\begin{figure}[h]
\centering
\subfigure[]
{\begin{picture}(5,2)\thicklines
   \put(5,1){\vector(-1,0){3.8}}
   \put(1,0){\vector(0,1){2}}
   \put(4,0){\line(0,1){0.8}}
   \put(4,2){\line(0,-1){0.8}}
   \put(0,1){\dashline{0.5}(0,0)(0.8,0)}
   \put(1.2,0.6){$n_1$}
   \put(4.1,0.6){$n_2$}
   \put(2.5,0.5){$w_l$}
   \put(2.5,1.2){$w_m$}
   \put(0.1,1.2){$I$}
   \put(0.1,0.5){$w_j$}
\end{picture}}\hspace{1cm}
\subfigure[]
{\begin{picture}(8,2)\thicklines
   \put(9,1){\vector(-1,0){2.8}}
   \put(1,0){\vector(0,1){2}}
   \put(3,2){\line(0,-1){2}}
   \put(5,2){\line(0,-1){2}}
   \put(6,2){\line(0,-1){2}}
   \put(8,0){\line(0,1){0.8}}
   \put(8,2){\line(0,-1){0.8}}
   \put(2.8,1){\vector(-1,0){1.6}}
   \put(3.2,1){\line(1,0){1.6}}
   \put(5.2,1){\line(1,0){0.6}}
   \put(0,1){\dashline{0.5}(0,0)(0.8,0)}
   \put(1.2,0.6){$n_1$}
   \put(8.1,0.6){$n_2$}
   \put(1.8,0.5){$w_l$}
   \put(1.8,1.2){$w_m$}
   \put(3.8,0.5){$w_c$}
   \put(3.8,1.2){$w_d$}
   \put(6.8,0.5){$w_e$}
   \put(6.8,1.2){$w_f$}
   \put(5.3,1.5){$\cdots$}
   \put(5.3,0.5){$\cdots$}
   \put(0.1,1.2){$I$}
   \put(0.1,0.5){$w_j$}
\end{picture}}\caption{Two cases after the end point of $I$}\label{pic15}
\end{figure}

Because $n_1$ is the endpoint of $I$, the edge ${\rm D}_{n_1}{\rm E}_{n_1}$ 
of the octahedron on $n_1$ in Figure \ref{pic9}(a) is collapsed to a point ${\rm D}_{n_1}={\rm E}_{n_1}$
and becomes two tetrahedra as in Figure \ref{pic20} (if one more horizontal edge is collapsed here,
the result becomes one tetrahedron. This is the cases of equations (\ref{eq15}) and (\ref{eq16})).

\begin{figure}[h]
\centering
\includegraphics[scale=0.6]{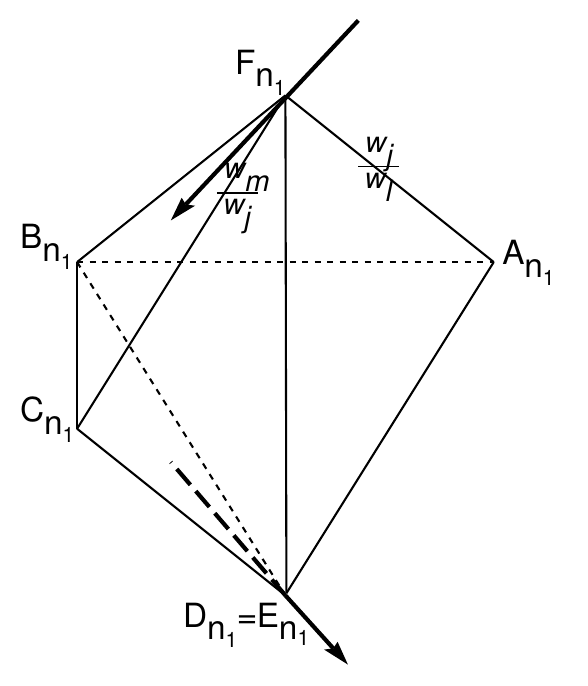}
\caption{Figure \ref{pic9}(a) after collapsing the edge ${\rm D}_{n_1}{\rm E}_{n_1}$}\label{pic20}
\end{figure}

The part of the cusp diagrams for each case are in Figure \ref{pic16}
(see Figure \ref{label} and Figure \ref{pic9} for the assigning rule of the shape parameters).

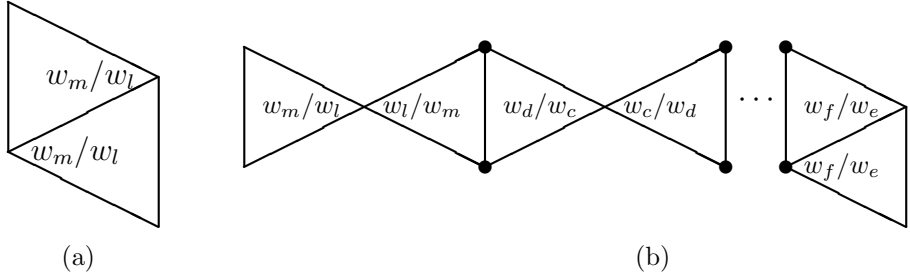
\begin{figure}[h]
\centering
\subfigure[]
{\begin{picture}(2,3)
  {\thicklines
  \put(0,1){\line(2,1){2}}
  \put(0,1){\line(0,1){2}}
  \put(0,3){\line(2,-1){2}}
  \put(2,2){\line(0,-1){2}}
  \put(0,1){\line(2,-1){2}}}
  \put(0.5,1.9){$w_m/w_l$}
  \put(0.3,0.9){$w_m/w_l$}
\end{picture}}\hspace{1cm}
\subfigure[]
  {\begin{picture}(11,3)
  {\setlength{\unitlength}{0.8cm}\thicklines
  \put(0,1){\line(2,1){4}}
  \put(0,1){\line(0,1){2}}
  \put(0,3){\line(2,-1){4}}
  \put(4,1){\line(0,1){2}}
  \put(4,1){\line(2,1){4}}
  \put(4,3){\line(2,-1){4}}
  \put(8,1){\line(0,1){2}}
  \put(4,1){\circle*{0.2}}
  \put(4,3){\circle*{0.2}}
  \put(8,1){\circle*{0.2}}
  \put(8,3){\circle*{0.2}}
  \put(9,1){\circle*{0.2}}
  \put(9,3){\circle*{0.2}}
  \put(8.2,2){$\cdots$}
  \put(0.3,1.9){\footnotesize$ w_m/w_l$}
  \put(2.3,1.9){\footnotesize$w_l/w_m$}
  \put(4.3,1.9){\footnotesize$w_d/w_c$}
  \put(6.3,1.9){\footnotesize$w_c/w_d$}
  \put(9.3,1.9){\footnotesize$w_f/w_e$}
  \put(9.3,0.9){\footnotesize$w_f/w_e$}
  \put(9,1){\line(2,1){2}}
  \put(9,1){\line(0,1){2}}
  \put(9,3){\line(2,-1){2}}
  \put(11,2){\line(0,-1){2}}
  \put(9,1){\line(2,-1){2}}}
\end{picture}}\caption{The parts of the cusp diagram corresponding to Figure \ref{pic15}}\label{pic16}
\end{figure}

In the case of Figure \ref{pic15}(a), the product of shape parameters assigned to the edges
${\rm C}_{n_1}{\rm D}_{n_1}={\rm D}_{n_1}{\rm A}_{n_1}$ of Figure \ref{pic20}
is $\left(\frac{w_m}{w_j}\right)''\left(\frac{w_j}{w_l}\right)'$. These edges are identified to
${\rm C}_{n_2}{\rm F}_{n_2}$, and $\frac{w_l}{w_m}$ is assigned to this edge.
This explains that (\ref{eq14}) is the product of shape parameters assigned to the edges
${\rm C}_{n_1}{\rm D}_{n_1}={\rm D}_{n_1}{\rm A}_{n_1}={\rm C}_{n_2}{\rm F}_{n_2}$.

In the case of Figure \ref{pic15}(b), the product of shape parameters assigned to the edges
${\rm C}_{n_1}{\rm D}_{n_1}={\rm D}_{n_1}{\rm A}_{n_1}$ of Figure \ref{pic20}
is $\left(\frac{w_m}{w_j}\right)''\left(\frac{w_j}{w_l}\right)'$. In Figure \ref{pic16}(b),
these edges are identified to the edges drawn by the dots,
and the product of shape parameters assigned to the edges is
$$\left(\frac{w_l}{w_m}\right)'\left(\frac{w_l}{w_m}\right)''\times 1\times\cdots\times (-1)
=\frac{w_m}{w_l}$$
by (\ref{eq17}) and (\ref{eq18}). This also explains (\ref{eq14})
is the product of shape parameters assigned to
${\rm C}_{n_1}{\rm D}_{n_1}={\rm D}_{n_1}{\rm A}_{n_1}$ and some other edges identified to this.
This fact is still true\footnote{
Even if the endpoint of $J$ lies between the crossings $n_1$ and $n_2$,
this fact is still true because the collapsing of the non-horizontal edges does not change the part of the cusp diagram
we are considering.}
even if some of the regions assigned by $w_c, w_d, \ldots,w_e, w_f$ are unbounded regions
because the collapsing of the horizontal edges makes the cusp diagrams of Figure \ref{pic12} and Figure \ref{pic13}
into edges. If the cusp diagram of Figure \ref{pic12} becomes an edge, then ignoring the diagram is enough for our consideration,
and if that of Figure \ref{pic13} becomes an edge, then considering the previous annulus is enough.
The previous annulus always exists because,
by the same argument as in the proof of Lemma \ref{lem41},
if we choose the next over-crossing point $n_3$ by following the horizontal lines backwards,
the cusp diagram between $n_2$ and $n_3$ becomes the previous annulus.\footnote{
There is a concern that the previous annulus is collapsed to an edge,
and all the previous annuli, following the horizontal line, are collapsed to edges. 
However, this cannot happen because Thurston triangulation is a triangulation of the hyperbolic knot complement $S^3-K$
and we assumed the existence of the geometric solution.}

Now we describe the meaning of (\ref{eq15}).
Let $n_1$ be the end point of $I$, $n_2$ be the previous over-crossing point and
$n_3$ be the previous under-crossing point. Also let $\tilde{n}$ be the previous point of $n_1$.
Assume the edges ${\rm D}_{n_1}{\rm E}_{n_1}$ and ${\rm A}_{n_1}{\rm B}_{n_1}$ of Figure \ref{pic9}(a) are collapsed.
Then ${\rm C}_{n_1}{\rm D}_{n_1}={\rm B}_{n_1}{\rm D}_{n_1}$,
and $\left(\frac{w_m}{w_j}\right)''\frac{w_m}{w_j}$ is assigned to this edge.
If $\tilde{n}=n_2$, then the edges identified to ${\rm C}_{n_1}{\rm D}_{n_1}={\rm B}_{n_1}{\rm D}_{n_1}$ appear
between the points $\tilde{n}=n_2$ and $n_3$ as the dots in Figure \ref{pic14},
and if $\tilde{n}\neq n_2$, then the edges appear between $\tilde{n}$ and $n_2$ in the same way.
Particularly, Figure \ref{pic14}(a) may appear many times, but Figure \ref{pic14}(b) appears only one time
at the points $n_3$ or $n_2$, respectively.
By (\ref{eq17}) and (\ref{eq18}), the product of all shape parameters assigned to the dots is $-1$, so
(\ref{eq15}) is the product of shape parameters assigned to the edges
${\rm C}_{n_1}{\rm D}_{n_1}={\rm B}_{n_1}{\rm D}_{n_1}$ and some others identified to these.
This fact is still true when some of the horizontal edges or non-horizontal edges of the octahedra are collapsed
because of the same reason explained above for the case of (\ref{eq14}).

The same relations hold for (\ref{eq16}) and the cases of other potential functions of the endpoints of $I$ and $J$ by the same arguments.

Therefore, we conclude that $\mathcal{H}_2$ becomes all the edge relations of $\mathcal{A}$ except
the one horizontal edge whose region is assigned as 0 instead of the variables $w_1,\ldots,w_m$.
For an ideal tetrahedron parametrized with $u\in\mathbb{C}$ as in Figure \ref{pic10},
the product of all shape parameters assigned to all edges in the tetrahedron is $(u u' u'')^2=1$.
This implies the product of all edge relations becomes 1.
On the other hand, from Lemma \ref{lem41} and the above arguments, we found all but one edge relation
by $\mathcal{H}_2$.
Therefore, the remaining edge relation holds automatically.

Finally, we prove $\mathcal{H}_2$ contains the cusp condition.
Note that edges $\alpha_1\alpha_4$ and $\alpha_2\alpha_3$ in Figure \ref{pic13} are meridians of the cusp diagram.
The same shape parameter $\frac{w_a}{w_b}$ is assigned to the corners
$\angle\alpha_2\alpha_1\alpha_3$ and $\angle\alpha_1\alpha_3\alpha_4$, so
one of the cusp conditions is trivially satisfied by the method of assigning shape parameters to edges.
If we have all the edge relations and one cusp condition of a meridian,
then we can obtain all remaining cusp conditions using these relations.
Therefore, we conclude $\mathcal{H}_2$ are the hyperbolicity equations
of Thurston triangulation of $S^3-K$.

\end{proof}

We remark one technical fact. For Thurston triangulation,
let the shape parameters of the ideal tetrahedra be $s_1,\ldots,s_h$. 
These parameters are defined by the ratios of a solution $w_1,\ldots,w_m$ of $\mathcal{H}_2$, so
if the values of $w_1,\ldots,w_m$ are fixed, then the values of $s_1,\ldots,s_h$ are uniquely determined
and satisfy the hyperbolicity equation.
Likewise, if the values of $s_1,\ldots,s_h$ satisfying the hyperbolicity equations are fixed,
then we can uniquely determine the solution of $w_1,\ldots,w_m$ of $\mathcal{H}_2$ as follows:
First, we can determine some of the values of $w_1,\ldots,w_m$, which are assigned to the regions adjacent to
the region assigned with the number 0. 
Once a value $w_l$ of a region is determined,
then all the values of the adjacent regions can be determined.
Therefore, all $w_1,\ldots,w_m$ can be determined. Furthermore, those values are well-defined
and become a solution of $\mathcal{H}_2$ because of the hyperbolicity equations.

In the next section, we will show the shape parameters of Yokota triangulation determines that of Thurston triangulation,
and with certain restriction, vice versa.
By the above discussion, this correspondence means
each essential solution of $\mathcal{H}_1$ determines a unique solution of $\mathcal{H}_2$. 
Furthermore, if all the determined solutions of $\mathcal{H}_2$ are essential, then each essential solution of $\mathcal{H}_2$
determines a unique essential solution of $\mathcal{H}_1$.

\section{Proof of Theorem \ref{thm}}\label{ch5}

We start this section with the proof of Lemma \ref{lem12}.

\begin{proof}[Proof of Lemma \ref{lem12}]
For a hyperbolic ideal octahedron in Figure \ref{pic17},
we assign shape parameters $t_1$, $t_2$, $t_3$, $t_4$, $u_1$, $u_2$, $u_3$ and $u_4$
to the edges CD, DA, AB, BC, CF, DE, AF and BE, respectively. Let $u_5:=\frac{1}{u_1u_3}=\frac{1}{u_2u_4}$,
which is also a shape parameter assigned to the edges AC and BD of the tetrahedron ABCD.

\begin{figure}[h]
\centering
\includegraphics[scale=0.7]{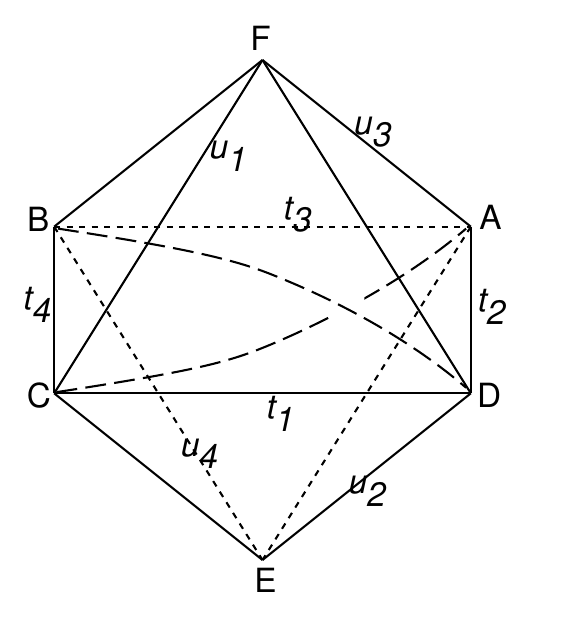}
  \caption{Assignment of shape parameters}\label{pic17}
\end{figure}
Then we obtain the following relations.
\begin{eqnarray}\label{eq20}
      \left\{
      \begin{array}{ll}
     u_1=t_1' t_4'',\\
     u_2=t_1' t_2'',\\
     u_3=t_3' t_2'',\\
     u_4=t_3' t_4'',\\
     u_5=\left(t_1't_2''t_3't_4''\right)^{-1},
      \end{array}\right.
      \left\{
      \begin{array}{ll}
     t_1=u_1'' u_2'' u_5',\\
     t_2=u_2' u_3' u_5'',\\
     t_3=u_3'' u_4'' u_5',\\
     t_4=u_4' u_1' u_5'',\\
     t_1t_2t_3t_4=1.
      \end{array}\right.
\end{eqnarray}

Note that $t_1,\ldots,t_4$ and $u_1,\ldots,u_5$ are the shape parameters of the tetrahedra 
in Yokota triangulation and in Thurston triangulation, respectively.
According to Observation \ref{obs}, we know these two triangulations are related by 3-2 moves and 4-5 moves
on collapsed octahedra and non-collapsed octahedra, respectively. 
Equation (\ref{eq20}) shows the correspondence between
the shape parameters under 4-5 moves, so if $t_1,\ldots,t_4\notin\{0,1,\infty\}$, then we can determine the values of $u_1,\ldots,u_5$
from the left side of (\ref{eq20}). Also the equation corresponding to 3-2 move can be obtained easily 
(see (\ref{eq30}) for example).
This implies that the shape parameters of Yokota triangulation determine that of Thurston triangulation. Furthermore,
if all $u_1,\ldots,u_5\notin\{0,1,\infty\}$, then the shape parameters of Thurston triangulation
recover that of Yokota triangulation by the right side of (\ref{eq20}). This completes the proof.

\end{proof}

Our goal of this section is to prove
\begin{equation*}
V_0(z_1,\ldots,z_g)\equiv W_0(w_1,\ldots,w_m)\modulos,
\end{equation*}
for any essential solution $(z_1,\ldots,z_g)$ of $\mathcal{H}_1$ and 
the corresponding essential solution $(w_1,\ldots,w_m)$ of $\mathcal{H}_2$.
To prove this, we introduce the dilogarithm identities of an ideal octahedron in Lemma \ref{lem5}.
Note that the functions $\li(z)$ and $\log z$ are multi-valued functions. Therefore, to obtain well-defined values,
we have to select a proper branch of the logarithm by choosing $\arg z$ and $\arg(1-z)$.

Let $D(z):=\imaginary\li(z)+\log\vert z\vert\arg(1-z)$ be the Bloch-Wigner function for $z\in\mathbb{C}-\{0,1\}$.
It is a well-known fact $D(z)$ is invariant under any choice of log-branch and that $D(z)=-D(\frac{1}{z})=\vol(T_z)$,
where $T_z$ is the hyperbolic ideal tetrahedron with the shape parameter $z$.
Therefore, from Figure \ref{pic17}, we obtain
\begin{equation}\label{eq21}
D(t_1)+D(t_2)+D(t_3)+D(t_4)=D(u_1)+D(u_2)+D(u_3)+D(u_4)+D(u_5).
\end{equation}

\begin{lem}\label{lem5}
Let $t_1, t_2, t_3, t_4, u_1, u_2, u_3, u_4, u_5\notin\{0,1,\infty\}$ be the shape parameters defined in the hyperbolic octahedron
in Figure \ref{pic17} satisfying (\ref{eq20}) and (\ref{eq21}).
Then the following identities hold for any choice of log-branch.

\begin{eqnarray}
\lefteqn{\li(t_1)-\li(\frac{1}{t_2})+\li(t_3)-\li(\frac{1}{t_4})}\nonumber\\
\lefteqn{~\equiv\li(u_1)+\li(u_2)-\li(\frac{1}{u_3})-\li(\frac{1}{u_4})+\li(u_5)-\frac{\pi^2}{6}+\log u_1\log u_2}\label{eq22}\\
  &&-\left(-\log(1-t_1)+\log(1-\frac{1}{t_4})\right)\log u_2-\left(-\log(1-t_1)+\log(1-\frac{1}{t_2})\right)\log u_1\nonumber\\
  &&+\left(-\log(1-t_1)+\log(1-\frac{1}{t_4})\right)\log(1-{u_1})
  +\left(-\log(1-t_1)+\log(1-\frac{1}{t_2})\right)\log(1-u_2)\nonumber\\
  &&+\left(-\log(1-t_3)+\log(1-\frac{1}{t_2})\right)\log(1-\frac{1}{u_3})
  +\left(-\log(1-t_3)+\log(1-\frac{1}{t_4})\right)\log(1-\frac{1}{u_4})\nonumber\\
  &&+\left(\log(1-t_1)-\log(1-\frac{1}{t_2})+\log(1-t_3)-\log(1-\frac{1}{t_4})\right)\log(1-{u_5})~~~\modulos\nonumber
\end{eqnarray}\vspace{-0.3cm}
\begin{eqnarray}
\lefteqn{~\equiv\li(u_1)-\li(\frac{1}{u_2})-\li(\frac{1}{u_3})+\li({u_4})-\li(\frac{1}{u_5})+\frac{\pi^2}{6}-\log u_2\log u_3}\label{eq23}\\
  &&+\left(-\log(1-t_3)+\log(1-\frac{1}{t_2})\right)\log u_2+\left(-\log(1-t_1)+\log(1-\frac{1}{t_2})\right)\log u_3\nonumber\\
  &&+\left(-\log(1-t_1)+\log(1-\frac{1}{t_4})\right)\log(1-{u_1})
  +\left(-\log(1-t_1)+\log(1-\frac{1}{t_2})\right)\log(1-\frac{1}{u_2})\nonumber\\
  &&+\left(-\log(1-t_3)+\log(1-\frac{1}{t_2})\right)\log(1-\frac{1}{u_3})
  +\left(-\log(1-t_3)+\log(1-\frac{1}{t_4})\right)\log(1-{u_4})\nonumber\\
  &&+\left(\log(1-t_1)-\log(1-\frac{1}{t_2})+\log(1-t_3)-\log(1-\frac{1}{t_4})\right)\log(1-\frac{1}{u_5})~~~\modulos\nonumber
\end{eqnarray}\vspace{-0.3cm}
\begin{eqnarray}
\lefteqn{~\equiv-\li(\frac{1}{u_1})-\li(\frac{1}{u_2})+\li(u_3)+\li(u_4)+\li(u_5)-\frac{\pi^2}{6}+\log u_3\log u_4}\label{eq24}\\
  &&-\left(-\log(1-t_3)+\log(1-\frac{1}{t_4})\right)\log u_3-\left(-\log(1-t_3)+\log(1-\frac{1}{t_2})\right)\log u_4\nonumber\\
  &&+\left(-\log(1-t_1)+\log(1-\frac{1}{t_4})\right)\log(1-\frac{1}{u_1})
  +\left(-\log(1-t_1)+\log(1-\frac{1}{t_2})\right)\log(1-\frac{1}{u_2})\nonumber\\
  &&+\left(-\log(1-t_3)+\log(1-\frac{1}{t_2})\right)\log(1-u_3)
  +\left(-\log(1-t_3)+\log(1-\frac{1}{t_4})\right)\log(1-u_4)\nonumber\\
  &&+\left(\log(1-t_1)-\log(1-\frac{1}{t_2})+\log(1-t_3)-\log(1-\frac{1}{t_4})\right)\log(1-u_5)~~~\modulos\nonumber
\end{eqnarray}\vspace{-0.3cm}
\begin{eqnarray}
\lefteqn{~\equiv-\li(\frac{1}{u_1})+\li(u_2)+\li(u_3)-\li(\frac{1}{u_4})-\li(\frac{1}{u_5})+\frac{\pi^2}{6}-\log u_1\log u_4}\label{eq25}\\
  &&+\left(-\log(1-t_1)+\log(1-\frac{1}{t_4})\right)\log u_4+\left(-\log(1-t_3)+\log(1-\frac{1}{t_4})\right)\log u_1\nonumber\\
  &&+\left(-\log(1-t_1)+\log(1-\frac{1}{t_4})\right)\log(1-\frac{1}{u_1})
  +\left(-\log(1-t_1)+\log(1-\frac{1}{t_2})\right)\log(1-u_2)\nonumber\\
  &&+\left(-\log(1-t_3)+\log(1-\frac{1}{t_2})\right)\log(1-u_3)
  +\left(-\log(1-t_3)+\log(1-\frac{1}{t_4})\right)\log(1-\frac{1}{u_4})\nonumber\\
  &&+\left(\log(1-t_1)-\log(1-\frac{1}{t_2})+\log(1-t_3)-\log(1-\frac{1}{t_4})\right)\log(1-\frac{1}{u_5})~~~\modulos.\nonumber
\end{eqnarray}
Furthermore,
\begin{eqnarray}
\lefteqn{\li({t_1})-\li(\frac{1}{t_2})-\li(\frac{1}{t_4})+\frac{\pi^2}{6}
\equiv\li(u_1)+\li(u_2)-\frac{\pi^2}{6}+\log u_1\log u_2\label{eq26}}\\
&&+\left(-\log(1-t_1)+\log(1-\frac{1}{t_4})\right)\left(-\log u_2+\log(1-u_1)\right)\nonumber\\
&&+\left(-\log(1-t_1)+\log(1-\frac{1}{t_2})\right)\left(-\log u_1+\log(1-u_2)\right)~~~\modulos\nonumber
\end{eqnarray}
when AB is collapsed to a point,
\begin{eqnarray}
\lefteqn{\li({t_1})-\li(\frac{1}{t_2})+\li({t_3})-\frac{\pi^2}{6}
\equiv-\li(\frac{1}{u_2})-\li(\frac{1}{u_3})+\frac{\pi^2}{6}-\log u_2\log u_3\label{eq27}}\\
&&+\left(-\log(1-t_3)+\log(1-\frac{1}{t_2})\right)\left(\log u_2+\log(1-\frac{1}{u_3})\right)\nonumber\\
&&+\left(-\log(1-t_1)+\log(1-\frac{1}{t_2})\right)\left(\log u_3+\log(1-\frac{1}{u_2})\right)~~~\modulos\nonumber
\end{eqnarray}
when BC is collapsed to a point,
\begin{eqnarray}
\lefteqn{-\li(\frac{1}{t_2})+\li(t_3)-\li(\frac{1}{t_4})+\frac{\pi^2}{6}
\equiv\li(u_3)+\li(u_4)-\frac{\pi^2}{6}+\log u_3\log u_4\label{eq28}}\\
&&+\left(-\log(1-t_3)+\log(1-\frac{1}{t_4})\right)\left(-\log u_3+\log(1-u_4)\right)\nonumber\\
&&+\left(-\log(1-t_3)+\log(1-\frac{1}{t_2})\right)\left(-\log u_4+\log(1-u_3)\right)~~~\modulos\nonumber
\end{eqnarray}
when CD is collapsed to a point, and
\begin{eqnarray}
\lefteqn{\li({t_1})+\li(t_3)-\li(\frac{1}{t_4})-\frac{\pi^2}{6}
\equiv-\li(\frac{1}{u_1})-\li(\frac{1}{u_4})+\frac{\pi^2}{6}-\log u_1\log u_4\label{eq29}}\\
&&+\left(-\log(1-t_1)+\log(1-\frac{1}{t_4})\right)\left(\log u_4+\log(1-\frac{1}{u_1})\right)\nonumber\\
&&+\left(-\log(1-t_3)+\log(1-\frac{1}{t_4})\right)\left(\log u_1+\log(1-\frac{1}{u_4})\right)~~~\modulos\nonumber
\end{eqnarray}
when DA is collapsed to a point.

\end{lem}

\begin{proof}
For a function $F$ consisting of dilogarithms and logarithms with certain fixed log-branch, 
we denote by $F^*$ the same function with different log-branch corresponding to an analytic continuation of $F$.
It is a well-known fact that
\begin{equation}\label{bra1}
\li^*(z)\equiv\li(z)+2a\pi i\log z\modulos
\end{equation}
for certain integer $a$. Let $A:=\li(z)-\left(z\frac{\partial \li(z)}{\partial z}\right)\log z$. Then using (\ref{bra1}), we have
\begin{eqnarray*}
A^*&=&\li^*(z)-\left(z\frac{\partial \li^*(z)}{\partial z}\right)\log^* z
\equiv\li(z)+2a\pi i\log z -\left(z\frac{\partial \li(z)}{\partial z}+2a\pi i\right)\log^*z\\
&\equiv &\li(z)-\left(z\frac{\partial \li(z)}{\partial z}\right)\log^* z\modulos
\end{eqnarray*}
and
\begin{equation}\label{bra2}
A^*-A\equiv-\left(z\frac{\partial \li(z)}{\partial z}\right)(\log^*z-\log z) \modulos.
\end{equation}
Similarly, for $B:=\li(1/z)-\left(z\frac{\partial \li(1/z)}{\partial z}\right)\log z$, we have
\begin{equation}\label{bra3}
B^*-B\equiv-\left(z\frac{\partial \li(1/z)}{\partial z}\right)(\log^*z-\log z) \modulos.
\end{equation}

Now, we consider (\ref{eq22}). Let 
\begin{eqnarray*}
X(t_1,\ldots,t_4)&:=&\li(t_1)-\li(\frac{1}{t_2})+\li(t_3)-\li(\frac{1}{t_4}),\\
X_0(t_1,\ldots,t_4)&:=&X-\sum_{k=1}^4 \left(t_k\frac{\partial X}{\partial t_k}\right)\log t_k,\\
Y(u_1,\ldots,u_5)&:=&\li(u_1)+\li(u_2)-\li(\frac{1}{u_3})-\li(\frac{1}{u_4})+\li(u_5)-\frac{\pi^2}{6}+\log u_1\log u_2,\\
Y_0(u_1,\ldots,u_5)&:=&Y-\sum_{l=1}^5 \left(u_l\frac{\partial Y}{\partial u_l}\right)\log u_l,
\end{eqnarray*}
and
$$Z:=\text{(right side of (\ref{eq22}))--(left side of (\ref{eq22}))}.$$
Then by using (\ref{bra2}), (\ref{bra3}) and 
\begin{eqnarray*}
\lefteqn{\log^*u_1\log^*u_2-\log u_1\log u_2}\\
&&=\log^*u_1(\log^* u_2-\log u_2+\log u_2)-(\log^*u_1-\log^*u_1+\log u_1)\log u_2\\
&&\equiv\log u_1(\log^*u_2-\log u_2)+\log u_2(\log^*u_1-\log u_1)\modulos,
\end{eqnarray*}
we obtain
\begin{eqnarray}
\lefteqn{(X^*_0-X_0)-(Y_0^*-Y_0)}\nonumber\\
&&\equiv-\sum_{k=1}^4 t_k\frac{\partial X}{\partial t_k}(\log^* t_k-\log t_k)
+\sum_{l=1}^5 u_l\frac{\partial Y}{\partial u_l}(\log^* u_l-\log u_l)\modulos\label{bra4}
\end{eqnarray}
and
\begin{equation}\label{bra5}
u_l\frac{\partial Y^*}{\partial u_l}-u_l\frac{\partial Y}{\partial u_l}\equiv0\modu,
\end{equation}
for 
$l=1,\ldots,5$.

First, we will prove $Z$ is invariant modulo $4\pi^2$ for any choice of log-branch by showing
\begin{equation}\label{bra6}
(Z+X_0-Y_0)^*-(Z+X_0-Y_0)\equiv(X^*_0-X_0)-(Y_0^*-Y_0)\modulos.
\end{equation}
Note that
\begin{eqnarray}
\lefteqn{Z+X_0-Y_0=
  \left(-\log(1-t_1)+\log(1-\frac{1}{t_4})-\log u_1\right)(-u_1\frac{\partial Y}{\partial u_1})\label{bra7}}\\
  &&+\left(-\log(1-t_1)+\log(1-\frac{1}{t_2})-\log u_2\right)(-u_2\frac{\partial Y}{\partial u_2})\nonumber\\
  &&+\left(-\log(1-t_3)+\log(1-\frac{1}{t_2})-\log u_3\right)(-u_3\frac{\partial Y}{\partial u_3})\nonumber\\
  &&+\left(-\log(1-t_3)+\log(1-\frac{1}{t_4})-\log u_4\right)(-u_4\frac{\partial Y}{\partial u_4})\nonumber\\
  &&+\left(\log(1-t_1)-\log(1-\frac{1}{t_2})+\log(1-t_3)-\log(1-\frac{1}{t_4})-\log u_5\right)
  (-u_5\frac{\partial Y}{\partial u_5})\nonumber\\
  &&-\sum_{k=1}^4 t_k\frac{\partial X}{\partial t_k}\log t_k.\nonumber
\end{eqnarray}
From (\ref{eq20}), we know
\begin{eqnarray*}
\lefteqn{-\log(1-t_1)+\log(1-\frac{1}{t_4})-\log u_1
\equiv-\log(1-t_1)+\log(1-\frac{1}{t_2})-\log u_2}\\
&&\equiv-\log(1-t_3)+\log(1-\frac{1}{t_2})-\log u_3
\equiv-\log(1-t_3)+\log(1-\frac{1}{t_4})-\log u_4\\
&&\equiv\log(1-t_1)-\log(1-\frac{1}{t_2})+\log(1-t_3)-\log(1-\frac{1}{t_4})-\log u_5\equiv0\modu.
\end{eqnarray*}
Therefore, from (\ref{bra7}) and the above, we have
\begin{eqnarray}
\lefteqn{(Z+X_0-Y_0)^*\equiv
  \left(-\log^*(1-t_1)+\log^*(1-\frac{1}{t_4})-\log^* u_1\right)(-u_1\frac{\partial Y}{\partial u_1})\label{bra8}}\\
  &&+\left(-\log^*(1-t_1)+\log^*(1-\frac{1}{t_2})-\log^* u_2\right)(-u_2\frac{\partial Y}{\partial u_2})\nonumber\\
  &&+\left(-\log^*(1-t_3)+\log^*(1-\frac{1}{t_2})-\log^* u_3\right)(-u_3\frac{\partial Y}{\partial u_3})\nonumber\\
  &&+\left(-\log^*(1-t_3)+\log^*(1-\frac{1}{t_4})-\log^* u_4\right)(-u_4\frac{\partial Y}{\partial u_4})\nonumber\\
  &&+\left(\log^*(1-t_1)-\log^*(1-\frac{1}{t_2})+\log^*(1-t_3)-\log^*(1-\frac{1}{t_4})-\log^* u_5\right)
  (-u_5\frac{\partial Y}{\partial u_5})\nonumber\\
  &&-\sum_{k=1}^4 t_k\frac{\partial X^*}{\partial t_k}\log^* t_k~~~\modulos.\nonumber
\end{eqnarray}
Combining (\ref{bra7}) and (\ref{bra8}), we obtain
\begin{eqnarray}
\lefteqn{(Z+X_0-Y_0)^*-(Z+X_0-Y_0)
\equiv\sum_{l=1}^5 u_l\frac{\partial Y}{\partial u_l}(\log^* u_l-\log u_l)\label{bra9}}\\
&&+(\log^*(1-t_1)-\log(1-t_1))
(u_1\frac{\partial Y}{\partial u_1}+u_2\frac{\partial Y}{\partial u_2}-u_5\frac{\partial Y}{\partial u_5})\nonumber\\
&&+(\log^*(1-\frac{1}{t_2})-\log(1-\frac{1}{t_2}))
(-u_2\frac{\partial Y}{\partial u_2}-u_3\frac{\partial Y}{\partial u_3}+u_5\frac{\partial Y}{\partial u_5})\nonumber\\
&&+(\log^*(1-t_3)-\log(1-t_3))
(u_3\frac{\partial Y}{\partial u_3}+u_4\frac{\partial Y}{\partial u_4}-u_5\frac{\partial Y}{\partial u_5})\nonumber\\
&&+(\log^*(1-\frac{1}{t_4})-\log(1-\frac{1}{t_4}))
(-u_1\frac{\partial Y}{\partial u_1}-u_4\frac{\partial Y}{\partial u_4}+u_5\frac{\partial Y}{\partial u_5})\nonumber\\
&&-\sum_{k=1}^4 t_k\frac{\partial X^*}{\partial t_k}\log^* t_k+\sum_{k=1}^4 t_k\frac{\partial X}{\partial t_k}\log t_k~~~\modulos.\nonumber
\end{eqnarray}
From (\ref{eq20}), we know
\begin{eqnarray*}
\lefteqn{u_1\frac{\partial Y}{\partial u_1}+u_2\frac{\partial Y}{\partial u_2}-u_5\frac{\partial Y}{\partial u_5}}\\
&&=-\log(1-u_1)+\log u_2-\log(1-u_2)+\log u_2+\log(1-u_5)\equiv-\log^* t_1\modu,
\end{eqnarray*}
and
\begin{eqnarray*}
-u_2\frac{\partial Y}{\partial u_2}-u_3\frac{\partial Y}{\partial u_3}+u_5\frac{\partial Y}{\partial u_5}\equiv-\log^* t_2\modu,\\
u_3\frac{\partial Y}{\partial u_3}+u_4\frac{\partial Y}{\partial u_4}-u_5\frac{\partial Y}{\partial u_5}\equiv-\log^* t_3\modu,\\
-u_1\frac{\partial Y}{\partial u_1}-u_4\frac{\partial Y}{\partial u_4}+u_5\frac{\partial Y}{\partial u_5}\equiv-\log^* t_4\modu.
\end{eqnarray*}
Applying (\ref{bra4}) and (\ref{bra5}) to (\ref{bra9}), we obtain
\begin{eqnarray*}
\lefteqn{(Z+X_0-Y_0)^*-(Z+X_0-Y_0)\equiv\sum_{l=1}^5 u_l\frac{\partial Y}{\partial u_l}(\log^* u_l-\log u_l)}\\
&&-\sum_{k=1}^4(t_k\frac{\partial X^*}{\partial t_k}-t_k\frac{\partial X}{\partial t_k})(-\log^* t_k)
-\sum_{k=1}^4 t_k\frac{\partial X^*}{\partial t_k}\log^* t_k+\sum_{k=1}^4 t_k\frac{\partial X}{\partial t_k}\log t_k\\
&&\equiv (X_0^*-X_0)-(Y_0^*-Y_0)\modulos,
\end{eqnarray*}
which shows (\ref{bra6}).

Now we will prove $Z=0$ for certain log-branch. 
Direct calculation shows the imaginary part of (\ref{eq22}) becomes
\begin{eqnarray*}
\lefteqn{D(t_1)-D(\frac{1}{t_2})+D(t_3)-D(\frac{1}{t_4})-\log|t_1|\arg(1-t_1)}\\
&&~~~-\log|t_2|\arg(1-\frac{1}{t_2})-\log|t_3|\arg(1-t_3)-\log|t_4|\arg(1-\frac{1}{t_4})\\
&&=D(u_1)+D(u_2)-D(\frac{1}{u_3})-D(\frac{1}{u_4})+D(u_5)+\log|u_1|\arg u_2+\arg u_1\log |u_2|\\
&&~~~-\log|u_1|\arg(1-u_1)-\log|u_2|\arg(1-u_2)-\log|u_3|\arg(1-\frac{1}{u_3})\\
&&~~~-\log|u_4|\arg(1-\frac{1}{u_4})-\log|u_5|\arg(1-u_5)\\
&&~~~-\log|u_1|\arg u_2 -\log|u_2|\arg u_1\\
&&~~~+\log|u_1|\arg(1-u_1)+\log|u_2|\arg(1-u_2)+\log|u_3|\arg(1-\frac{1}{u_3})\\
&&~~~+\log|u_4|\arg(1-\frac{1}{u_4})+\log|u_5|\arg(1-u_5)\\
&&~~~-\arg(1-t_1)\log\left|u_2^{-1}u_1^{-1}(1-u_1)(1-u_2)(1-u_5)^{-1}\right|\\
&&~~~-\arg(1-\frac{1}{t_2})\log\left|u_1(1-u_2)^{-1}(1-\frac{1}{u_3})^{-1}(1-u_5)\right|\\
&&~~~-\arg(1-t_3)\log\left|(1-\frac{1}{u_3})(1-\frac{1}{u_4})(1-u_5)^{-1}\right|\\
&&~~~-\arg(1-\frac{1}{t_4})\log\left|u_2(1-u_1)^{-1}(1-\frac{1}{u_4})^{-1}(1-u_5)\right|.
\end{eqnarray*}
Using $u_5=\frac{1}{u_1u_3}=\frac{1}{u_2u_4}$, we obtain
\begin{eqnarray*}
u_2^{-1}u_1^{-1}(1-u_1)(1-u_2)(1-u_5)^{-1}=u_1''u_2''u_5'=t_1,\\
u_1(1-u_2)^{-1}(1-\frac{1}{u_3})^{-1}(1-u_5)=u_2'u_3'u_5''=t_2,\\
(1-\frac{1}{u_3})(1-\frac{1}{u_4})(1-u_5)^{-1}=u_3''u_4''u_5'=t_3,\\
u_2(1-u_1)^{-1}(1-\frac{1}{u_4})^{-1}(1-u_5)=u_4'u_1'u_5''=t_4.
\end{eqnarray*}
By applying these, we can verify the imaginary part of (\ref{eq22}) is equivalent to
$$D(t_1)-D(\frac{1}{t_2})+D(t_3)-D(\frac{1}{t_4})=D(u_1)+D(u_2)-D(\frac{1}{u_3})-D(\frac{1}{u_4})+D(u_5),$$
which is also equivalent to (\ref{eq21}). On the other hand, (\ref{eq22}) is an analytic function
on certain 3-dimensional open set, so the real part is some real constant.
After evaluating (\ref{eq22}) at $t_1=t_2=t_3=t_4=u_1=u_2=u_3=u_4=i$ and $u_5=-1$,\footnote{
Note that $\li(-1)=-\frac{\pi^2}{12}$.}
we find the real constant is zero. Therefore, we complete the proof of (\ref{eq22}).

The identity (\ref{eq23}) can be obtained from (\ref{eq22}) 
by substituting $t_1$, $t_2$, $t_3$, $t_4$ for $\frac{1}{t_2}$, $\frac{1}{t_3}$, $\frac{1}{t_4}$, $\frac{1}{t_1}$, respectively, 
and applying the following identity
\begin{eqnarray*}
\lefteqn{\log\frac{1}{u_2}\log\frac{1}{u_3}+\left(-\log(1-t_3)+\log(1-\frac{1}{t_2})\right)\log\frac{1}{u_2}}\\
&&~~+\left(-\log(1-t_1)+\log(1-\frac{1}{t_2})\right)\log\frac{1}{u_3}\\
&&=\left(-\log(1-t_3)+\log(1-\frac{1}{t_2})+\log\frac{1}{u_3}\right)\log\frac{1}{u_2}\\
&&~~+\left(-\log(1-t_1)+\log(1-\frac{1}{t_2})\right)\log\frac{1}{u_3}\\
&&\equiv-\left(-\log(1-t_3)+\log(1-\frac{1}{t_2})+\log\frac{1}{u_3}\right)\log{u_2}\\
&&~~+\left(-\log(1-t_1)+\log(1-\frac{1}{t_2})\right)\log\frac{1}{u_3}\\
&&=-\left(-\log(1-t_3)+\log(1-\frac{1}{t_2})\right)\log{u_2}\\
&&~~+\left(-\log(1-t_1)+\log(1-\frac{1}{t_2})-\log u_2\right)\log\frac{1}{u_3}\\
&&\equiv\log{u_2}\log{u_3}-\left(-\log(1-t_3)+\log(1-\frac{1}{t_2})\right)\log{u_2}\\
&&~~-\left(-\log(1-t_1)+\log(1-\frac{1}{t_2})\right)\log{u_3}\modulos.
\end{eqnarray*}

The identities (\ref{eq24}) and (\ref{eq25}) are directly obtained from (\ref{eq22}) and (\ref{eq23}).

Now we assume the edge AB is collapsed to a point (see Figure \ref{cola}).
Then we obtain the following relations.
\begin{eqnarray}\label{eq30}
      \left\{
      \begin{array}{ll}
     u_1=t_1' t_4'',\\
     u_2=t_1' t_2'',\\
      \end{array}\right.
      \left\{
      \begin{array}{ll}
     t_1=u_1'' u_2'',\\
     t_2=u_1 u_2',\\
     t_4=u_1' u_2,\\
     t_1t_2t_4=1.
      \end{array}\right.
\end{eqnarray}

\begin{figure}[h]
\centering
\includegraphics[scale=0.7]{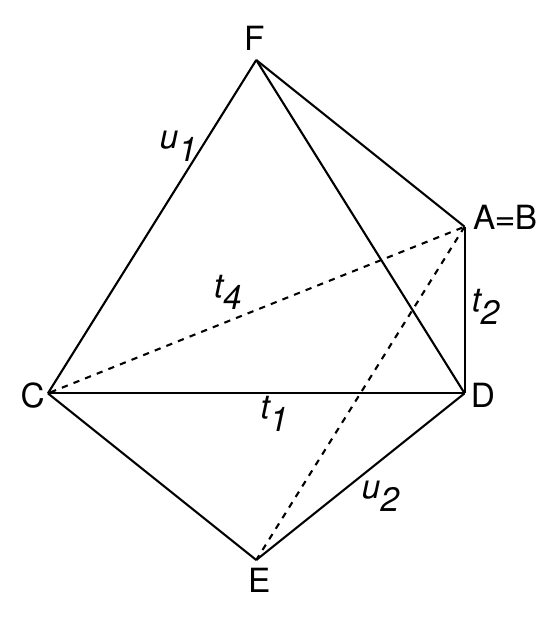}
  \caption{Assignment of shape parameters when the edge AB is collapsed}\label{cola}
\end{figure}

The identity (\ref{eq26}) and the relation (\ref{eq30}) can be obtained from (\ref{eq22}) and (\ref{eq20})
by sending $t_3\rightarrow 1$ and using the following property
$$\lim_{t\rightarrow 1}\left(\log t\,\log(1-t)\right)=0.$$
The identities (\ref{eq27}), (\ref{eq28}) and (\ref{eq29}) can be obtained from (\ref{eq23}), (\ref{eq24}) and (\ref{eq25})
by sending $t_4\rightarrow 1$, $t_1\rightarrow 1$ and $t_2\rightarrow 1$, respectively.

\end{proof}

\begin{proof}[Proof of Theorem \ref{thm}]
Now we prove the theorem by calculating the potential functions on each crossing $n$.
First, consider the case in which no edge of the octahedron on the positive crossing $n$ is collapsed.
Let the variables assigned to the contributing sides be $z_a,\ldots,z_d$ as in Figure \ref{label} and
let $t_1=\frac{z_b}{z_a}$, $t_2=\frac{z_c}{z_b}$, $t_3=\frac{z_d}{z_c}$, $t_4=\frac{z_a}{z_d}$ as in Figure \ref{pic9}(a).
Then the Yokota potential function of the crossing becomes
$$X(z_a,\ldots,z_d):=\li(t_1)-\li(\frac{1}{t_2})+\li(t_3)-\li(\frac{1}{t_4})$$
and
\begin{eqnarray}
\lefteqn{X_0(z_a,\ldots,z_d)=\li(t_1)-\li(\frac{1}{t_2})+\li(t_3)-\li(\frac{1}{t_4})}\label{eq31}\\
&&+\left(-\log(1-t_1)+\log(1-\frac{1}{t_4})\right)\log z_a-\left(-\log(1-t_1)+\log(1-\frac{1}{t_2})\right)\log z_b\nonumber\\
&&+\left(-\log(1-t_3)+\log(1-\frac{1}{t_2})\right)\log z_c-\left(-\log(1-t_3)+\log(1-\frac{1}{t_4})\right)\log z_d.\nonumber
\end{eqnarray}
Likewise, let the variables assigned to the regions be $w_j,\ldots,w_m$ as in Figure \ref{label} and
let $u_1=\frac{w_m}{w_j}$, $u_2=\frac{w_k}{w_j}$, $u_3=\frac{w_k}{w_l}$, $u_4=\frac{w_m}{w_l}$,
$u_5=\frac{w_j w_l}{w_k w_m}$ as in Figure \ref{pic9}(a).
Then the potential function of the colored Jones polynomial of the crossing becomes
$P_f$, which was defined in Lemma \ref{lem1} for $f=1,\ldots,4$, and
\begin{eqnarray}
\lefteqn{P_{10}=\li(u_1)+\li(u_2)-\li(\frac{1}{u_3})-\li(\frac{1}{u_4})+\li(u_5)-\frac{\pi^2}{6}+\log u_1\log u_2}\label{eq32}\\
&&+\left(-\log(1-u_1)-\log(1-u_2)+\log(1-u_5)+\log u_1 +\log u_2\right)\log w_j \nonumber\\
&&+\left(\log(1-u_2)+\log(1-\frac{1}{u_3})-\log(1-u_5)-\log u_1\right)\log w_k \nonumber\\
&&+\left(-\log(1-\frac{1}{u_3})-\log(1-\frac{1}{u_4})+\log(1-u_5)\right)\log w_l \nonumber\\
&&+\left(\log(1-u_1)+\log(1-\frac{1}{u_4})-\log(1-u_5)-\log u_2\right)\log w_m. \nonumber
\end{eqnarray}

We define {\it the remaining term} $Z_n$  by the difference of two potential functions
$V_0-W_0$ of the crossing $n$. In this case, $Z_n=X_0-P_{10}$.

Assume $z_a,\ldots,z_d,w_j,\ldots,w_m$ satisfy the assumption of Lemma \ref{lem5}.\footnote{
Any essential solution $(z_a,\ldots,z_d)$ of $\mathcal{H}_1$ and the corresponding essential solution
$(w_j,\ldots,w_m)$ of $\mathcal{H}_2$ satisfy this assumption.} Let
\begin{eqnarray*}
      &&\left\{
      \begin{array}{ll}
     U_1:=-\log(1-t_1)+\log(1-\frac{1}{t_4}),\\
     U_2:=-\log(1-t_1)+\log(1-\frac{1}{t_2}),\\
     U_3:=-\log(1-t_3)+\log(1-\frac{1}{t_2}),\\
     U_4:=-\log(1-t_3)+\log(1-\frac{1}{t_4}),
      \end{array}\right.\\
      &&\left\{
      \begin{array}{ll}
     T_1:=\log(1-u_1)+\log(1-u_2)-\log(1-u_5)-\log u_1 -\log u_2,\\
     T_2:=-\log(1-u_2)-\log(1-\frac{1}{u_3})+\log(1-u_5)+\log u_1,\\
     T_3:=\log(1-\frac{1}{u_3})+\log(1-\frac{1}{u_4})-\log(1-u_5),\\
     T_4:=-\log(1-u_1)-\log(1-\frac{1}{u_4})+\log(1-u_5)+\log u_2.
      \end{array}\right.
\end{eqnarray*}
Then by (\ref{eq20}),
\begin{eqnarray*}
\left\{
      \begin{array}{ll}
    U_1\equiv\log u_1\equiv\log w_m-\log w_j\modu,\\
    U_2\equiv\log u_2\equiv\log w_k-\log w_j\modu,\\
    U_3\equiv\log u_3\equiv\log w_k-\log w_l\modu,\\
    U_4\equiv\log u_4\equiv\log w_m-\log w_l\modu,\\
      \end{array}\right.
\left\{       \begin{array}{ll}
    T_1\equiv\log t_1\equiv\log z_b-\log z_a\modu,\\
    T_2\equiv\log t_2\equiv\log z_c-\log z_b\modu,\\
    T_3\equiv\log t_3\equiv\log z_d-\log z_c\modu,\\
    T_4\equiv\log t_4\equiv\log z_a-\log z_d\modu,
      \end{array}\right.
\end{eqnarray*}
and $ U_1+U_3=U_2+U_4$, $T_1+T_2+T_3+T_4=0$.
Applying these and (\ref{eq22}) to (\ref{eq31}) and (\ref{eq32}),
we obtain the remaining term $Z_n$ of the crossing $n$ as follows.
\begin{eqnarray*}
\lefteqn{Z_n=X_0-P_{10}\equiv U_1\log z_a-U_2\log z_b+U_3 \log z_c -U_4\log z_d}\\
&&+T_1\log w_j +T_2\log w_k+T_3\log w_l+T_4\log w_m-U_1\log u_2-U_2\log u_1\\
&&+U_1\log(1-u_1)+U_2\log(1-u_2)+U_3\log(1-\frac{1}{u_3})+U_4\log(1-\frac{1}{u_4})-(U_1+U_3)\log(1-u_5)\\
\lefteqn{~~=T_1\log w_j +T_2\log w_k+T_3\log w_l+T_4\log w_m}\\
&&+U_1\left(\log z_a-\log z_d+\log(1-u_1)+\log(1-\frac{1}{u_4})-\log(1-u_5)-\log u_2\right)\\
&&+U_2\left(-\log z_b+\log z_d+\log(1-u_2)-\log(1-\frac{1}{u_4})-\log u_1\right)\\
&&+U_3\left(\log z_c-\log z_d +\log(1-\frac{1}{u_3})+\log(1-\frac{1}{u_4})-\log(1-u_5)\right)\\
\lefteqn{~~=T_2(\log w_k-\log w_j)+T_3(\log w_l-\log w_j)+T_4(\log w_m-\log w_j)}\\
&&+U_1\left(\log z_a-\log z_d-T_4\right)+U_2\left(-\log z_b+\log z_d-T_2-T_3\right)
+U_3\left(\log z_c-\log z_d +T_3\right)\\
\lefteqn{~~\equiv T_2(\log w_k-\log w_j)+T_3(\log w_l-\log w_j)+T_4(\log w_m-\log w_j)}\\
&&+(\log w_m-\log w_j)\left(\log z_a-\log z_d-T_4\right)
  +(\log w_k-\log w_j)\left(-\log z_b+\log z_d-T_2-T_3\right)\\
&&+(\log w_k-\log w_l)\left(\log z_c-\log z_d +T_3\right)\hspace{6cm}\modulos\\
\lefteqn{~~=-(\log w_j-\log w_m)\log z_a-(\log w_k-\log w_j)\log z_b+(\log w_k-\log w_l)\log z_c}\\
&&+(\log w_l-\log w_m)\log z_d.
\end{eqnarray*}
By the same method, we can prove that the remaining term of the negative crossing in Figure \ref{label} is
the same as that of the positive crossing.

Now we consider the case in which only one horizontal edge is collapsed in an octahedron on a positive crossing $n$.
Let the region assigned to $r_l$ be the unbounded region and $z_c=z_d=1$ in Figure \ref{label}.
Also let $t_1=\frac{z_b}{z_a}$, $t_2=\frac{1}{z_b}$, $t_4={z_a}$ and
$u_1=\frac{w_m}{w_j}$, $u_2=\frac{w_k}{w_j}$.
Then the Yokota potential function of the crossing becomes
$$X(z_a,z_b):=\li(t_1)-\li(\frac{1}{t_2})-\li(\frac{1}{t_4})+\frac{\pi^2}{6}$$
and
\begin{eqnarray}
\lefteqn{X_0(z_a,z_b)=\li(t_1)-\li(\frac{1}{t_2})-\li(\frac{1}{t_4})+\frac{\pi^2}{6}}\label{eq33}\\
&&+\left(-\log(1-t_1)+\log(1-\frac{1}{t_4})\right)\log z_a-\left(-\log(1-t_1)+\log(1-\frac{1}{t_2})\right)\log z_b.\nonumber
\end{eqnarray}
The potential function of the colored Jones polynomial of the crossing becomes
$$Y(w_j,w_k,w_m):=P_1(w_j,w_k,0,w_m)=\li(u_1)+\li(u_2)-\frac{\pi^2}{6}+\log u_1\log u_2$$
and
\begin{eqnarray}
\lefteqn{Y_0(w_j,w_k,w_m)=\li(u_1)+\li(u_2)-\frac{\pi^2}{6}+\log u_1\log u_2}\label{eq34}\\
&&+(-\log(1-u_1)-\log(1-u_2)+\log u_1+\log u_2)\log w_j\nonumber\\
&&+(\log(1-u_2)-\log u_1)\log w_k+(\log(1-u_1)-\log u_2)\log w_m.\nonumber
\end{eqnarray}
In this case, the remaining term is $Z_n=X_0-Y_0$. Let
\begin{eqnarray*}
      &&\left\{
      \begin{array}{ll}
     U_1:=-\log(1-t_1)+\log(1-\frac{1}{t_4}),\\
     U_2:=-\log(1-t_1)+\log(1-\frac{1}{t_2}),
      \end{array}\right.\\
      &&\left\{
      \begin{array}{ll}
     T_1:=\log(1-u_1)+\log(1-u_2)-\log u_1 -\log u_2,\\
     T_2:=-\log(1-u_2)+\log u_1,\\
     T_4:=-\log(1-u_1)+\log u_2.
      \end{array}\right.
\end{eqnarray*}
Then by (\ref{eq30}),
\begin{eqnarray*}
\left\{
      \begin{array}{ll}
    U_1\equiv\log u_1\equiv\log w_m-\log w_j\modu,\\
    U_2\equiv\log u_2\equiv\log w_k-\log w_j\modu,\\
      \end{array}\right.
\left\{       \begin{array}{ll}
    T_1\equiv\log t_1\equiv\log z_b-\log z_a\modu,\\
    T_2\equiv\log t_2\equiv-\log z_b\modu,\\
    T_4\equiv\log t_4\equiv\log z_a\modu,
      \end{array}\right.
\end{eqnarray*}
and $T_1+T_2+T_4=0$. Applying these and (\ref{eq26}) to (\ref{eq33}) and (\ref{eq34}),
we obtain the remaining term $Z_n$ as follows.
\begin{eqnarray*}
Z_n&=&X_0-Y_0\equiv U_1\log z_a-U_2\log z_b+T_1\log w_j+T_2\log w_k+T_4\log w_m-U_1T_4-U_2T_2\\
&=&U_1\log z_a-U_2\log z_b+T_2(\log w_k-\log w_j-U_2)+T_4(\log w_m-\log w_j-U_1)\\
&\equiv&U_1\log z_a-U_2\log z_b\\
&&-\log z_b(\log w_k-\log w_j-U_2)+\log z_a(\log w_m-\log w_j-U_1)\modulos\\
&=&-(\log w_j-\log w_m)\log z_a-(\log w_k-\log w_j)\log z_b.
\end{eqnarray*}
By the same method, we can prove the remaining term of the negative crossing in this case is
the same as that of the positive crossing.
On the other hand, the remaining term becomes
$$Z_n=-(\log w_k-\log w_j)\log z_b+(\log w_k-\log w_l)\log z_c$$
when the region assigned to $w_m$ is the unbounded region,
$$Z_n=(\log w_k-\log w_l)\log z_c+(\log w_l-\log w_m)\log z_d$$
when the region assigned to $w_j$ is the unbounded region, and
$$Z_n=-(\log w_j-\log w_m)\log z_a+(\log w_l-\log w_m)\log z_d$$
when the region assigned to $w_k$ is the unbounded region.

Now we consider the case when the crossing point $n$ is the endpoint of $I$ or $J$.
There are four cases as in Figure \ref{pic19}.
We only prove the case of Figure \ref{pic19}(a) because the others can be proved by the same method.

\begin{figure}[h]
\centering
\setlength{\unitlength}{0.4cm}
\subfigure[]
  {\begin{picture}(9,7)\thicklines
    \put(6,5){\vector(-1,-1){4}}
    \dashline{0.5}(4.2,2.8)(6,1)
    \put(5.2,1.8){\vector(1,-1){0.8}}
    \put(2,5){\line(1,-1){1.8}}
    \put(3.5,1){$w_j$}
    \put(3.5,4.5){$w_l$}
    \put(1,2.5){$w_m$}
    \put(1,5.3){$z_d$}
    \put(6,5.3){$z_c$}
    \put(1,0.2){$z_a$}
    \put(6,0.2){1}
  \end{picture}}
\subfigure[]
  {\begin{picture}(9,7)\thicklines
    \put(2,5){\vector(1,-1){4}}
    \put(6,5){\line(-1,-1){1.8}}
    \dashline{0.5}(2,1)(3.8,2.8)
    \put(2.8,1.8){\vector(-1,-1){0.8}}
    \put(3.5,1){$w_j$}
    \put(6,2.5){$w_k$}
    \put(3.5,4.5){$w_l$}
    \put(1,5.3){$z_d$}
    \put(6,5.3){$z_c$}
    \put(1.2,0.2){1}
    \put(6,0.2){$z_b$}
  \end{picture}}
\subfigure[]
  {\begin{picture}(9,7)\thicklines
   \put(4.2,2.8){\vector(1,-1){1.8}}
   \put(2,5){\line(1,-1){1.8}}
   \dashline{0.5}(6,5)(4,3)
   \put(4,3){\vector(-1,-1){2}}
    \put(3.5,1){$w_j$}
    \put(6,2.5){$w_k$}
    \put(1,2.5){$w_m$}
    \put(6,5.3){1}
    \put(1,0.2){$z_a$}
    \put(6,0.2){$z_b$}
    \put(1,5.3){$z_d$}
  \end{picture}}
\subfigure[]
  {\begin{picture}(9,7)\thicklines
   \put(3.8,2.8){\vector(-1,-1){1.8}}
   \put(4.2,3.2){\line(1,1){1.8}}
   \dashline{0.5}(2,5)(4,3)
   \put(4,3){\vector(1,-1){2}}
    \put(3.5,1){$w_j$}
    \put(6,2.5){$w_k$}
    \put(3.5,4.5){$w_l$}
    \put(6,5.3){$z_c$}
    \put(1,0.2){$z_a$}
    \put(6,0.2){$z_b$}
    \put(1.2,5.3){1}
  \end{picture}}\caption{Four cases of the endpoint of $I$ or $J$}\label{pic19}
\end{figure}
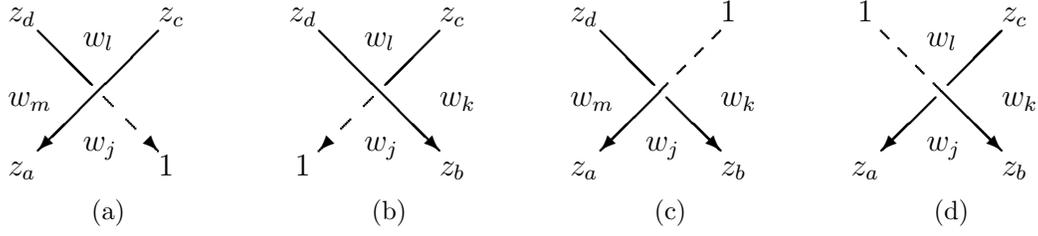

First, we assume all three regions in Figure \ref{pic19}(a) are bounded.
Then in Figure \ref{pic9}(a), the edge ${\rm B}_n{\rm E}_n$ is collapsed to a point and
$\frac{z_d}{z_c}$, $\frac{z_a}{z_d}$, $\frac{w_m}{w_j}$, $\frac{w_j}{w_l}$ are assigned to
the edges ${\rm C}_n{\rm D}_n$, ${\rm D}_n{\rm A}_n$, ${\rm A}_n{\rm F}_n$, ${\rm C}_n{\rm F}_n$, respectively.
Also, we obtain
\begin{equation}\label{eq35}
\frac{z_d}{z_c}=\left(\frac{w_j}{w_l}\right)''=1-\frac{w_l}{w_j} ~\text{ and }~
\frac{w_m}{w_j}=\left(\frac{z_a}{z_d}\right)''=1-\frac{z_d}{z_a}.
\end{equation}
Applying (\ref{eq35}) to Yokota potential function
$X(z_a,z_c,z_d):=\li(\frac{z_d}{z_c})-\li(\frac{z_d}{z_a})$, we obtain
\begin{eqnarray*}
\lefteqn{X_0=\li(\frac{z_d}{z_c})-\li(\frac{z_d}{z_a})+\log(1-\frac{z_d}{z_a})\log z_a-\log(1-\frac{z_d}{z_c})\log z_c}\\
&&~~-\left(-\log(1-\frac{z_d}{z_c})+\log(1-\frac{z_d}{z_a})\right)\log z_d\\
&&=\li(\frac{z_d}{z_c})-\li(1-\frac{w_m}{w_j})+\log\frac{w_m}{w_j}(\log z_a-\log z_d)
+\log\frac{w_l}{w_j}(\log z_d-\log z_c).\\
\end{eqnarray*}
Also, applying (\ref{eq35}) to the potential function of the colored Jones polynomial
$Y(w_j,w_l,w_m):=P_1(w_j,w_j,w_l,w_m)=\li(\frac{w_m}{w_j})-\li(\frac{w_l}{w_j})$, we obtain
\begin{eqnarray*}
\lefteqn{Y_0=\li(\frac{w_m}{w_j})-\li(\frac{w_l}{w_j})
-\left(\log(1-\frac{w_m}{w_j})-\log(1-\frac{w_l}{w_j})\right)\log w_j}\\
&&-\log(1-\frac{w_l}{w_j})\log w_l+\log(1-\frac{w_m}{w_j})\log w_m\\
&&=\li(\frac{w_m}{w_j})-\li(1-\frac{z_d}{z_c})-\log\frac{z_d}{z_a}(\log w_j-\log w_m)
-\log\frac{z_d}{z_c}(\log w_l-\log w_j).
\end{eqnarray*}
Using the well-known identity $\li(z)+\li(1-z)=\frac{\pi^2}{6}-\log z\log(1-z)$ for $z\in\mathbb{C}-\{0,1\}$ from \cite{Lewin1},
we obtain the remaining term
\begin{eqnarray*}
Z_n&=&X_0-Y_0\equiv-\log\frac{z_d}{z_c}\log\frac{w_l}{w_j}+\log\frac{w_m}{w_j}\log\frac{z_d}{z_a}\\
&&+\log\frac{w_m}{w_j}(\log z_a-\log z_d)+\log\frac{w_l}{w_j}(\log z_d-\log z_c)\\
&&+\log\frac{z_d}{z_a}(\log w_j-\log w_m)+\log\frac{z_d}{z_c}(\log w_l-\log w_j)\\
&=&\log\frac{w_l}{w_j}(-\log\frac{z_d}{z_c}+\log z_d-\log z_c)
+\log\frac{w_m}{w_j}(\log\frac{z_d}{z_a}+\log z_a-\log z_d)\\
&&+\log\frac{z_d}{z_a}(\log w_j-\log w_m)+\log\frac{z_d}{z_c}(\log w_l-\log w_j)\\
&\equiv&(\log{w_l}-\log{w_j})\left(-\log\frac{z_d}{z_c}+\log z_d-\log z_c\right)\\
&&+(\log{w_m}-\log{w_j})\left(\log\frac{z_d}{z_a}+\log z_a-\log z_d\right)\\
&&+\log\frac{z_d}{z_a}(\log w_j-\log w_m)+\log\frac{z_d}{z_c}(\log w_l-\log w_j)\modulos\\
&=&-(\log{w_j}-\log{w_m})\log z_a+(\log{w_j}-\log{w_l})\log z_c+(\log w_l-\log w_m)\log z_d.
\end{eqnarray*}

Finally, we consider the case when the region assigned with $w_l$ in Figure \ref{pic19}(a) is unbounded.
Then the edges ${\rm B}_n{\rm E}_n$ and ${\rm C}_n{\rm D}_n$ are collapsed to points. Furthermore,
$z_c=z_d=1$ and $w_l=0$, and $z_a$, $\frac{w_m}{w_j}$ are assigned
to the edges ${\rm D}_n{\rm A}_n$, ${\rm A}_n{\rm F}_n$ in Figure \ref{pic9}(a), respectively.
Applying
\begin{equation*}
\frac{w_m}{w_j}=z_a''=1-\frac{1}{z_a}
\end{equation*}
to Yokota potential function
$X(z_a):=-\li(\frac{1}{z_a})+\frac{\pi^2}{6}$, we obtain
$$X_0=-\li(\frac{1}{z_a})+\frac{\pi^2}{6}+\log(1-\frac{1}{z_a})\log z_a
=-\li(\frac{1}{z_a})+\frac{\pi^2}{6}+\log\frac{w_m}{w_j}\log z_a,$$
and to the potential function of the colored Jones polynomial
$Y(w_j,w_m):=P_1(w_j,w_j,0,w_m)=\li(\frac{w_m}{w_j})$, we obtain
$$Y_0=\li(\frac{w_m}{w_j})+\log(1-\frac{w_m}{w_j})(\log w_m-\log w_j)
=\li(1-\frac{1}{z_a})+\log\frac{1}{z_a}(\log w_m-\log w_j).$$
Therefore, we obtain the remaining term
\begin{eqnarray*}
Z_n&:=&X_0-Y_0\equiv\log\frac{1}{z_a}\log\frac{w_m}{w_j}+\log\frac{w_m}{w_j}\log z_a
-\log\frac{1}{z_a}(\log w_m-\log w_j)\\
&=&\log\frac{1}{z_a}(\log\frac{w_m}{w_j}-\log w_m+\log w_j)+\log\frac{w_m}{w_j}\log z_a\\
&\equiv&-\log{z_a}(\log\frac{w_m}{w_j}-\log w_m+\log w_j)+\log\frac{w_m}{w_j}\log z_a\modulos\\
&=&-(\log w_j-\log w_m)\log z_a.
\end{eqnarray*}
Likewise, we can show the remaining term becomes
$$Z_n=(\log w_l-\log w_m)\log z_d$$
when the region assigned to $w_m$ in Figure \ref{pic19}(a) is unbounded.
The remaining three cases in Figure \ref{pic19} can be obtained by the same method.

We complete the proof by proving 
$$\sum_{n~:~\text{crossings of }G}Z_n=0.$$
Note that we defined a contributing side of $G$ in Section \ref{ch31}. 
Assume the side assigned by $z_a$ in Figure \ref{final} is a contributing side of $G$.
(This means that $z_a\neq 1$.)

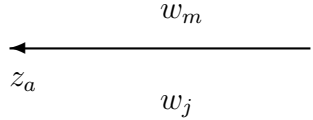
\begin{figure}[h]
\centering
\begin{picture}(4,2)\thicklines
   \put(4,1){\vector(-1,0){4}}
   \put(2,1.4){$w_m$}
   \put(2,0.2){$w_j$}
   \put(0,0.5){$z_a$}
  \end{picture}
  \caption{Contributing side assigned by $z_a$}\label{final}
\end{figure}

If the side goes out of the crossing point $n_1$,
then the coefficient of $\log z_a$ in $Z_{n_1}$ is $-(\log w_j-\log w_l)$, and if
the side goes into the crossing point $n_2$,
then the coefficient of $\log z_a$ in $Z_{n_2}$ is $(\log w_j-\log w_l)$. They are cancelled by each other,
and this happens for all the contributing sides.

\end{proof}

\appendix
\section{Appendix}
\subsection{Formal substitution of the colored Jones polynomial and the potential function}\label{app1}

In this Appendix, we induce the potential function $W(w_1,\ldots,w_m)$ defined in Section \ref{ch32}
from the formal substitution (\ref{formalsubsti}) of the colored Jones polynomial.

The colored Jones polynomial is determined
by the R-matrix and the local maxima/minima (see \cite{Murakami00a} for reference).
However, as seen in (\ref{formalsubsti}), the local maxima/minima do not have an effect on
the formal substitution. So we only consider the R-matrix of the colored Jones polynomial: 

\begin{eqnarray*}
R_{l,m}^{j,k}&=&\delta_{m,j-h}\delta_{l,k+h}\frac{(q^{-1})_j(q^{-1})_k^{-1}}{(q^{-1})_h(q^{-1})_l^{-1}(q^{-1})_m}(-1)^{k+m+1}q^{-km-(k+m+1)/2},\\
(R^{-1})_{l,m}^{j,k}&=&\delta_{m,j+h}\delta_{l,k-h}\frac{(q)_j^{-1}(q)_k}{(q)_h(q)_l(q)_m^{-1}}(-1)^{j+l+1}q^{jl+(j+l+1)/2},
\end{eqnarray*}
where $j,k,l,m,h\in\{0,1,\ldots,N-1\}$ and 
$\delta_{j,k}$
is the Kronecker's delta. If $R_{l,m}^{j,k}\neq 0$, then $h$ is uniquely determined by the formula $h=j-m=l-k$,
and if $(R^{-1})_{l,m}^{j,k}\neq 0$, then $h=m-j=k-l$. 

Note that this R-matrix is the inverse of the one in \cite{Murakami00a}.
This implies the colored Jones polynomial of a knot $K$ here 
is the one of the mirror image $\overline{K}$ in \cite{Murakami00a}.
This choice is natural to \cite{YokotaPre} and Theorem \ref{thm}.

Let $K$ be the hyperbolic knot with a fixed diagram and 
$G$ be the diagram defined in Section \ref{ch21} with the orientation from $J$ to $I$. 
We assign 0 to one bounded region of $G$, then
assign variables $r_1,\ldots,r_m\in\{0,1,\ldots,N-1\}$ to the remaining bounded regions of $G$ and $r_{m+1}\in\{0,1,\ldots,N-1\}$ to the unbounded region.
We assign variables to each side according to the signed sum of variables of adjacent regions with orientations modulo $N$
(see Figure \ref{pic8} for an example).

\begin{figure}[h]
\centering
  \includegraphics[scale=0.5]{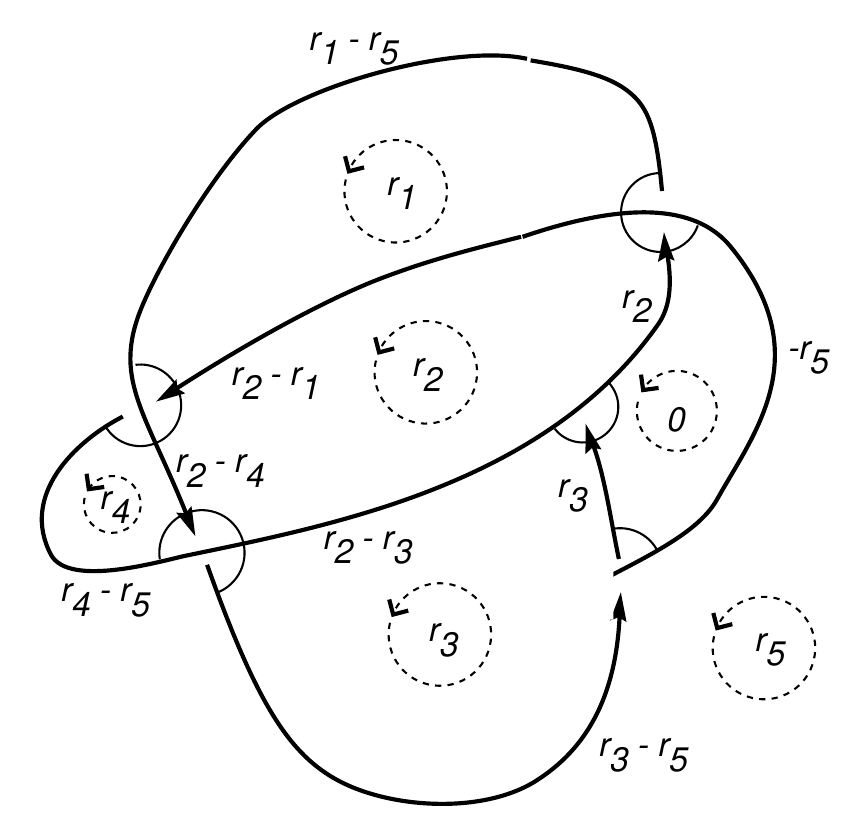}
  \caption{Assigning variables to each region and side}\label{pic8}
\end{figure}

For each non-trivalent vertex of $G$, we assign the R-matrix to the positive crossing and the inverse to the negative crossing. Then we apply the formal substitution (\ref{formalsubsti}) to each R-matrix
and substitute $q^{r_n}$ to $w_n$ as below.
In the substitution process, if $r_n=0$, then we put $w_n=1$.
Note that we apply the same R-matrix or its inverse in different forms according to
the position of the collapsed horizontal edge. If none of the horizontal edges are collapsed in the octahedron, then we
choose any formal substitution among the four possibilities.
For positive crossings :


  {\setlength{\unitlength}{0.4cm}
  \begin{picture}(18,9.5)\thicklines
    \put(6,7){\vector(-1,-1){4}}
    \put(2,7){\line(1,-1){1.8}}
    \put(4.2,4.8){\vector(1,-1){1.8}}
    \put(3.5,3){$r_j$}
    \put(6,4.5){$r_k$}
    \put(3.5,6.5){$r_l$}
    \put(1,4.5){$r_m$}
    \put(0,7.3){$r_l-r_m$}
    \put(5,7.3){$r_k-r_l$}
    \put(0,2.2){$r_j-r_m$}
    \put(5,2.2){$r_k-r_j$}
    \put(4,5){\arc(-0.6,0.6){270}}
    \put(9,5){: $\displaystyle\frac{(q)_{r_l-r_m}(q^{-1})^{-1}_{r_k-r_l}}{(q)_{r_j+r_l-r_k-r_m}(q^{-1})^{-1}_{r_j-r_m}(q)_{r_k-r_j}}(-1)^{r_l+r_j+1}$}
    \put(10,2.5){$\times q^{(r_m-r_j)(r_k-r_j)-(2r_k-r_l-r_j+1)/2}$}
    \put(0,0){$\sim\exp\left\{\frac{N}{2\pi i}\left(-\li(\frac{w_l}{w_m})-\li(\frac{w_l}{w_k})+\li(\frac{w_jw_l}{w_kw_m})+\li(\frac{w_m}{w_j})+\li(\frac{w_k}{w_j})-\frac{\pi^2}{6}+\log\frac{w_m}{w_j}\log\frac{w_k}{w_j}\right)\right\}$,}
  \end{picture}}

  {\setlength{\unitlength}{0.4cm}
  \begin{picture}(18,9.5)\thicklines
    \put(6,7){\vector(-1,-1){4}}
    \put(2,7){\line(1,-1){1.8}}
    \put(4.2,4.8){\vector(1,-1){1.8}}
    \put(3.5,3){$r_j$}
    \put(6,4.5){$r_k$}
    \put(3.5,6.5){$r_l$}
    \put(1,4.5){$r_m$}
    \put(0,7.3){$r_l-r_m$}
    \put(5,7.3){$r_k-r_l$}
    \put(0,2.2){$r_j-r_m$}
    \put(5,2.2){$r_k-r_j$}
    \put(4,5){\arc(-0.6,-0.6){270}}
    \put(9,5){: $\displaystyle\frac{(q^{-1})_{r_l-r_m}(q^{-1})^{-1}_{r_k-r_l}}{(q^{-1})_{r_j+r_l-r_k-r_m}(q^{-1})^{-1}_{r_j-r_m}(q^{-1})_{r_k-r_j}}(-1)^{r_l+r_j+1}$}
    \put(10,2.5){$\times q^{-(r_k-r_l)(r_k-r_j)-(2r_k-r_l-r_j+1)/2}$}
    \put(0,0){$\sim\exp\left\{\frac{N}{2\pi i}\left(\li(\frac{w_m}{w_l})-\li(\frac{w_l}{w_k})-\li(\frac{w_kw_m}{w_jw_l})+\li(\frac{w_m}{w_j})-\li(\frac{w_j}{w_k})+\frac{\pi^2}{6}-\log\frac{w_k}{w_l}\log\frac{w_k}{w_j}\right)\right\}$,}
  \end{picture}}

  {\setlength{\unitlength}{0.4cm}
  \begin{picture}(18,9.5)\thicklines
    \put(6,7){\vector(-1,-1){4}}
    \put(2,7){\line(1,-1){1.8}}
    \put(4.2,4.8){\vector(1,-1){1.8}}
    \put(3.5,3){$r_j$}
    \put(6,4.5){$r_k$}
    \put(3.5,6.5){$r_l$}
    \put(1,4.5){$r_m$}
    \put(0,7.3){$r_l-r_m$}
    \put(5,7.3){$r_k-r_l$}
    \put(0,2.2){$r_j-r_m$}
    \put(5,2.2){$r_k-r_j$}
    \put(4,5){\arc(0.6,-0.6){270}}
    \put(9,5){: $\displaystyle\frac{(q^{-1})_{r_l-r_m}(q)^{-1}_{r_k-r_l}}{(q)_{r_j+r_l-r_k-r_m}(q)^{-1}_{r_j-r_m}(q^{-1})_{r_k-r_j}}(-1)^{r_l+r_j+1}$}
    \put(10,2.5){$\times q^{(r_m-r_l)(r_k-r_l)-(2r_k-r_l-r_j+1)/2}$}
    \put(0,0){$\sim\exp\left\{\frac{N}{2\pi i}\left(\li(\frac{w_m}{w_l})+\li(\frac{w_k}{w_l})+\li(\frac{w_jw_l}{w_kw_m})-\li(\frac{w_j}{w_m})-\li(\frac{w_j}{w_k})-\frac{\pi^2}{6}+\log\frac{w_m}{w_l}\log\frac{w_k}{w_l}\right)\right\}$,}
  \end{picture}}

  {\setlength{\unitlength}{0.4cm}
  \begin{picture}(18,9.5)\thicklines
    \put(6,7){\vector(-1,-1){4}}
    \put(2,7){\line(1,-1){1.8}}
    \put(4.2,4.8){\vector(1,-1){1.8}}
    \put(3.5,3){$r_j$}
    \put(6,4.5){$r_k$}
    \put(3.5,6.5){$r_l$}
    \put(1,4.5){$r_m$}
    \put(0,7.3){$r_l-r_m$}
    \put(5,7.3){$r_k-r_l$}
    \put(0,2.2){$r_j-r_m$}
    \put(5,2.2){$r_k-r_j$}
    \put(4,5){\arc(0.6,0.6){270}}
    \put(9,5){: $\displaystyle\frac{(q)_{r_l-r_m}(q)^{-1}_{r_k-r_l}}{(q^{-1})_{r_j+r_l-r_k-r_m}(q)^{-1}_{r_j-r_m}(q)_{r_k-r_j}}(-1)^{r_l+r_j+1}$}
    \put(10,2.5){$\times q^{-(r_m-r_l)(r_m-r_j)-(r_l+r_j-2r_m+1)/2}$}
    \put(0,0){$\sim\exp\left\{\frac{N}{2\pi i}\left(-\li(\frac{w_l}{w_m})+\li(\frac{w_k}{w_l})-\li(\frac{w_kw_m}{w_jw_l})-\li(\frac{w_j}{w_m})+\li(\frac{w_k}{w_j})+\frac{\pi^2}{6}-\log\frac{w_m}{w_l}\log\frac{w_m}{w_j}\right)\right\}$.}
  \end{picture}}

\vspace{0.5cm}

For negative crossings :

  {\setlength{\unitlength}{0.4cm}
  \begin{picture}(18,9.5)\thicklines
    \put(2,7){\vector(1,-1){4}}
   \put(6,7){\line(-1,-1){1.8}}
   \put(3.8,4.8){\vector(-1,-1){1.8}}
    \put(3.5,3){$r_j$}
    \put(6,4.5){$r_k$}
    \put(3.5,6.5){$r_l$}
    \put(1,4.5){$r_m$}
    \put(0,7.3){$r_l-r_m$}
    \put(5,7.3){$r_k-r_l$}
    \put(0,2.2){$r_j-r_m$}
    \put(5,2.2){$r_k-r_j$}
    \put(4,5){\arc(-0.6,0.6){270}}
    \put(9,5){: $\displaystyle\frac{(q)^{-1}_{r_l-r_m}(q^{-1})_{r_k-r_l}}{(q^{-1})_{r_k+r_m-r_j-r_l}(q^{-1})_{r_j-r_m}(q)^{-1}_{r_k-r_j}}(-1)^{r_l+r_j+1}$}
    \put(10,2.5){$\times q^{-(r_j-r_m)(r_j-r_k)+(r_l+r_j-2r_m+1)/2}$}
    \put(0,0){$\sim\exp\left\{\frac{N}{2\pi i}\left(\li(\frac{w_l}{w_m})+\li(\frac{w_l}{w_k})-\li(\frac{w_jw_l}{w_kw_m})-\li(\frac{w_m}{w_j})-\li(\frac{w_k}{w_j})+\frac{\pi^2}{6}-\log\frac{w_j}{w_m}\log\frac{w_j}{w_k}\right)\right\}$,}
  \end{picture}}

  {\setlength{\unitlength}{0.4cm}
  \begin{picture}(18,9.5)\thicklines
    \put(2,7){\vector(1,-1){4}}
   \put(6,7){\line(-1,-1){1.8}}
   \put(3.8,4.8){\vector(-1,-1){1.8}}
    \put(3.5,3){$r_j$}
    \put(6,4.5){$r_k$}
    \put(3.5,6.5){$r_l$}
    \put(1,4.5){$r_m$}
    \put(0,7.3){$r_l-r_m$}
    \put(5,7.3){$r_k-r_l$}
    \put(0,2.2){$r_j-r_m$}
    \put(5,2.2){$r_k-r_j$}
    \put(4,5){\arc(-0.6,-0.6){270}}
    \put(9,5){: $\displaystyle\frac{(q^{-1})^{-1}_{r_l-r_m}(q^{-1})_{r_k-r_l}}{(q)_{r_k+r_m-r_j-r_l}(q^{-1})_{r_j-r_m}(q^{-1})^{-1}_{r_k-r_j}}(-1)^{r_l+r_j+1}$}
    \put(10,2.5){$\times q^{(r_l-r_k)(r_j-r_k)+(2r_k-r_l-r_j+1)/2}$}
    \put(0,0){$\sim\exp\left\{\frac{N}{2\pi i}\left(-\li(\frac{w_m}{w_l})+\li(\frac{w_l}{w_k})+\li(\frac{w_k w_m}{w_j w_l})-\li(\frac{w_m}{w_j})+\li(\frac{w_j}{w_k})-\frac{\pi^2}{6}+\log\frac{w_l}{w_k}\log\frac{w_j}{w_k}\right)\right\}$,}
  \end{picture}}

  {\setlength{\unitlength}{0.4cm}
  \begin{picture}(18,9.5)\thicklines
    \put(2,7){\vector(1,-1){4}}
   \put(6,7){\line(-1,-1){1.8}}
   \put(3.8,4.8){\vector(-1,-1){1.8}}
    \put(3.5,3){$r_j$}
    \put(6,4.5){$r_k$}
    \put(3.5,6.5){$r_l$}
    \put(1,4.5){$r_m$}
    \put(0,7.3){$r_l-r_m$}
    \put(5,7.3){$r_k-r_l$}
    \put(0,2.2){$r_j-r_m$}
    \put(5,2.2){$r_k-r_j$}
    \put(4,5){\arc(0.6,-0.6){270}}
    \put(9,5){: $\displaystyle\frac{(q^{-1})^{-1}_{r_l-r_m}(q)_{r_k-r_l}}{(q^{-1})_{r_k+r_m-r_j-r_l}(q)_{r_j-r_m}(q^{-1})^{-1}_{r_k-r_j}}(-1)^{r_l+r_j+1}$}
    \put(10,2.5){$\times q^{-(r_l-r_m)(r_l-r_k)+(r_l+r_j-2r_m+1)/2}$}
    \put(0,0){$\sim\exp\left\{\frac{N}{2\pi i}\left(-\li(\frac{w_m}{w_l})-\li(\frac{w_k}{w_l})-\li(\frac{w_j w_l}{w_k w_m})+\li(\frac{w_j}{w_m})+\li(\frac{w_j}{w_k})+\frac{\pi^2}{6}-\log\frac{w_l}{w_m}\log\frac{w_l}{w_k}\right)\right\}$,}
  \end{picture}}

  {\setlength{\unitlength}{0.4cm}
  \begin{picture}(18,9.5)\thicklines
    \put(2,7){\vector(1,-1){4}}
   \put(6,7){\line(-1,-1){1.8}}
   \put(3.8,4.8){\vector(-1,-1){1.8}}
    \put(3.5,3){$r_j$}
    \put(6,4.5){$r_k$}
    \put(3.5,6.5){$r_l$}
    \put(1,4.5){$r_m$}
    \put(0,7.3){$r_l-r_m$}
    \put(5,7.3){$r_k-r_l$}
    \put(0,2.2){$r_j-r_m$}
    \put(5,2.2){$r_k-r_j$}
    \put(4,5){\arc(0.6,0.6){270}}
    \put(9,5){: $\displaystyle\frac{(q)^{-1}_{r_l-r_m}(q)_{r_k-r_l}}{(q)_{r_k+r_m-r_j-r_l}(q)_{r_j-r_m}(q)^{-1}_{r_k-r_j}}(-1)^{r_l+r_j+1}$}
    \put(10,2.5){$\times q^{(r_l-r_m)(r_j-r_m)+(r_l+r_j-2r_m+1)/2}$}
    \put(0,0){$\sim\exp\left\{\frac{N}{2\pi i}\left(\li(\frac{w_l}{w_m})-\li(\frac{w_k}{w_l})+\li(\frac{w_k w_m}{w_j w_l})+\li(\frac{w_j}{w_m})-\li(\frac{w_k}{w_j})-\frac{\pi^2}{6}+\log\frac{w_l}{w_m}\log\frac{w_j}{w_m}\right)\right\}$.}
  \end{picture}}

\vspace{0.5cm}

For the trivalent vertices of $G$, we assign $0$ to the sides in $I$ or $J$,
then apply the same formal substitution to the R-matrix as follows
(here, we use the same form of the R-matrix disregarding whether certain horizontal edge is collapsed or not).

For the endpoint of $I$ :

  {\setlength{\unitlength}{0.4cm}
  \begin{picture}(18,7.5)\thicklines
    \put(6,5){\vector(-1,-1){4}}
    \dashline{0.5}(4.2,2.8)(6,1)
    \put(5.2,1.8){\vector(1,-1){0.8}}
    \put(2,5){\line(1,-1){1.8}}
    \put(3.5,1){$r_j$}
    \put(3.5,4.5){$r_l$}
    \put(1,2.5){$r_m$}
    \put(0,5.3){$r_l-r_m$}
    \put(5,5.3){$r_j-r_l$}
    \put(0,0.2){$r_j-r_m$}
    \put(6,0.2){0}
    \put(9,3.5){: $\displaystyle\frac{(q^{-1})_{r_j-r_m}}{(q^{-1})_{r_j-r_l}}(-1)^{r_l+r_j+1}q^{-(r_j-r_l+1)/2}$}
    \put(9,1){$\sim\exp\left\{\frac{N}{2\pi i}\left(\li(\frac{w_m}{w_j})-\li(\frac{w_l}{w_j})\right)\right\}$,}
  \end{picture}}

  {\setlength{\unitlength}{0.4cm}
  \begin{picture}(18,7)\thicklines
    \put(2,5){\vector(1,-1){4}}
    \put(6,5){\line(-1,-1){1.8}}
    \dashline{0.5}(2,1)(3.8,2.8)
    \put(2.8,1.8){\vector(-1,-1){0.8}}
    \put(3.5,1){$r_j$}
    \put(3.5,4.5){$r_l$}
    \put(6,2.5){$r_k$}
    \put(0,5.3){$r_l-r_j$}
    \put(5,5.3){$r_k-r_l$}
    \put(1.3,0){0}
    \put(5,0.2){$r_k-r_j$}
    \put(9,3.5){: $\displaystyle\frac{(q)_{r_k-r_j}}{(q)_{r_l-r_j}}(-1)^{r_l+r_j+1}q^{(r_l-r_j-1)/2}$}
    \put(9,1){$\sim\exp\left\{\frac{N}{2\pi i}\left(-\li(\frac{w_k}{w_j})+\li(\frac{w_l}{w_j})\right)\right\}$.}
  \end{picture}}

\vspace{0.5cm}

For the endpoint of $J$ :

  {\setlength{\unitlength}{0.4cm}
  \begin{picture}(18,7)\thicklines
   \put(4.2,2.8){\vector(1,-1){2}}
   \put(2,5){\line(1,-1){1.8}}
   \dashline{0.5}(6,5)(4,3)
   \put(4,3){\vector(-1,-1){2}}
    \put(3.5,1){$r_j$}
    \put(6,2.5){$r_k$}
    \put(1,2.5){$r_m$}
    \put(6.3,5.3){0}
    \put(0,0.2){$r_j-r_m$}
    \put(5,0.2){$r_k-r_j$}
    \put(0,5.3){$r_k-r_m$}
    \put(9,3.5){: $\displaystyle\frac{(q^{-1})_{r_k-r_m}}{(q^{-1})_{r_k-r_j}}(-1)^{r_k+r_j+1}q^{-(r_k-r_j+1)/2}$}
    \put(9,1){$\sim\exp\left\{\frac{N}{2\pi i}\left(\li(\frac{w_m}{w_k})-\li(\frac{w_j}{w_k})\right)\right\}$,}
  \end{picture}}

  {\setlength{\unitlength}{0.4cm}
  \begin{picture}(18,7)\thicklines
   \put(3.8,2.8){\vector(-1,-1){1.8}}
   \put(4.2,3.2){\line(1,1){1.8}}
   \dashline{0.5}(2,5)(4,3)
   \put(4,3){\vector(1,-1){2}}
    \put(3.5,1){$r_j$}
    \put(6,2.5){$r_k$}
    \put(3.5,4.5){$r_l$}
    \put(5,5.3){$r_k-r_l$}
    \put(0,0.2){$r_j-r_l$}
    \put(5.5,0.2){$r_k-r_j$}
    \put(1.3,5){0}
    \put(9,3.5){: $\displaystyle\frac{(q)_{r_k-r_l}}{(q)_{r_j-r_l}}(-1)^{r_l+r_j+1}q^{(r_j-r_l+1)/2}$}
    \put(9,1){$\sim\exp\left\{\frac{N}{2\pi i}\left(-\li(\frac{w_k}{w_l})+\li(\frac{w_j}{w_l})\right)\right\}$.}
  \end{picture}}

Note that the colored Jones polynomial is expressed by 
the products of various forms of the R-matrices of crossings or trivalent vertices of $G$ (with slight modification by the local maxima/minima)
and summed over all the possible indices $r_1,\ldots,r_{m+1}$ 
(see \cite{Murakami00a} for the calculation of the colored Jones polynomial; 
the description in \cite{Murakami00a} may look slightly different from ours, 
but removing the sides of the tangle diagram assigned with 0 in \cite{Murakami00a} gives the diagram $G$).
Now we define a potential function $\widetilde{W}(w_1,\ldots,w_{m+1})$ of the knot diagram by letting
the product of all formal substitutions of $G$ to be $\exp\left\{\frac{N}{2\pi i}\widetilde{W}(w_1,\ldots,w_{m+1})\right\}$.
One important property of $\widetilde{W}$ is that the variable $w_{m+1}$ assigned to the unbounded region
appears only in the numerator. Therefore, we can define another potential function
$W(w_1,\ldots,w_m):=\widetilde{W}(w_1,\ldots,w_m,0)$,\footnote{Note that $\li(0)=0$.}
which coincides with the potential function $W(w_1,\ldots,w_m)$ defined in Section \ref{ch32}.

For example, $\widetilde{W}$ and $W$ of Figure \ref{pic8} become
\begin{eqnarray*}
\lefteqn{\widetilde{W}(w_1,\ldots,w_5)=\left\{\li(\frac{1}{w_2})-\li(\frac{w_3}{w_2})\right\}
+\left\{\li(\frac{w_5}{w_3})-\li(\frac{1}{w_3})\right\}}\\
&&+\left\{-\li(\frac{w_5}{w_4})+\li(\frac{w_4}{w_2})+\li(\frac{w_5w_2}{w_4w_3})-\li(\frac{w_5}{w_3})+\li(\frac{w_3}{w_2})
-\frac{\pi^2}{6}+\log\frac{w_4}{w_2}\log\frac{w_3}{w_2}\right\}\\
&&+\left\{-\li(\frac{w_5}{w_1})+\li(\frac{w_1}{w_2})+\li(\frac{w_5w_2}{w_1w_4})-\li(\frac{w_5}{w_4})+\li(\frac{w_4}{w_2})
-\frac{\pi^2}{6}+\log\frac{w_1}{w_2}\log\frac{w_4}{w_2}\right\}\\
&&+\left\{-\li(w_5)+\li(\frac{1}{w_2})+\li(\frac{w_5w_2}{w_1})-\li(\frac{w_5}{w_1})+\li(\frac{w_1}{w_2})
-\frac{\pi^2}{6}+\log\frac{1}{w_2}\log\frac{w_1}{w_2}\right\},
\end{eqnarray*}
and
\begin{eqnarray*}
&&W(w_1,\ldots,w_4)=2\left\{\li(\frac{1}{w_2})+\li(\frac{w_4}{w_2})+\li(\frac{w_1}{w_2})\right\}-\li(\frac{1}{w_3})
-\frac{\pi^2}{2}\\
&&~~~+\log\frac{w_4}{w_2}\log\frac{w_3}{w_2}+\log\frac{w_1}{w_2}\log\frac{w_4}{w_2}
+\log\frac{1}{w_2}\log\frac{w_1}{w_2}.
\end{eqnarray*}
This potential function $W(w_1,\ldots,w_4)$ coincides with the one defined previously in (\ref{W}).

Note that using $W$ instead of $\widetilde{W}$ does not violate the formulation of the optimistic limit because,
for a solution $(w_1^{(0)},\ldots,w_m^{(0)})$ of
$\mathcal{H}_2=\left\{\exp\left(w_l\frac{\partial W}{\partial w_l}\right)=1\,\vert\, l=1,\ldots,m\right\}$,
$(w_1^{(0)},\ldots,w_m^{(0)},0)$ becomes a solution of
$\widetilde{\mathcal{H}}_2:=\left\{\exp\left(w_l\frac{\partial \widetilde{W}}{\partial w_l}\right)=1\,\vert\, l=1,\ldots,m+1\right\}$.
We are considering only the solutions of $\widetilde{\mathcal{H}}_2$ with the condition $w_{m+1}=0$ because
this condition corresponds to the collapsing process of tetrahedra of Thurston triangulation in Section \ref{ch22}
and the solutions correspond to the triangulation. However, other solutions with the condition $w_{m+1}\neq 0$
also have good geometric meanings and this will be discussed in later papers.

\subsection{Inessential solutions induced by essential solutions}\label{app2}

Let $\bold{z}$ and $\bold{w}$ be the solutions in Lemma \ref{lem12}.
In this Appendix, we determine the condition when an essential solution induces an inessential solution.
Note that solutions $\bold{z}$ and $\bold{w}$ uniquely determine 
shape parameters of ideal tetrahedra in Yokota triangulation and in Thurston triangulation, respectively,
and that, by definition, essential solution determines the shape parameters with none of them belonging to $\{0,1,\infty\}$.
Therefore, we focus on the shape parameters of each triangulation.
We call the set of shape parameters of ideal tetrahedra {\it essential} when no elements of it belongs to $\{0,1,\infty\}$.

Note that the shape parameters of two triangulations are determined by the local picture at each crossings
and that, from Observation \ref{obs}, what we have to consider are 3-2 moves and 4-5 moves at the crossings.
Consider the two cases of Figure \ref{pic17} and Figure \ref{cola}, which correspond to
4-5 move and 3-2 move, respectively, 
and for which we have the determining relations of shape parameters in (\ref{eq20}) and in (\ref{eq30}), respectively.

\begin{lem}\label{apple}
\begin{enumerate}
\item In Figure \ref{cola}, if $\{t_1,t_2,t_4\}$ is essential, then $\{u_1,u_2\}$ is essential. 
Conversely, if $\{u_1,u_2\}$ is essential, then $\{t_1,t_2,t_4\}$ is essential if and only if 
\begin{equation}\label{a0}u_1+u_2=1.\end{equation}

\item In Figure \ref{pic17}, if $\{t_1,t_2,t_3, t_4\}$ is essential, then $\{u_1,u_2, u_3, u_4, u_5\}$ is essential if and only if
\begin{equation}\label{a1}
t_1-\frac{1}{t_2}\neq 0,~\frac{1}{t_2}-t_3\neq 0,~t_3-\frac{1}{t_4}\neq 0,~\frac{1}{t_4}-t_1\neq 0, ~
~t_1-\frac{1}{t_2}+t_3-\frac{1}{t_4}\neq 0.
\end{equation}
(Note that $u_5=\frac{1}{u_1 u_3}=\frac{1}{u_2 u_4}$.) 
Conversely, if $\{u_1,u_2, u_3, u_4, u_5\}$ is essential, then $\{t_1,t_2,t_3, t_4\}$ is essential if and only if
\begin{eqnarray}\label{a2}
      \left\{
      \begin{array}{l}
     \frac{1}{u_1}+\frac{1}{u_2}-\frac{1}{u_1u_2}\neq u_5,\\
     u_2+u_3-u_2u_3\neq\frac{1}{u_5},\\
     \frac{1}{u_3}+\frac{1}{u_4}-\frac{1}{u_3u_4}\neq u_5,\\
     u_4+u_1-u_4u_1\neq\frac{1}{u_5}.
      \end{array}\right.
\end{eqnarray}
\end{enumerate}
\end{lem}

\begin{proof}
From the relations (\ref{eq30}) and (\ref{eq20}), if one of the sets $\{t_1,t_2,t_3, t_4\}$ and $\{u_1,u_2, u_3, u_4, u_5\}$ is essential,
then the shape parameters of the other set are expressed by products of nonzero and non-infinity numbers.
This implies any shape parameter in the other set cannot be zero nor infinity.
Therefore, what we have to check is the case when $t_k=1$ or $u_l=1$ for some $k,l$.

Consider Figure \ref{cola}. Assume $\{t_1,t_2,t_4\}$ is essential and $u_1=1$. Then from $u_1=t_1't_4''=1$, we obtain $t_1t_4=1$.
Using $t_1 t_2 t_4=1$, this induces $t_2=1$, which contradicts the essentiality of $\{t_1,t_2,t_4\}$. The case when $u_2=1$ is the same.

Conversely, assume $\{u_1,u_2\}$ is essential. By direct calculation from (\ref{eq30}), we obtain
$$t_1=u_1''u_2''=1\iff  u_1+u_2=1\iff t_2=u_1u_2'=1\iff t_4=u_1'u_2=1.$$

Now consider Figure \ref{pic17}. Assume $\{t_1,t_2,t_3, t_4\}$ is essential. 
Then direct calculation from (\ref{eq20}) shows (\ref{a1}) is equivalent to
$$u_2\neq 1,~u_3\neq 1,~u_4\neq 1,~u_1\neq 1,~u_5\neq 1.$$
For example, using $t_1t_3=\frac{1}{t_2t_4}$, we have
$$u_5=(t_1't_2''t_3't_4'')^{-1}=1\iff (1-\frac{1}{t_2})(1-\frac{1}{t_4})=(1-t_1)(1-t_3)\iff t_1-\frac{1}{t_2}+t_3-\frac{1}{t_4}=0.$$

Conversely, assume $\{u_1,u_2, u_3, u_4, u_5\}$ is essential.
Then direct calculation from (\ref{eq20}) shows (\ref{a2}) is equivalent to
$$t_1\neq1, ~t_2\neq1,~t_3\neq1,~t_4\neq1.$$

\end{proof}

From the above, if the essential solution $\bold{z}$ in Lemma \ref{lem12} 
determines the shape parameters of Yokota triangulation that satisfy the conditions (\ref{a1}) in Lemma \ref{apple},
then the corresponding solution $\bold{w}$ is also essential. Conversely, the essential solution $\bold{w}$ in Lemma \ref{lem12}
determines the shape parameters of Thurston triangulation that satisfy the conditions (\ref{a0}) and (\ref{a2}) in Lemma \ref{apple},
then the corresponding solution $\bold{z}$ is also essential. We expect these conditions hold for almost all cases. For example,
the essential solutions $\bold{z}$ and $\bold{w}$ of twist knots in \cite{Cho09b} and \cite{Cho10a},
and the geometric solutions $\bold{w}$ of the two-bridge knots in \cite{Ohnuki05} satisfy these conditions.
Furthermore, if every octahedron in the Yokota triangulation have one collapsed horizontal edge,
then the essential solution $\bold{z}$ always satisfies the condition.
Therefore, essential solutions $\bold{z}$ coming from the standard diagrams of 2-bridge knots in \cite{Ohnuki05} always induce
the essential solutions $\bold{w}$.

\vspace{5mm}
\begin{ack}
The authors show gratitude to Yoshiyuki Yokota for sending us his preprint in advance before publication.
  This work was started when the first author was visiting Waseda university with Grant-in-Aid for JSPS Fellows 21.09221 and 
  he is POSCO TJ Park fellow now.
  The first author was supported by Korea Research Foundation Grant funded by the Korean Government (KRF-2008-341-C00004) and
   the second author was supported by Grant-in-Aid for
Scientific Research no. 22540236.
\end{ack}

\bibliography{VolConj}
\bibliographystyle{abbrv}

{\sc
\begin{flushleft}
  School of Mathematics, Korea Institute for Advanced Study, 85 Hoegiro, Dongdaemun-gu, Seoul 130-722, Republic of Korea\\[5pt]
  Department of Mathematics, Faculty of Science and Engineering, Waseda University, 3-4-1 Okubo, Shinjuku-ku, Tokyo 169-8555, Japan\\
E-mail: dol0425@gmail.com\\
\mbox{\phantom{E-mail: }murakami@waseda.jp}
\end{flushleft}}
\end{document}